\nonstopmode
\documentclass[10pt,reqno]{amsart}
\usepackage{graphicx}
\usepackage{latexsym}
\usepackage{fancyhdr}
\usepackage{amsmath, amssymb, amsthm}
\usepackage[all]{xy}
\usepackage{pdflscape}
\usepackage{longtable}
\usepackage{rotating}
\usepackage{verbatim}
\usepackage{hyperref}
\hypersetup{
    linkcolor = blue,
    citecolor = blue,
    urlcolor  = blue,
    colorlinks = true,
}

\usepackage{cleveref}
\usepackage{subfigure}
\usepackage{mathrsfs}
\usepackage{mdwlist}
\usepackage{dsfont}
\usepackage{mathtools}
\usepackage{float}
\usepackage{color}
\usepackage{stmaryrd}

\usepackage{tkz-base}
\usepackage{pgfplots}
\pgfplotsset{compat=newest}

\usepackage{tkz-euclide}
\usepackage{etoolbox}

\definecolor{teal}{rgb}{0.0, 0.5, 0.5}

\newcounter{mnotecount}[section]

\allowdisplaybreaks

\DeclareFontFamily{U}{mathb}{\hyphenchar\font45}
\DeclareFontShape{U}{mathb}{m}{n}{
      <5> <6> <7> <8> <9> <10> gen * mathb
      <10.95> mathb10 <12> <14.4> <17.28> <20.74> <24.88> mathb12
      }{}
\DeclareSymbolFont{mathb}{U}{mathb}{m}{n}
\DeclareFontSubstitution{U}{mathb}{m}{n}

\theoremstyle{plain}
\newtheorem*{theorem*}{Theorem}
\newtheorem{theorem}{Theorem}[section]
\newtheorem*{lemma*}{Lemma}
\newtheorem{lemma}[theorem]{Lemma}
\newtheorem*{assumption*}{Assumption}

\newtheorem*{proposition*}{Proposition}
\newtheorem{proposition}[theorem]{Proposition}
\newtheorem*{corollary*}{Corollary}
\newtheorem{corollary}[theorem]{Corollary}
\newtheorem*{claim*}{Claim}

\newtheorem*{conjecture*}{Conjecture}

\newtheorem*{question*}{Question}
\newtheorem{open}[theorem]{Open Problem}
\newtheorem*{result*}{Result}

\theoremstyle{definition}
\newtheorem*{definition*}{Definition}
\newtheorem{definition}[theorem]{Definition}
\newtheorem*{example*}{Example}
\newtheorem{example}[theorem]{Example}

\newtheorem*{algorithm*}{Algorithm}
\newtheorem*{remark*}{Remark}
\newtheorem*{remarks*}{Remarks}
\newtheorem{remark}[theorem]{Remark}

\newtheorem*{convention*}{Convention}

\def\al{\alpha}
\def\be{\beta}
\def\ga{\gamma}
\def\de{\delta}
\def\ep{\epsilon}

\def\et{\eta}
\def\th{\theta}

\def\la{\lambda}

\def\si{\sigma}

\def\ta{\tau}

\def\ph{\phi}
\def\vh{\varphi}
\def\ch{\chi}
\def\ps{\psi}
\def\om{\omega}

\def\Ga{\Gamma}
\def\De{\Delta}
\def\Th{\Theta}

\def\Si{\Sigma}

\def\Om{\Omega}

\def\C{\mathbb{C}}

\def\I{\mathbb{I}}
\def\K{\mathbb{K}}
\def\N{\mathbb{N}}

\def\R{\mathbb{R}}
\def\Z{\mathbb{Z}}

\def\cA{\mathcal{A}}

\def\cF{\mathcal{F}}

\def\cH{\mathcal{H}}

\def\cK{\mathcal{K}}
\def\cL{\mathcal{L}}

\def\cN{\mathcal{N}}

\def\sC{\mathscr{C}}
\def\sD{\mathscr{D}}

\def\sI{\mathscr{I}}

\def\sP{\mathscr{P}}

\def\sS{\mathscr{S}}

\def\p{\partial}

\def\ansqrt{\theta}

\renewcommand{\Im}{\mathrm{Im}}
\def\<{\langle}
\def\>{\rangle}
\renewcommand{\o}{\circ}
\def\cq{{/\!\!/}}

\newcommand{\at}{^o}

\def\ol{\overline}
\def\ul{\underline}

\def\Lip{\on{Lip}}
\def\Hyp{\on{Hyp}}

\def\cq{{/\!\!/}}
\def\acts{\circlearrowleft}
\def\rep{(G\acts V,d,\si)}
\def\Lip{\on{Lip}}
\def\Hoeld{\on{H\ddot{o}ld}}
\def\kk{{m}}

\DeclareMathOperator{\discr}{\Delta}
\DeclareMathOperator{\charp}{\chi}

\let\on=\operatorname

\newcommand{\sr}[1]%
{\ifmmode{}^\dagger\else${}^\dagger$\fi\ifvmode
\vbox to 0pt{\vss
 \hbox to 0pt{\hskip\hsize\hskip1em
 \vbox{\hsize3cm\raggedright\pretolerance10000
 \noindent #1\hfill}\hss}\vss}\else
 \vadjust{\vbox to0pt{\vss%
 \hbox to 0pt{\hskip\hsize\hskip1em%
 \vbox{\hsize3cm\raggedright\pretolerance10000%
 \noindent #1\hfill}\hss}\vss}}\fi%
}

\providecommand{\mapsfrom}{\kern.2em%
\setbox0=\hbox{$\leftarrow$\kern-.10em\rule[0.26mm]{0.1mm}{1.3mm}}\box0%
\kern.3em}

\title[Perturbation theory of polynomials and linear operators]
{Perturbation theory of polynomials\\ and linear operators}

\author[Adam Parusi\'nski and  Armin Rainer]
{Adam Parusi\'nski and Armin Rainer}

\address {Adam Parusi\'nski: Universit\'e C\^ote d'Azur,  CNRS,  LJAD, UMR 7351, 06108 Nice, France}
\email{adam.parusinski@univ-cotedazur.fr}

\address{Armin Rainer: Fakult\"at f\"ur Mathematik, Universit\"at Wien,
Oskar-Morgenstern-Platz~1, A-1090 Wien, Austria}
\email{armin.rainer@univie.ac.at}

\begin{document}

\begin{abstract} 
This survey revolves around the question how the roots of a monic polynomial (resp.\ the spectral decomposition of a linear operator), 
    whose coefficients depend in a smooth way 
    on parameters, depend on those parameters. 
    The parameter dependence of the polynomials (resp.\ operators) ranges from real analytic over $C^\infty$ to differentiable of finite order
    with often drastically different regularity results for the roots (resp.\ eigenvalues and eigenvectors).
    Another interesting point is the difference between the perturbation theory of hyperbolic polynomials (where, by definition, all roots are real) 
    and that of general complex polynomials.
    The subject, which started with Rellich's work in the 1930s, enjoyed sustained interest through time
    that intensified in the last two decades, bringing some definitive optimal results. 
    Throughout we try to explain the main proof ideas; Rellich's theorem and Bronshtein's theorem on 
    hyperbolic polynomials are presented with full proofs.
    The survey is written for readers interested in singularity theory but also for 
    those who intend to apply the results in other fields.
\end{abstract}

\thanks{This research was funded in whole or in part by the Austrian Science Fund (FWF) DOI 10.55776/P32905.
For open access purposes, the author has applied a CC BY public copyright license to any author-accepted manuscript version arising from this submission.}
\date{\today}

\maketitle

\tableofcontents
\section{Introduction}

The question how the spectral decomposition of a linear operator in Hilbert space depends on parameters of the operator 
is a natural and important one. Its investigation -- stimulated without doubt by the development of quantum mechanics --  
began with Rellich's paper \cite{Rellich37} published in 1937 
(the first in a series which culminated in Rellich's book \cite{Rellich69}).
In finite dimensions, the problem reduces to the questions about the regularity of the eigenvalues and eigenvectors of 
families of matrices and, in consequence, one is led to ask how regular the roots of a monic polynomial depending on parameters 
can be chosen as functions of these parameters.

In this survey, we will concentrate on the finite dimensional setting of matrices and polynomials; occasionally we will provide 
references for generalizations to operators in infinite dimensional Hilbert space.

\subsubsection*
{Real analytic perturbations of normal matrices and hyperbolic polynomials}  
In \cite{Rellich37}, Rellich was concerned with analytic families of linear operators, in particular, he proved that 
real analytic curves of Hermitian matrices admit real analytic eigenvalues and a real analytic frame of orthonormal eigenvectors.
To this end, he showed that a real analytic curve of monic hyperbolic polynomials admits a real analytic choice of its roots.
A monic polynomial $P_a(Z) = Z^d + \sum_{j=1}^d a_j Z^{d-j}$ with real coefficients is called \emph{hyperbolic}\index{hyperbolic polynomial}\index{polynomial!hyperbolic} 
if all $d$ roots 
of $P_a$ (with multiplicities) are real.

Rellich also realized that the analytic dependence of the spectral decomposition breaks down as soon as the Hermitian matrix 
depends on two or more real parameters. The analytic multiparameter case was taken up and continued 
by Kurdyka and P\u aunescu \cite{KurdykaPaunescu2008}: there exists an analytic modification of the parameter space such that, after this 
modification, the eigenvalues and the eigenvectors (or the roots in the case of hyperbolic polynomials) 
admit real analytic parameterizations, locally.  
Actually, these results hold in the more general setting of normal matrices, as observed by Rainer \cite{RainerN},
where the eigenvalues are generally complex valued. We discuss the analytic perturbation theory in Section \ref{sec:analytic}.

\subsubsection*
{Singularity theory methods} 
Many results of perturbation theory can be proven using singularity theory.  
Puiseux's theorem is an important ingredient of Rellich's proof and the multiparameter case is proven using its multivariable counterpart, 
the Abhyankar--Jung theorem.  In Section \ref{sec:analytic}  we present a different proof of Rellich's theorem on perturbation of hyperbolic polynomials, 
based on a nowadays standard singularity theory tool, the splitting, which we also recall briefly.  
Similar methods are used in the multiparameter case and in  the study of the perturbation of normal matrices, 
see Parusi\'nski and Rond \cite{ParusinskiRondMatr}, 
where the splitting is replaced by a version of Hensel's lemma for normal matrices, or in the problem of lifting over invariants of group 
representations, 
where it is replaced by the slice theorem (see Theorem \ref{thm:slice}).  We describe how 
this method is used in several places of the survey.  

In the multiparameter case, the resolution of singularities is used to make ideals principal generated by normal crossings.  
In most cases, this is the ideal generated by the discriminant of $P_a$, or of the square-free reduction $(P_a)_{\on{red}}$, so that one can use then the Abhyankar--Jung theorem.  
Interestingly, the resolution of singularities is used as well to show that the continuous roots of $P_a$, with $a=a(t)$ 
depending smoothly on $t\in \R$, 
are locally absolutely continuous,
see Section \ref{sec:formulas} and Parusi\'nski and Rainer \cite{ParusinskiRainerAC}.

In the case of $C^\infty$ coefficients (or coefficients of finite differentiability), 
besides splitting, another important ingredient is Glaeser's inequality and its generalizations.  
We recall some of its manifestations in 
Section \ref{sec:Bronshtein}.  The higher order Glaeser inequalities appear in 
Section \ref{sec:regularity}.    

\subsubsection*
{$C^\infty$ perturbations of normal matrices and hyperbolic polynomials}  
Another field of application which greatly stimulated the research on the regularity problem for roots of polynomials is the theory of PDEs -- 
most notably the Cauchy problem for hyperbolic PDOs. 
With this goal in mind, Bronshtein \cite{Bronshtein79} proved that any continuous root of a monic hyperbolic polynomial $P_a$ of degree 
$d$ with coefficients $C^{d-1,1}$ functions $a_j : U \to \R$, where $U$ is an open subset of $\R^n$, 
is locally Lipschitz. (Glaeser's inequality gives Bronshtein's result in the simplest nontrivial case.) 
Note that ordering the roots $P_a$ increasingly yields a continuous system of the roots.
This result is sharp. If $n=1$ and the coefficients are $C^d$ (resp.\ $C^{2d}$), then there is a $C^1$ (resp.\ twice differentiable) 
choice of the roots, but there are examples of nonnegative $C^\infty$ functions $f$ such that $Z^2=f$ has no $C^{1,\al}$ solution 
for any $\al>0$. We give a full proof of Bronshtein's theorem with uniform bounds for the Lipschitz constant in Section \ref{sec:Bronshteinproof} 
(it is based on Parusi\'nski and Rainer \cite{ParusinskiRainerHyp}).   

The step from analytic to $C^\infty$ coefficients entails a big drop of regularity for the roots, because it allows for oscillation.
The same is true for the eigenvalues of Hermitian and, more generally, normal matrices; for the eigenvectors it is even worse: 
there are $C^\infty$ curves of symmetric $2 \times 2$ matrices that do not admit a continuous system of eigenvectors.
On the other hand, the conditions that guarantee locally Lipschitz or $C^1$ eigenvalues of Hermitian (even normal) matrices 
are much weaker than for hyperbolic polynomials: for instance, a $C^{0,1}$ (resp.\ $C^1$) curve of Hermitian (or, more generally, normal) matrices admits a 
system of $C^{0,1}$ (resp.\ $C^1$) eigenvalues. We explore several explanations for this phenomenon: one of them is that 
the determinant is a G{\aa}rding hyperbolic polynomial with respect to the identity matrix on the real vector space of 
$d \times d$ Hermitian matrices (see Section \ref{sec:Garding} and Section \ref{sec:Hermitian}).
This also gives that $C^{1,1}$ curves of Hermitian matrices admit eigenvalues that are locally of Sobolev class $W^{2,1}$ 
(which is optimal in view of the Sobolev inequality).

There are different approaches (discussed in Section \ref{sec:sufficienthyp}) to ascertain better $C^p$ regularity of the roots and the eigenvalues  
by imposing stronger assumptions.
For instance, given any integer $p\ge 2$, forcing $f \ge 0$ defined on $\R$ and sufficiently many of its derivatives to vanish on all local minima of $f$ 
guarantees that the equation $Z^2 = f$ has a $C^p$ solution. 
Or, one can define a regularity class $\cF^\be$, $\be >0$, of nonnegative $C^\be$ functions on $\R$, incorporating flatness of $f$ near its zeros 
in a certain sense, 
such that $f^\al \in \cF^{\al\be}$ for all $f \in \cF^\be$ and $\al \in (0,1]$.
In the case of a curve of monic hyperbolic polynomials or normal matrices,
sufficient conditions for the existence of $C^p$ roots or eigenvalues can be given in terms of the differentiability of the coefficients and the 
finite order of contact of the roots or eigenvalues. If the coefficients are $C^\infty$ and definable in any (not necessarily polynomially bounded) 
o-minimal expansion of the real field, 
then the roots and eigenvalues can be chosen $C^\infty$ and definable; this reaffirms that oscillation is to blame for the loss of regularity.

\subsubsection*
{Perturbations of polynomials.  General (nonhyperbolic) case}  
Other problems from PDEs (cf.\ \cite{Spagnolo00}) and geometric analysis demand to abandon the hyperbolicity assumption:
how regular can the roots of a general (nonhyperbolic) monic polynomial $P_a$ with smooth coefficients $a_j : U \to \C$ be?
This problem was solved only recently by Ghisi and Gobbino \cite{GhisiGobbino13} in the radical case and 
by Parusi\'nski and Rainer \cite{ParusinskiRainerAC,ParusinskiRainerOpt,ParusinskiRainerBV} in the general case. 
We present in Section \ref{sec:regularity} two approaches to this problem: one is based on formulas 
for the roots of the universal monic polynomial of degree $d$, obtained by resolution of singularities, the other one is elementary and yields the optimal Sobolev regularity result.   
It states that the roots of a curve of monic complex polynomials $P_a$ of degree $d$ 
with coefficients of class $C^{d-1,1}(\ol I)$, where $I\subseteq \R$ is a bounded open interval,  admit an absolutely continuous parameterization whose derivative is in $L^p(I)$ for each $1 \le p <d/(d-1)$.
Actually, each continuous root has this property, in particular, it belongs to the Sobolev space $W^{1,p}(I)$ for each $1 \le p <d/(d-1)$.
Moreover, there are uniform bounds for the Sobolev norm of the roots in terms of the $C^{d-1,1}$ norm of the coefficients.  This result is optimal.  

A multiparameter version for continuous roots follows by standard arguments, provided there exists a continuous root (perhaps on some subset).
But in general there are obstructions for continuous roots, due to monodromy.
It turns out that the possibly discontinuous roots still can be represented by functions of bounded variation, cf. Section \ref{sec:bvariation}.
The proof uses the formulas for the roots of the universal monic polynomial alluded to above.

\subsubsection*
{Lifting from the orbit space} 
The problem of finding roots with optimal regularity of monic (hyperbolic and nonhyperbolic) polynomials
has a representation-theoretic interpretation: it is equivalent to a lifting problem over the orbit map $\si$ of the 
standard representation of the symmetric group $\on{S}_d$ on $\R^d$ or $\C^d$, respectively, where the components of $\si$ 
are the elementary symmetric functions. This point of view can be generalized considerably: we will discuss in Section \ref{sec:lifting} the lifting problem 
for real orthogonal representations of compact Lie groups and for complex rational representations of 
linearly reductive groups (the representation spaces will always be finite dimensional).
In this general setting, the lifting occurs over a map $\si$ consisting of basic invariants, that is a finite collection of generators of the algebra of invariant 
polynomials.
It gives another heuristic explanation why the eigenvalues of symmetric matrices are better behaved than the roots of monic hyperbolic
polynomials: only a ``partial lifting'' is necessary (see Section \ref{sec:partiallifting}).

\subsubsection*
{Applications} 
In the last section, Section \ref{sec:applications}, we give two applications of the results for general (nonhyperbolic) polynomials.
The first concerns the zero set of $C^\infty$ functions $f$.
If $f$ vanishes to some finite order $\ga$ at a point $x_0 \in \R^n$, 
then locally near $x_0$ the zero set of $f$ is given by the real roots of a monic polynomial $P_a$ of degree $\ga$,
by Malgrange's preparation theorem. 
We present several consequences of the optimal Sobolev regularity of the roots of $P_a$, 
most of them observed by Beck, Becker-Kahn, and Hanin \cite{Beck:2018hg}. 
Often one knows in advance the vanishing order of a function so that this setup applies,
e.g., for solutions of second order elliptic equations, Laplace eigenfunctions or finite linear combinations of such.

The second application is taken from the recent paper \cite{ACL:preprint} by Antonini, Cavalletti, and Lerario.
This paper brings forward a reinterpretation and an extension of the regularity problem for the roots of polynomials 
to the study of the Wasserstein distance on the space of $d$-degree hypersurfaces in $\C\mathbb P^n$.
This space is embedded in the space of measures on the projective space 
and thus the optimal transport problem between hypersurfaces is studied.
An inner Wasserstein distance is obtained which turns the projective space of homogeneous polynomials in 
a complete geodesic space.

\medskip
In the appendix, we collect definitions and basics on function spaces used in the survey (Section \ref{Appendix:A}) and 
we present some topological and geometric properties of the space of monic hyperbolic polynomials of a fixed degree (Section \ref{sec:spaceHd}).
There are several open problems that we included all over the article.

\section*{Notation}

We use $\N := \{0, 1,2,\ldots\}$ and $\N_{\ge m} := \{n \in \N : n \ge m\}$ for $m \in \N$. 
Similarly, $\R_{>a} := \{x \in \R : x>a\}$ etc.; for instance $\R_{\ge 0} = [0,\infty)$ and $\R_{>0} = (0,\infty)$.

For a subset $S$ of a topological space, $S^\o$, $\ol S$, and $\p S$ denote the interior, the closure, and the boundary of $S$, respectively.
We consider $\R^n$ with its Euclidean structure $|x| = \sqrt{x_1^2+\cdots+x_n^2}$, if not stated otherwise.
The open ball in $\R^n$ with center $x$ and radius $r>0$ is $B(x,r) := \{y \in \R^n : |x-y| < r\}$; 
if $n=1$, we use $I(x,r):=B(x,r)$.
The unit sphere in $\R^n$ is denoted by $\mathbb S^{n-1} := \{x \in \R^n : |x|=1\}$.
If $U \subseteq \R^n$ is an open subset of $\R^n$, then we write $V \Subset U$ to denote a relatively compact open subset $V$ of $U$.

The Lebesgue measure in $\R^n$ is denoted by $\cL^n$; we also use $|E| := \cL^n(E)$ for a Lebesgue measurable set $E \subseteq \R^n$.\index{Ln@$\cL^n$}
The $k$-dimensional Hausdorff measure is $\cH^{k}$\index{Hk@$\cH^k$} and $\cH^k \llcorner E$ denotes the restriction of $\cH^k$ to a subset $E \subseteq \R^n$, 
i.e., $(\cH^k \llcorner E)(F) := \cH^k(E \cap F)$.

For a function $f$, we write $Z_f$ for its zero set and $\Ga_f$ for its graph.
If $f$ and $g$ are real valued functions, $f\lesssim g$ means that $f \le C\, g$ for some constant $C>0$.
Generic constants are denoted by $C$ or $C(n,d,\ldots)$ to indicate that it depends on $n,d,\ldots$; 
their value may differ from line to line.

We use standard multiindex notation.
We write $df$ for the total derivative and $d_v f(x) := \p_t|_{t=0} f(x + t v)$ for the directional derivative.
In particular, $\p_j f = \frac{\p}{\p x_j} f = d_{e_j} f$, where $e_j$ is the $j$-th standard unit vector in $\R^n$.
We write $\nabla f = (\p_1 f, \ldots,\p_n f)$ for the gradient.
Higher order partial derivatives are denoted by 
$f^{(\al)} = \p^\al f = \p_1^{\al_1} \cdots \p_n^{\al_n} f$, for $\al \in \N^n$.

A monic polynomial of degree $d$ is $P_a(Z) = Z^d + \sum_{j=1}^d a_j Z^{d-j}$, 
where the subscript indicates the vector of coefficients $a = (a_1,\ldots,a_d)$. 
The coefficients may be real or complex numbers, formal power series, or functions; in the latter cases 
we often write $P_a(X)(Z)$, $P_a(t)(Z)$, or $P_a(x)(Z)$ to indicate the variables on which the coefficients depend.
If the first coefficient $a_1$ is zero, then the polynomial is said to be in Tschirnhausen form 
and in that case we consistently denote its coefficients by $\tilde a_j$, i.e., $P_{\tilde a}(Z) = Z^d + \sum_{j=2}^d \tilde a_j Z^{d-j}$.  

By a parameterization\index{parameterization of the roots} or a system\index{system of the roots} 
of the roots of a family of polynomials $P_a(x)$, where $x$ ranges over a subset $U$ of $\R^n$, we mean 
a collection of $d$ functions $\la_1,\ldots,\la_d$ such that 
$P_a(x)(Z) = \prod_{i=1}^d (Z-\la_i(x))$ for all $x \in U$.

The space of monic hyperbolic polynomials of degree $d$ is denoted by $\Hyp(d)$,\index{Hypd@$\Hyp(d)$}
its subspace of polynomials in Tschirnhausen form by $\Hyp_T(d)$, \index{HypTd@$\Hyp_T(d)$}
and $\Hyp_T^0(d) := \{P_{\tilde a} \in \Hyp_T(d) : \tilde a_2 = -1\}$.\index{HypT0d@$\Hyp_T^0(d)$}

For a polynomial $P(Z)$, we denote by $\discr_P$ its discriminant.\index{DeltaP@$\discr_P$} 
For an analytic function or a formal power series $f$, we denote by $f_{\on{red}}$\index{fred@$f_{\on{red}}$} its  
square-free reduction.  For a square matrix $A$, we denote by $\charp_A$\index{chiA@$\charp_A$} its characteristic polynomial
and by 
$\discr_A$  the discriminant of $(\chi_A)_ {\on{red}}$.\index{DeltaA@$\discr_A$}

We write $\on{Mat}_{m,n}(R)$\index{Matmn@$\on{Mat}_{m,n}(R)$} for the set of $m \times n$ matrices over the ring $R$ and 
set $\on{Mat}_n(R):=\on{Mat}_{n,n}(R)$; \index{Matn@$\on{Mat}_n(R)$}
for instance, 
$\on{Mat}_{m,n}(\K[[X]])$ is the set of $m \times n$ matrices with entries formal power series in $X$ with coefficients in $\K$.
For $R = \K \in \{\R,\C\}$ we also use $\K^{m \times n}:=\on{Mat}_{m,n}(\K)$. 
For $A \in \on{Mat}_{m,n}(\C)$ let $A^* = \ol A^t$ be its conjugate transpose.
The identity matrix is denoted by $\I$;\index{I@$\I$} its size will always be clear from the context.
We write $\on{U}(n) = \on{U}_n(\C)$\index{UnC@$\on{U}_n(\C)$} for the group of unitary matrices and $\on{O}(n) = \on{O}_n(\R)$\index{OnR@$\on{O}_n(\R)$} for the group of orthogonal matrices.
More abstractly, $\on{O}(V)$ is the group of orthogonal endomorphisms of a finite dimensional Euclidean vector space $V$.
$\on{GL}_n(\C)\index{GLnC@$\on{GL}_n(\C)$}$ or $\on{GL}_n(\R)$ denotes the general linear group. 
The set of diagonal $n \times n$ matrices is denoted by $\on{Diag}(n)$,\index{Diag@$\on{Diag}(d)$} 
the set of real symmetric $n \times n$ matrices by $\on{Sym}(n)$\index{Sym@$\on{Sym}(d)$} and the set of complex Hermitian  $n \times n$ matrices by $\on{Herm}(n)$.\index{Herm@$\on{Herm}(d)$}

The symmetric group of permutations of $n$ elements is $\on{S}_n$.\index{Sn@$\on{S}_n$} 
It acts on $\K^n$, where $\K\in \{\R,\C\}$, by permuting the coordinates and thus induces an equivalence relation.
The equivalence class of $x = (x_1,\ldots,x_n) \in \K^n$ can be identified with the unordered $n$-tuple $[x] = [x_1,\ldots,x_n]$
(which may be represented by $\sum_{j=1}^n \de_{x_j}$, where $\de_{x_j}$ is the Dirac measure).\index{unordered tuple}

\section{Perturbation theory with real analytic coefficients} \label{sec:analytic}

\subsection{Rellich's theorems}

In 1937 F. Rellich proved the following result.  

\begin{theorem}[Rellich {\cite[Satz I]{Rellich37}}, {\cite[Theorem 1]{Rellich69}}]\label{thm:Rellich1}\index{Rellich's theorem}\index{theorem!Rellich's}
Let $A(t)$ be a family of $ d\times d$ Hermitian matrices whose coefficients are real analytic in $t\in I$, 
where $I$ is an open  neighborhood of $0$ in  $\R$. Then, maybe in a smaller neighborhood of $0$, the eigenvalues and the eigenvectors of $A(t)$ can be chosen real analytic in $t$. 

 More precisely, there exist real valued functions  
$\lambda_i(t)$ and complex valued $d$-vectors $v_i(t)$, $i= 1, \ldots , d$, 
depending analytically on $t$, such that for $t$ in an open neighborhood $I'$ of $0$ in $\R$: 
\begin{enumerate}
    \item
$v_1(t), \ldots, v_d(t)$ is an orthonormal frame in $\C^d$;
\item
$A(t) v_i(t) = \lambda_i(t) v_i(t)$. 
\end{enumerate}
\end{theorem}

Note that the statement of this theorem is local.  
It was originally stated and proven in terms of convergent power series but clearly, by unique analytic continuation, 
it implies a global statement on an arbitrary open interval $I$ and with $I'=I$.  For proving the  above theorem, Rellich shows a result that we rephrase as follows. 

\begin{theorem}[Rellich \cite{Rellich37}]\label{thm:Rellich2}\index{Rellich's theorem}\index{theorem!Rellich's}
Let $I \subseteq \R$ be an open neighborhood of $0$ in $\R$ and let 
\[
    P_a(t)(Z) = Z^d + \sum_{j=1}^d a_j(t)Z^{d-j}, \quad t \in I,
\]
be a monic hyperbolic polynomial with real analytic coefficients $a_j$.   
Then there exist real analytic functions $\la_i$, $1 \le i \le d$, 
defined maybe in a smaller neighborhood $I'$ of $0$ in $\R$, 
which parameterize the roots of $P_a$, i.e.,
\[
    P_a(t)(Z) = \prod_{i=1}^d (Z-\la_i(t)),\quad t \in I'.
\] 
\end{theorem}

Actually, Rellich states this result for complex analytic functions that have real values for real $t$, 
see Rellich \cite[Hilfssatz II]{Rellich37}.  Again, for real $t$, it is easy to see that it holds globally, that is for an arbitrary $I$ and with $I'=I$.  

Rellich's proof of Theorem \ref{thm:Rellich2} is based on Puiseux's theorem.  Its idea goes as follows. Let $\lambda_i (t)$ be the complex roots of $P_a (t)$.  By Puiseux's theorem, there is $l\in \N$ such that $\eta_i(s) = \la_i (s^l)$ are all complex analytic, but also, by the assumption, every $\eta_i(s)$ is real whenever $s^l$ is real.  But this means that $\eta_i$, as a power series, depends only on $s^l$.  Indeed, suppose this is not true and let $c_k s^k$ be the lowest degree term in the expansion of $\eta_i$ such that $l$ does not divide $k$.
Then $\eta_i$ can not be real for all $s$ such that $s^l$ is real. 
(Rellich gives a slightly different argument still based on Puiseux's theorem in  \cite[pages 30-31]{Rellich69}.) 

Then to show Theorem \ref{thm:Rellich1} Rellich, for each analytic eigenvalue $\la (t)$ of $A(t)$, solves the equation $A(t)v = \la (t) v$ using formulas involving the minors of the matrix $A(t)$.  We refer the reader to Rellich's paper for this elementary argument.

\begin{remark}\label{rem:normal}
It was observed by Rainer \cite{RainerN} that Theorem \ref{thm:Rellich1} holds as well for normal matrices.  
This shows, surprisingly, that it is not the hyperbolicity of the characteristic polynomial of $A(t)$ that is 
important but the fact that we can diagonalize the matrices in orthonormal bases, but see also Example \ref{ex:3x3notnormal}. 
This was made even more transparent in the proof for normal matrices given in \cite{ParusinskiRondMatr}, 
where analytic diagonalizibility is showed directly without proving first the analyticity of the eigenvalues. See also Section \ref{ssec:normal}. 
\end{remark}

\subsection{A proof of Theorem \ref{thm:Rellich2}} \label{sec:Rellichproof}
We present here a different proof based on the splitting principle.

\subsubsection{Tschirnhausen form} \label{sec:Tschirnhausen}

Every monic real polynomial
\[
    P_a(Z) = Z^d + \sum_{j=1}^d a_j Z^{d-j}
\]
of degree $d$ can be identified with the point $a = (a_1,\ldots,a_d) \in \R^d$.
We say that $P_a$ is in \emph{Tschirnhausen form}\index{Tschirnhausen form} if $a_1 =0$. 
Every $P_a$ can be put in Tschirnhausen form by the substitution
\[
    P_{\tilde a}(Z) = P_a(Z - \tfrac{a_1}d) = Z^d + \sum_{j=2}^d \tilde a_j Z^{d-j}.
\]
We call it \emph{Tschirnhausen transformation}\index{Tschirnhausen transformation}.
Note that 
\begin{equation} \label{eq:Tschirn}
    \tilde a_j  = \sum_{i=0}^j C_i a_i a_1^{j-i}, \quad 2 \le j \le d,
\end{equation}
where the $C_i$ are universal constants and $a_0=1$.
For a polynomial $P_{\tilde a}$ in Tschirnhausen form we have 
\begin{equation*}
    -2 \tilde a_2 = \la_1^2 + \cdots + \la_d^2,
\end{equation*}
where $\la_1,\ldots,\la_d$ are the roots of $P_{\tilde a}$ (with multiplicities), 
consequently, for a hyperbolic polynomial, 
\begin{equation*}
    \tilde a_2 \le 0.
\end{equation*}
Recall that the coefficients (up to their sign) are the elementary symmetric polynomials in the roots, by Vieta's formulas.

\begin{lemma} \label{lem:dominant}
    The coefficients of a hyperbolic polynomial $P_{\tilde a}$ in Tschirnhausen form satisfy
    \begin{equation}
        |\tilde a_j|^{1/j} \le \sqrt 2\, |\tilde a_2|^{1/2}, \quad j = 1,\ldots,d.
    \end{equation}
\end{lemma}

\begin{proof}
    Let $s_k := \la_1^k + \cdots + \la_d^k$, $k \ge 1$, be the Newton polynomials\index{Newton polynomials}\index{polynomial!Newton} in $\la_1,\ldots,\la_d$.
    We claim that 
    \begin{equation} \label{eq:Newtondominant}
        |s_k|^{1/k} \le  |s_2|^{1/2}, \quad k = 2,\ldots,d.
    \end{equation}
    Setting $\la = (\la_1,\ldots,\la_d)$, we 
    have, for $k\ge 2$,
    \[
        |s_k|^{1/k} \le \|\la\|_k \le \|\la\|_2 = |s_2|^{1/2},
    \]
    by a well-known relation between the $p$-norms.
    The Newton identities\index{Newton identities} between the Newton polynomials and the elementary symmetric polynomials $\si_k$,\index{polynomial!elementary symmetric}
    \[
        j \si_j = \sum_{i=1}^j (-1)^{i-1} \si_{j-i} s_i, \quad d \ge j \ge 1,
    \]
    imply
    \[
        |\tilde a_j| \le \frac{1}{j} \sum_{i=2}^j |\tilde a_{j-i}| |s_i| 
        \le \frac{1}{j}  \sum_{i=2}^j |\tilde a_{j-i}| |s_2|^{i/2}.
    \]
    By \eqref{eq:Newtondominant}, it is now 
    easy to conclude the statement using induction.
\end{proof}

\subsubsection{Splitting} \label{sec:splittingRellich}\index{splitting}

The following well-known lemma can be found for instance in \cite{BM1990}.  

\begin{lemma} \label{lem:splitting}
    Let $P_a = P_b P_c$, where $P_b$ and $P_c$ are monic real polynomials without common (complex) root.
    Then 
    we have $P=P_{b(P)}P_{c(P)}$ for analytic mappings $P \mapsto b(P) \in \R^{\deg P_b}$ and $P \mapsto c(P) \in \R^{\deg P_c}$, 
    defined for $P$ near $P_a$ in $\R^{\deg P_a}$, with the given initial values.
\end{lemma}

\begin{proof}
    The product $P_a = P_a P_c$ defines on the coefficients a polynomial map $\vh$ such that $a = \vh(b,c)$.
    Its Jacobian determinant $\det \on{Jac}_{\vh}(b,c)$ equals the resultant of $P_b$ and $P_c$ which is nonzero, by assumption.
    Thus $\vh$ can be inverted locally, by the inverse function theorem.
\end{proof}

We denote by $\Hyp(d)$\index{Hypd@$\Hyp(d)$} the space of monic hyperbolic polynomials of degree $d$.
Let $\Hyp_T(d)$\index{HypTd@$\Hyp_T(d)$} be the space of monic hyperbolic polynomials of degree $d$ in Tschirnhausen form 
and let $\Hyp_T^0(d)$\index{HypT0d@$\Hyp_T^0(d)$} be the subspace of polynomials $P_{\tilde a}$ with $\tilde a_2 = -1$. 
By Lemma~\ref{lem:dominant}, $\Hyp_T^0(d)$ is compact. 
Every polynomial in $\Hyp_T^0(d)$ can be written as a product of two monic real polynomials of positive degree without common root 
(indeed, the sum of the roots vanishes while the sum of their squares is $2$).

Let us now prove Theorem \ref{thm:Rellich2}.
The case $\tilde a_2=0$ is trivial, since then all $\tilde a_j=0$.  Thus let $P_{\tilde a} \in \Hyp_T(d)$ be such that $\tilde a_2 \ne 0$.  Since $- \tilde a_2(t)$ is positive,  
it is of the form $- \tilde a_2(t)= t^{2k} u(t)$, $u(0)>0$.  Choose one of the
two analytic roots of $- \tilde a_2(t)$, 
$\pm  t^{k} u^{\frac 1 2}(t)$, and denote it by $\ansqrt (t)$.   
Then the polynomial
\[
    Q_{\ul a}(Z):= \ansqrt^{-d} P_{\tilde a} (\ansqrt Z)  = 
    Z^d - Z^{d-2} + \sum_{j=3}^d \ansqrt^{-j} \tilde a_j Z^{d-j}
\]    
belongs to $\Hyp_T^0(d)$.
By Lemma \ref{lem:splitting}, 
we have 
\[
    Q_{\ul a} = Q_{\ul b} Q_{\ul c},
\]
on some open ball $B(P_{\tilde a},r) \subseteq \R^d$ 
such that $\deg Q_{\ul b} <d$, $\deg Q_{\ul c} <d$, and
\[
    \ul b_i = \ps_i(\ansqrt^{-3}\tilde a_3,\ldots,\ansqrt^{-d}\tilde a_d), \quad i = 1,\dots, \deg Q_{\ul b},
\]
where $\ps_i$ are real analytic functions; likewise for $\ul c_i$. 
Note that, by Lemma \ref{lem:dominant}, $\ansqrt^{-i}\tilde a_i$ are real analytic and hence so are $b_i$ and $c_i$.  
If $Q_{\ul a}$ is hyperbolic, then also $Q_{\ul b}$ and $Q_{\ul c}$ are hyperbolic.
This induces a splitting
\begin{align}\label{eq:spliting}
    P_{\tilde a} = P_b P_c, \quad \text{ on } B(P_{\tilde a},r), 
\end{align}
where 
\begin{equation*} 
    b_i = \ansqrt^{i} \ps_i(\ansqrt^{-3}\tilde a_3,\ldots,\ansqrt^{-d}\tilde a_d), \quad i = 1,\dots, \deg P_{b} , 
\end{equation*}
are real analytic functions of $t$  and likewise for $c_i$. 
To conclude we proceed by induction on $\deg P_a$. 
\qed

\begin{remark}
    At first the above proof looks more complicated than
    Rellich's original argument.  This is because 
Rellich uses Puiseux's theorem, that can be proven itself using the splitting.  
But the splitting allows to work with coefficients $a_i$ that are not necessarily real analytic, but 
 $C^p$ or $C^\infty$ for instance, and that cannot be easily complexified.  It can also be used in the 
multiparameter case.   
\end{remark}

It is natural to ask what happens if the entries of a Hermitian matrix are $C^\infty$ functions of $t \in \R$.
Then, in general, there is a drastic drop of regularity for the eigenvalues (this will be discussed in detail in Section \ref{sec:Hermitian} and Section \ref{sec:normal}), 
while the eigenvectors may not even admit a continuous choice, as seen in the following example, see also Example \ref{ex:defnormal}.  

\begin{example}[{\cite[\S 2.2]{Rellich37}}]
   The $C^\infty$ curve of symmetric matrices 
   \[
       A(t) := e^{-1/t^2} 
       \begin{pmatrix}
           \cos t^{-1} & \sin t^{-1}
       \\
           \sin t^{-1} & - \cos t^{-1}
       \end{pmatrix}, \quad t \ne 0,
   \]
   and $A(0):=0$
   has the $C^\infty$ eigenvalues $\pm e^{-1/t^2}$, but the normalized eigenvectors cannot be continuous at $t=0$. 
\end{example}

Under the additional assumption that no two continuous eigenvalues meet to infinite order of flatness,  Rellich's theorem \ref{thm:Rellich1} 
still holds in the $C^\infty$ category; see \cite[Theorem 7.6]{AKLM98} and \cite{RainerN} for normal matrices.

\subsection{Multiparameter case}

Both Rellich's theorems fail in the multiparameter case.

\begin{example}[{Rellich \cite[{\S 2}, 1]{Rellich37}}]
The eigenvalues of the symmetric matrix 
    \[
        A(x_1,x_2) = 
        \begin{pmatrix}
            1+2x_1 & x_1+x_2
            \\
            x_1+x_2  & 1+2x_2
        \end{pmatrix},
    \]
$\lambda_\pm=1+ x_1+x_2 \pm \sqrt {2(x_1^2 +x_2^2)}$, are not analytic in any neighborhood of $(0,0)$. 
\end{example}

In \cite{KurdykaPaunescu2008} K. Kurdyka and L. P\u aunescu proposed multiparameter versions of Rellich's theorems.  The following is a multiparameter generalization of Theorem \ref{thm:Rellich2}.

\begin{theorem}[{\cite[Theorems 4.1 and 5.8]{KurdykaPaunescu2008}}]\label{thm:KP1}
 Let $U \subseteq \R^n$ be open and 
\[
    P_a(x)(Z) = Z^d + \sum_{j=1}^d a_j(x)Z^{d-j}, \quad x\in U,
\]
be a monic hyperbolic polynomial with real analytic coefficients $a_j$.    
Then there exist locally Lipschitz functions $\la_i : U \to \R$, $i= 1, \ldots , d$,
which parameterize the roots of $P_a$.  

Moreover, there exists a modification $\sigma : W \to U$ (by a locally finite composition of  
blowings-up with global smooth centers) such that the roots of $P_a(\sigma (w))(Z)$ can be parameterized, locally on $W$, by real analytic functions.  
\end{theorem}

Since all roots of $P_a(x)$ are real, we can order them 
$$\la_1(x) \le \la_2(x) \le \cdots \le \la _d(x) $$ 
and the first claim of the above theorem implies that all 
such $\la_i$ are locally Lipschitz.  

\begin{example}[{\cite[Example 5.9]{KurdykaPaunescu2008}}]
Consider $P_a(x)(Z) = Z^ 2- (x_1^2 + x_2 ^2)$. The zero set of $P_a$ is the double cone. 
 Then both the upper and the lower nappes are graphs of Lipschitz functions. But note that if we restrict them to a line through the origin 
 these restrictions are not real analytic.  We have to interchange them while passing through the origin to get analytic roots 
 (that exist by
 Rellich's theorem).

If we blow up the origin in the real plane, then the roots become locally, but not globally, real analytic.  Indeed, 
in the chart $x_1=w_1, x_2 = w_1 w_2$ the roots are $\pm w_1 \sqrt{1+w_2^2}$.
Going along the exceptional divisor, 
the center circle of the Möbius band,  changes the sign of the root, i.e., interchanges them.    
\end{example}

To show the second claim of Theorem \ref{thm:KP1}, Kurdyka and P\u aunescu   use
resolution of singularities, splitting,  and the Abhyankar--Jung theorem.

\subsubsection{Abhyankar--Jung theorem}\label{sec:AJ}
Let $\K$ be a field and let 
\begin{align}
P(Z)=Z^d+a_1(X) Z^{d-1}+\cdots+a_d(X) \in \K[[X]][Z]
\end{align}
be a monic polynomial whose coefficients are formal power series in 
$X= (X_1, \ldots, X_n)$.   
Such a polynomial  $P$ is called \emph{quasi-ordinary}\index{quasi-ordinary polynomial}\index{polynomial!quasi-ordinary} if the  discriminant $\discr_P (X)$ 
equals $X_1^{\alpha_1} \cdots X_n^{\alpha_n} U(X)$, with  $\alpha_i \in \N$ and $U(0)\ne 0$.   

\begin{theorem}[Abhyankar--Jung theorem]\label{thm:AJ}\index{Abhyankar--Jung theorem}\index{theorem!Abhyankar--Jung}
Let $\K$ be an algebraically closed field of characteristic zero and let 
$P \in \K[[X]][Z]$ be a quasi-ordinary polynomial 
such that the  discriminant of $P$ is of the form  
$\discr_P (X)= X_1^{\alpha_1} \cdots X_r^{\alpha_r} U(X)$, where 
$U(0)\ne 0$, and $r\le n$.  
Then there is $q\in\N\setminus \{0\}$ such that $P(Z)$ has its roots in 
$\K[[X_1^{\frac{1}{q}},...,X_r^{\frac{1}{q}}, X_{r+1},...,X_n]]$.
\end{theorem} 

The Abhyankar--Jung theorem can be understood as a multiparameter version of Puiseux's theorem. 
It has first been proven by Jung in 1908 for $n=2$ and $\K=\C$, cf. \cite{J}.  
 The first complete proof of Theorem \ref{thm:AJ} appeared in \cite{Ab}. 

The Abhyankar--Jung theorem still holds if $P$ has multiple factors 
(so its discriminant is identically equal to zero) but the discriminant of the square-free reduction $P_{\on{red}}$ 
of $P$ is of the form $X_1^{\alpha_1} \cdots X_r^{\alpha_r} U(X)$.

The Abhyankar--Jung theorem can be proven using splitting, as it is done in \cite{ParusinskiRondA-J}.  
But in order to apply the splitting one needs the following result that is much more difficult to prove, see 
\cite[Theorem 1.1 and Remark 1.4]{ParusinskiRondA-J}.

\begin{theorem} \label{thm:towardsAJ}
Let $\K$ be a (not necessarily algebraically closed) field of characteristic zero and let 
$P \in \K[[X]][Z]$ be a quasi-ordinary polynomial in Tschirnhausen form  (i.e. $a_1=0$).   
Then  the ideal $(a_i^{d!/i}(X))_{i=2,\ldots,d}$ is generated by one of $a_i^{d!/i}(X)$ and this generator equals a monomial in $X_1, \ldots, X_n$ times a unit.  
\end{theorem}

\subsubsection{Proof of Theorem \ref{thm:KP1}}
We sketch a proof of this theorem.  

The first claim of the theorem follows from Bronshtein's theorem, see Theorem~\ref{thm:Bronshtein} below.  The proof given in \cite{KurdykaPaunescu2008} is different and based on real analytic geometry arguments. 

To show the second claim one first applies Hironaka's resolution of singularites 
to make the discriminant 
$\discr_{P_{\on{red}}}$ of $P_{\on{red}}$ 
normal crossings.  For this it suffices, by \cite[Main Theorem II''(N), page 158]{Hironaka:1964}, to apply a modification given by the composition of a locally finite number of global blowings-up with smooth centers. Therefore, 
one may suppose that $P_{\on{red}}$ is quasi-ordinary.  To conclude  Kurdyka and P{\u a}unescu use the Abhyankar--Jung theorem applied to a hyperbolic quasi-ordinary polynomial, see \cite[Proposition 5.4]{KurdykaPaunescu2008} for details.  
Instead,  one may use here directly Theorem \ref{thm:towardsAJ} and splitting.  Firstly, after applying the Tschirnhausen transformation we may suppose $a_1=0$.  
Secondly, by Lemma \ref{lem:dominant}, it is $a_2^{d!/2}$ that generates the ideal 
$(a_i^{d!/i})_{i=2,\ldots ,d}$.  Denote by $\theta (x)$ one of the two analytic square roots of $- a_2(x)$, 
that is, if $-a_2 (x)=  u(x) \prod x_k^{2 \alpha_ k} $, then $\theta(x) = \pm u^{\frac 1 2}(x) \prod x_k^{\alpha_ k}$.  
Then, one concludes exactly as in the proof of Rellich's theorem by splitting the polynomial $P$ and proceeding by induction on its degree.  
\qed

\subsection{Perturbation of normal matrices}\label{ssec:normal}

Kurdyka and P\u aunescu showed multiparameter counterparts of Theorem \ref{thm:Rellich1} 
for Hermitian matrices.

\begin{theorem}[{\cite[Theorem 6.2 and Remark 6.3]{KurdykaPaunescu2008}}] \label{thm:KP2}
Let $U \subseteq \R^n$ be open and let $A(x)$, $x\in U$, be an analytic family of Hermitian matrices.
   Then, there exists a locally finite composition of blowings-up with smooth global 
centers    $ \sigma: W \to U$, such that locally on $W$   the corresponding family
$ A \circ \sigma $ admits a simultaneous analytic diagonalization.  
\end{theorem}

They also gave a version for real antisymmetric matrices, {\cite[Theorem 7.2]{KurdykaPaunescu2008}}.  
 This was extended to normal matrices by Rainer \cite{RainerN} by showing that the eigenvalues and eigenspaces of normal matrices depend analytically on the parameters after a modification of the parameter space.  Note that  
these results  generalize the one parameter case because there are no nontrivial modifications of one dimensional nonsingular spaces.  

Recall that for a square matrix $A$ we denote by $\Delta_A$ the discriminant of the reduced (i.e., square-free) form $(\charp_A)_{\on{red}}$ of its characteristic polynomial $\charp _A$.  
The normal crossings assumption for $\Delta_A$ is sufficient for the local simultaneous reduction of families of normal matrices as it is showed by Parusi\'nski and Rond in \cite{ParusinskiRondMatr}.  
The following three algebraic theorems, stated in terms of formal power series $\K[[X]]$, $\K=\R,\C$, were proven in \cite{ParusinskiRondMatr}. 
They hold as well for several subrings of $\K[[X]]$ including the algebraic power series $\K\langle X\rangle$ and the 
convergent power series $\K\{ X\}$, see Remark \ref{rem:rings}.

\begin{theorem}[{\cite[Theorem 2.5]{ParusinskiRondMatr}}] \label{thm:normaltheorem}
    Let $A(X)\in \on{Mat}_d (\C[[X]])$ be normal and suppose that  
$\Delta_A(X) = X_1^{\alpha_1} \cdots X_n^{\alpha_n} u (X)$ 
with $u(0)\ne 0$.
Then there is a unitary matrix 
$U(X) \in \on{U}_d (\C[[X]])$ such that 
$$U(X)^{-1} A(X) U(X)= D(X),$$ 
where $D(X)$ is a diagonal matrix with entries in $\C[[X]]$.    
\end{theorem}

\begin{theorem}[{\cite[Theorem 2.9]{ParusinskiRondMatr}}] \label{thm:realnormaltheorem}
    Let $A(X)\in \on{Mat}_d (\R[[X]])$ be normal and suppose that  $\Delta_A(X) = X_1^{\alpha_1} \cdots X_n^{\alpha_n} u (X)$ 
with $u(0)\ne 0$.   Then there exists an orthogonal matrix 
$O(X) \in \on{O}_d (\R[[X]])$ such that $O(X)^{-1}  A(X)  O(X)$ is block diagonal with the blocks of size  
 $1$ and $2$.  The blocks of size $2$ 
 are $2\times 2$ matrices of the form
\begin{align}\label{blocs}
\left(\begin{array}{cc} a(X) & b(X)\\ -b(X) & a(X)\end{array}\right)
\end{align}
           for some $a(X)$, $b(X)\in\R[[X]]$.
\end{theorem}

Note that it follows
that if $A(X)$ is symmetric, resp.\ antisymmetric, then $O(X)^{-1}  A(X)  O(X)$ is symmetric (i.e., diagonal), resp.\ antisymmetric.

\begin{example}[{Rainer \cite[Example 8.3]{RainerN}}] \label{ex:3x3notnormal}
The eigenvalues of the one parameter diagonalizable, but not normal, matrix 
    \[
        A(t) = 
        \begin{pmatrix}
            t & 0 & 0   \\
            0 & 0 & t^2 \\
            0 & t & 0 
        \end{pmatrix}, \quad t\in \R,
    \]
are $t, \pm t^{3/2}$ for $t\ge 0$ and $t, \pm i|t|^{3/2}$ for $t< 0$. 
\end{example}

The above example shows that the eigenvalues of a real analytic family of diagonalizable matrices are not necessarily analytic (even for one parameter families).  
As the following theorem shows, for arbitrary analytic families of matrices, not necessarily square, the counterparts of the above two theorems hold for the singular value decomposition.

Let $A\in \on{Mat}_{m,n}  (\C)$.  Then $A=  V D U^{-1} $ 
for some unitary matrices $V \in \on{U}_m (\C)$,  
$U \in \on{U}_n (\C)$, and  a (rectangular) diagonal matrix $D$ with real nonnegative coefficients.  The 
diagonal elements of $D$ are the nonnegative square roots of the 
eigenvalues of $A^* A$ or, equivalently, of $A A^*$. They are called \emph{singular values}\index{singular values} of $A$. 
If $A$ is real, then $V$ and $U$ can be chosen orthogonal.

\begin{theorem}[{\cite[Theorems 3.1 and 3.3]{ParusinskiRondMatr}}] \label{thm:SVD}\index{singular value decomposition}
    Let $A=A(X)\in \on{Mat}_{m,d}  (\C [[X]])$, $d\le m$,  and suppose that $\Delta_{A^*A} (X) = X_1^{\alpha_1} \cdots X_n^{\alpha_n} u (X)$ 
    with $u(0)\ne 0$.  Then there are unitary matrices $V \in \on{U}_m (\C[[X]])$,  
    $U \in \on{U}_d (\C[[X]])$ such that 
$$D(X) = V(X)^{-1} A(X) U(X)$$ is (rectangular) diagonal.     

If $A=A(X)\in \on{Mat}_{m,d}  (\R [[X]])$, then $U$ and $V$ can be chosen real (that is orthogonal) so that $V(X)^{-1} A(X) U(X)$ is 
block (rectangular) diagonal with blocks  as in Theorem \ref{thm:realnormaltheorem}.  

Suppose in addition that the last nonzero coefficient of 
$\Delta_{A^*A}(X)$ is of the form $X_1^{\beta_1} \cdots X_n^{\beta_n} h (X)$ 
with $h(0)\ne 0$.  Then, both in the real and the complex case,
 we may require that $V(X)^{-1} A(X) U(X)$ is (rectangular) diagonal with the entries on the diagonal in $\R [[X]]$, and,  
moreover, that those entries that are  nonzero  are of the form  $X^\alpha a (X) $,  with $a(0)> 0$,  and that the exponents $\alpha \in \N^n$ are well-ordered.
\end{theorem}

In general, even if $d=m=1$, one needs the extra assumption on the last nonzero coefficient of 
$\Delta_{A^*A}$, see \cite[Example 3.2]{ParusinskiRondMatr}, though this assumption is automatically satisfied if $n=1$.  
It is in general not possible to have all the entries of $D(X)$ non-negative. 

\subsubsection{Remarks on the proofs}
In \cite{KurdykaPaunescu2008} the authors  first use Theorem \ref{thm:KP1} in order to get, after a modification of 
the parameter space,  the eigenvalues real analytic.  
Then they  solve linear equations describing  the eigenspaces corresponding to the irreducible factors of the characteristic polynomial.  
This requires to monomialize the ideal defined by all the minors of the associated matrices which usually necessitates 
 further blowings-up.  
A similar approach is adapted in  \cite{RainerN}.  
First the eigenvalues are made analytic by blowings-up and then further blowings-up are necessary, for instance to make the coefficients of matrices and their differences normal crossing.  

In  \cite{grandjean2019} Grandjean shows results similar to those of 
\cite{KurdykaPaunescu2008} and \cite{RainerN} but by a different approach.  He does 
not treat the eigenvalues first, but considers directly 
the eigenspaces defined on the complement of the discriminant locus. We refer the reader to \cite{grandjean2019} for details.

The proofs in \cite{ParusinskiRondMatr} also consider the eigenvalues and eigenspaces at the same time and 
only the assumption that the discriminant of the reduced characteristic polynomial is normal crossings is necessary.    
These proofs adapt a strategy similar to the algorithm of the proof of the 
Abhyankar--Jung theorem of \cite{ParusinskiRondA-J} and of the proof of Rellich's theorem \ref{thm:Rellich2} that we presented in 
Section \ref{sec:Rellichproof}.  Given a normal matrix $A(X)\in \on{Mat}_d(\C[[X]])$,  first one subtracts from $A(X)$ the identity matrix times $\on{Tr} A$, so one can assume $\on{Tr} A=0$.  This corresponds to the Tschirnhausen transformation.  Then 
the following proposition replaces Theorem \ref{thm:towardsAJ}.

\begin{proposition}[{\cite[Proposition 2.7]{ParusinskiRondMatr}}]\label{lem:MainLemma} 
Suppose that  the assumptions of Theorem  \ref{thm:normaltheorem} are satisfied and that, moreover, $A=(a_{ij}(X))$ is 
nonzero and $\on{Tr} A=0$.  Then the ideal  $(a_{ij})_{i,j=1,\ldots ,d}$ of $\C [[X]]$   is generated by one of 
$a_{ij}$ and this generator equals a monomial in $X_1, \ldots, X_n$ times a unit.  
\end{proposition}

Note that all the coefficients $a_{ij}$ of $A$ have the same weight, 
in contrast to Theorem~\ref{thm:towardsAJ}, where suitable powers of the coefficients of $P$ must be considered.

Finally, the following version of Hensel's lemma for normal matrices replaces the splitting lemma,  Lemma \ref{lem:splitting}. 
This result is a strengthened version of Cohn's version of Hensel's lemma, see \cite[Lemma 1]{cohn:84-88}.

       \begin{lemma}[{\cite[Lemma 2.1]{ParusinskiRondMatr}}]\label{lem:SplitMat}
 Let $A(X)\in \on{Mat}_d(\C[[X]])$ be a normal matrix.      Assume that 
$$
A(0) =   \begin{pmatrix}
  B\at _1 & 0 \\
  0 & B\at _2 
  \end{pmatrix}, 
$$
 with $B\at _i \in \on{Mat}_{d_i} (\C)$, $d=d_1+d_2$,  such that  the characteristic polynomials of $B\at _1$ and 
 $B\at _2$ are coprime.  
Then  there is a unitary matrix  $U \in \on{U}_d (\C[[X]])$, $U(0) = \I$, such that 
   \begin{align}\label{eq:equationforA}
U ^{-1} A U =   \begin{pmatrix}
B_1 & 0 \\
  0 &  B_2 
  \end{pmatrix}, 
  \end{align} 
  and $B_i (0) = B\at _i$, $i=1,2$.  
  \end{lemma}

\begin{remark}\label{rem:rings}
Theorems \ref{thm:normaltheorem} (for $\K=\C$), \ref{thm:realnormaltheorem} (for $\K=\R$), and \ref{thm:SVD} (for $\K=\R, \C$) 
remain valid if we replace $\K[[X]]$ by a ring $\K\{\!\!\{X\}\!\!\}$ satisfying the following properties:
\begin{enumerate}
\item $\K\{\!\!\{X_1,\ldots, X_n\}\!\!\}$ contains $\K[X_1,\ldots, X_n]$,
\item $\K\{\!\!\{X_1,\ldots, X_n\}\!\!\}$ is a Henselian local ring with maximal ideal $(X_1$, \ldots, $X_n)$,
\item $\K\{\!\!\{X_1,\ldots, X_n\}\!\!\}\cap (X_i)\K[[X_1,\ldots, X_n]]=(X_i)\K\{\!\!\{X\}\!\!\}$ for  
$i=1, \ldots, n$.
\end{enumerate}
Important examples of such subrings are:  the algebraic power series $\K\langle X\rangle$, the 
convergent power series $\K\{ X\}$, and the ring of germs of quasianalytic $\K$-valued functions over $\R$ (i.e. 
$C^\infty$ functions satisfying (3.1) -- (3.6) of \cite{BM2004}). 
\end{remark}

\begin{corollary}[{\cite[Theorem 5.3]{ParusinskiRondMatr}}]
Let $M$ be a manifold defined in one of the following categories:
\begin{enumerate}
\item real analytic; 
\item real Nash;
\item quasianalytic (under the asssumptions (3.1) -- (3.6) of \cite{BM2004}).
\end{enumerate}
Let $A$ be a matrix whose coefficients are regular functions on $M$ (depending on the category) and let  $K$ be a compact subset of $M$. Then there exist a neighborhood $\Omega$ of $K$ and the composite of a finite sequence of blowings-up  with smooth centers $\pi: U\to \Omega$, such that locally on $U$ 
\begin{enumerate}
\item[a.] if $A$ is a complex normal matrix, then $A\circ\pi$ satisfies the conclusion of  Theorem~\ref{thm:normaltheorem};
\item[b.] if $A$ is a real normal matrix, then $A\circ\pi$ satisfies the conclusion of  Theorem~\ref{thm:realnormaltheorem};
\item[c.] if $A$ is a (not necessarily square) matrix, then $A\circ\pi$  satisfies the conclusion of Theorem~\ref{thm:SVD}.
\end{enumerate}
\end{corollary}

Here one has to work with real local coordinates, since a complex change of coordinates does not commute with 
complex conjugation and may destroy the assumption that the matrix is normal, as seen in the following example.

\begin{example}[{\cite[Example 6.1]{KurdykaPaunescu2008}}]
The eigenvalues of the symmetric matrix 
    \[
        A(x_1,x_2) = 
        \begin{pmatrix}
            x_1^2 & x_1x_2
            \\
            x_1x_2  & x_2^2
        \end{pmatrix},
    \]
are $0$ and $x_1^2 + x_2^2$ and $(1,x_2/x_1)$ and  $(1,-x_1/x_2)$ are the corresponding eigenvectors. 
 The discriminant of the 
characteristic polynomial $(x_1^2 + x_2^2)^ 2$ is not normal crossings over the reals.  It is so over the 
complex numbers, $(x_1 +i x_2)^ 2 (x_1 -i x_2)^ 2$, but in the new system of coordinates $(x_1 +i x_2, x_1 -i x_2)$ 
the matrix is not normal any longer.  
To diagonalize this family simultanously one has to blow-up the origin. 
\end{example}

\begin{remark}
    Many of the results presented for matrices extend (in appropriate form) to parameterized families of
    unbounded normal or selfadjoint operators 
    in Hilbert space with common domain of definition and compact resolvent. See e.g.\ \cite{AKLM98}, \cite{RainerN}, and references therein. 
\end{remark}

\section{Differentiable roots of hyperbolic polynomials}\label{sec:Bronshtein}

\subsection{Bronshtein's theorem}

Let $I \subseteq \R$ be an open interval and 
\[
    P_a(t)(Z) = Z^d + \sum_{j=1}^d a_j(t)Z^{d-j}, \quad t \in I,
\]
a monic hyperbolic polynomial. 
It is not hard to see (e.g., \cite[Theorem 4.3]{AKLM98} and Remark \ref{rem:orig-proof} below) that if $a_j \in C^d(I)$, $1 \le j \le d$, 
then there exist differentiable functions $\la_i : I \to \R$, $1 \le i \le d$, 
which parameterize the roots of $P_a$, i.e.,
\[
    P_a(t)(Z) = \prod_{i=1}^d (Z-\la_i(t)),\quad t \in I.
\]
Actually, it suffices to assume that the coefficients are of class $C^p$ if $p$ is the maximal multiplicity of the roots.

Bronshtein \cite{Bronshtein79} proved that the roots $\la_i$ have locally bounded derivative:

\begin{theorem}[Bronshtein's theorem I, {\cite{Bronshtein79}}] \label{thm:BronshteinI}\index{Bronshtein's theorem}\index{theorem!Bronshtein's}
    Let $I \subseteq \R$ be an open interval and $Y$ a compact Hausdorff topological space.
    Let $P_a(t,y)$, $(t,y) \in I \times Y$, be a monic hyperbolic polynomial of degree $d$ 
    such that the multiplicity of its roots does not exceed $p$ and
    $\p^k_t a_j(t,y)$, $1 \le j \le d$, $0 \le k \le p$, are continuous functions on $I \times Y$.
    Then, for each compact $K \subseteq I \times Y$, 
    there is a constant $C=C(K)>0$ such that the roots of $P$ can be represented by functions $\la_i(t,y)$ 
    that are differentiable in $t$ and
    satisfy 
    \[
        |\p_t \la_i(t,y)| \le C, \quad (t,y) \in K,\, 1 \le i \le d.
    \]
\end{theorem}

Bronshtein's original proof is not easy to penetrate (cf.\ Remark \ref{rem:orig-proof}). 
A much simpler proof was given by Parusi\'nski and Rainer \cite{ParusinskiRainerHyp}. 
It also gives uniform bounds for the Lipschitz constant of the roots in terms of the coefficients.

\begin{theorem}[Bronshtein's theorem II, {\cite{ParusinskiRainerHyp}}] \label{thm:Bronshtein}\index{Bronshtein's theorem}\index{theorem!Bronshtein's}
    Let $I \subseteq \R$ be an open interval.
    Let $P_a(t)$, $t \in I$, be a monic hyperbolic polynomial of degree $d$ with coefficients $a_j \in C^{d-1,1}(I)$, $1 \le j \le d$.
    Then any continuous root $\la \in C^0(I)$ of $P_a$ is locally Lipschitz 
    and, for any pair of relatively compact open intervals $I_0 \Subset I_1 \Subset I$,
    \begin{equation} \label{eq:Lipconst}
        \on{Lip}_{I_0}(\la) \le C\, \max_{1 \le j \le d} \|a_j\|^{1/j}_{C^{d-1,1}(\ol I_1)},
    \end{equation}
    with $C = C(d) \, \max \{\delta ^{-1}, 1\}$, where $\de := \on{dist}(I_0, \R \setminus I_1)$. 
\end{theorem}

The last claim of the above theorem, the form of the constant $C$, follows from the proof 
of \cite[Theorem 2.1]{ParusinskiRainerHyp}, see also Lemma \ref{lem:ass>adm} below.  
Note that the Lipschitz constant may blow up if we approach the boundary of $I$ as for 
$P_a(t) (Z) = Z^d -t$ and $I=(0,1)$.  We remark also that a system of continuous roots always exists, for hyperbolic polynomials one may order them for instance: 
$\la_1(t) \le \la_2(t) \le \ldots \le \la_d(t)$. 

\begin{remark} \label{rem:multiplicityhyp}
    If $p$ is the maximal multiplicity of the roots of $P_a$, then it is enough that $a_j \in C^{p-1,1}(I)$, $1 \le j \le d$, 
    for any continuous root $\la$ to be locally Lipschitz on $I$. The bound for the Lipschitz constant \eqref{eq:Lipconst} of $\la$ has to be modified; 
    it will also depend on a quantity that measures the ``uniformity'' of having multiplicity at most $p$, namely,  
    \[
        B:= \sup_{t \in I} \frac{\la_d(t)-\la_1(t)}{ \min_{1 \le i \le d-p}(\la_{i+p}(t) - \la_i(t))}, 
    \]
    where $\la_1(t) \le \la_2(t) \le \ldots \le \la_d(t)$, $t \in I$, is an increasing enumeration of the roots.
    See \cite[Section 4.7]{ParusinskiRainerHyp} for details.
\end{remark}

We will give a full proof of Theorem \ref{thm:Bronshtein} in Section \ref{sec:Bronshteinproof}.
A multiparameter version follows immediately (because a function defined on a box that is Lipschitz with respect to each variable is Lipschitz).

\begin{theorem}[{\cite[Theorem 2.2]{ParusinskiRainerHyp}}]
    Let $U \subseteq \R^n$ be open.
    Let $P_a(x)$, $x \in U$, be a $C^{d-1,1}$ family of monic hyperbolic polynomials of degree $d$.
    Then any continuous root $\la \in C^0(U)$ of $P_a$ is locally Lipschitz 
    and, for any pair of relatively compact open subsets $U_0 \Subset U_1 \Subset U$,
    \begin{equation}
        \on{Lip}_{U_0}(\la) \le C(d,n,U_0,U_1)\, \max_{1 \le i \le d} \|a_i\|^{1/i}_{C^{d-1,1}(\ol U_1)}.
    \end{equation}
\end{theorem}

Recall that $\Hyp(d)$ denotes the space of monic hyperbolic polynomials of degree $d$. 
Then $\Hyp(d)$ can be identified with a closed semialgebraic subset of $\R^d$; see Section~\ref{sec:spaceHd}. 
Consider the map $\la^\uparrow : \Hyp(d) \to \R^d$ which assigns to $P_a$ its increasingly ordered roots. 
Then $\la^\uparrow$ is continuous.

\begin{corollary} 
    Let $U \subseteq \R^n$ be open.
    Then the push forward 
    \[
        (\la^\uparrow)_* : C^{d-1,1}(U,\Hyp(d)) \to C^{0,1}(U,\R^d),  \quad P_a \mapsto \la^\uparrow \o P_a,
    \]
    is bounded.
\end{corollary}

It is natural to ask if the map $(\la^\uparrow)_*$ is continuous.
The answer is no:

\begin{example}\label{ex:Br-continuity}
    Let $f(t) := t^2$ and $f_n(t) := t^2 + 1/n^2$, $n \ge 1$.
    Then, for all $k\in \N$ and each bounded open interval $I \subseteq \R$,  $\|f-f_n\|_{C^k(\ol I)} = 1/n^2 \to 0$ as $n \to \infty$.
    Let $\la$ and $\la_n$ be the positive square roots of $f$ and $f_n$, respectively: 
    $\la(t) := |t|$ and $\la_n(t):= \sqrt{t^2 + 1/n^2}$. Then, for each bounded open interval $I \subseteq \R$ that contains $0$, 
    \begin{align*}
        \Lip_{I}(\la - \la_n) &\ge 
        \sup_{0< t \in I} \Big|\frac{(\la(t) - \la_n(t)) - (\la(0)-\la_n(0))}{t} \Big| 
        \\
                              &=  \sup_{0< t \in I} \Big|\frac{t - \sqrt{t^2 + \frac{1}{n^2}} + \frac{1}n}{t}\Big|
                              \ge \Big|\frac{\frac{1}n- \sqrt{\frac{1}{n^2}+ \frac{1}{n^2}} + \frac{1}n}{\frac{1}n}\Big| 
                              = 2-\sqrt 2,
    \end{align*}
    for large enough $n$.
    Consequently, the map $(\la^\uparrow)_* : C^k(\R,\Hyp(2)) \to C^{0,1}(\R,\R^2)$ (where $k\ge 2$) 
    is not continuous.
\end{example}

But $(\la^\uparrow)_* : C^{2}(\R,\Hyp(2)) \to W^{1,p}_{\on{loc}}(\R,\R^2)$ is continuous for each $1 \le p< \infty$:

\begin{remark}\label{rem:Br-continuity}
 Let $f \in C^{2}(\mathbb R, [0,\infty))$ and  $f_n \in C^{2}(\mathbb R, [0,\infty))$, $n \ge 1$, 
    be such that $\|f-f_n\|_{C^{2}(\overline I)} \to 0$ as $n \to \infty$, for each 
    bounded open interval $I \subseteq \mathbb R$.
    Let $\lambda$ and $\lambda_n$ be the nonnegative square roots of $f$ and $f_n$, respectively.

    Then $\|\lambda - \lambda_n\|_{L^\infty(I)} \to 0$ as $n\to \infty$, since
    \[
        |\lambda(t)-\lambda_n(t)|^2 \le |\lambda(t)-\lambda_n(t)||\lambda(t)+\lambda_n(t)| = |f(t)-f_n(t)|.
    \]

    By \Cref{thm:Bronshtein}, $\lambda$ and $\lambda_n$ are locally Lipschitz, thus differentiable almost everywhere, 
    and, for intervals $I_0 \Subset I_1$ and almost every $t \in I_0$, 
    \[
        |\lambda'(t)-\lambda_n'(t)| \le C(I_0,I_1) \Big( \|f\|_{C^{1,1}(\overline  I_1)}^{1/2} 
        + \sup_{n\ge 1} \|f_n\|_{C^{1,1}(\overline  I_1)}^{1/2}\Big).
    \]
    We will show that $\lambda_n' \to \lambda'$ pointwise almost everywhere as $n \to \infty$. 
    Thus the dominated convergence theorem implies that
    $\|\lambda'-\lambda_n'\|_{L^p(I_0)} \to 0$ as $n \to \infty$, for each $1 \le p < \infty$.

    If $f(t_0) \ne 0$, then it is easy to see that $\lambda_n'(t_0) \to \lambda'(t_0)$ as $n \to \infty$.
    Let $t_0$ be an accumulation point of the zero set of $f$. Then $f(t_0)=f'(t_0) = f''(t_0)=0$.
    Fix $\epsilon>0$. There exists $\delta >0$ such that $\|f''\|_{L^\infty(J)} \le \epsilon^2/2$, 
    where $J := \{t \in \mathbb R : |t-t_0|< \delta\}$, and there exists $n_0 \ge 1$ such that 
    $\|f'' - f_n''\|_{L^\infty(J)} \le \epsilon^2/2$ and $|f_n(t_0)| \le \delta^2 \epsilon^2$, for $n \ge n_0$. 
    If, moreover, $f_n(t_0) \ne 0$ then \Cref{lem:Glaeser} (applied to $f_n$ with $I = J$ and $M = \epsilon$) yields
    $|\lambda_n'(t_0)| \le \epsilon$. This implies the assertion because the isolated zeros of $f$ and $f_n$ 
    form a set of measure zero.    
\end{remark}

\begin{open}\label{op:Br-continuity}
    Is this also true for $d\ge 3$?\footnote{This was recently proved in \url{https://arxiv.org/abs/2410.01321}.}
\end{open}

It turns out that the derivatives of a differentiable choice of the roots not only are bounded but 
they are continuous, provided that the coefficients are $C^p$, where $p$ is the maximal multiplicity of the roots.
This was first observed by Colombini, Orr\'u, and Pernazza \cite{ColombiniOrruPernazza12} and a 
short proof was given by Parusi\'nski and Rainer \cite{ParusinskiRainerHyp}. 

\begin{theorem}[{\cite[Theorem 2.4]{ParusinskiRainerHyp}}] \label{thm:C1roots}
    Let $P_a(t)$, $t \in I$, be a $C^p$ curve of monic hyperbolic polynomials, where $p$ is 
    the maximal multiplicity of the roots.
    Then:
    \begin{enumerate}
        \item Any continuous root $\la$ of $P_a$ has left- and right-sided derivatives $\la^{\prime\pm}$ at each $t \in I$.
        \item $\la^{\prime-}$ and $\la^{\prime+}$ are continuous in the following sense: for every $t_0 \in I$, $\lim_{t \to t_0^-} \la^{\prime \pm}(t) = \la^{\prime-}(t_0)$ 
            and $\lim_{t \to t_0^+} \la^{\prime \pm}(t) = \la^{\prime+}(t_0)$. 
        \item There exist differentiable systems of the roots, any differentiable root is $C^1$.
    \end{enumerate}
\end{theorem}

If the coefficients in Theorem \ref{thm:C1roots} are of class $C^{2p}$, then 
there are even $C^1$ systems of the roots such that their derivatives of second order exist everywhere.
This is due to Colombini, Orr\'u, and Pernazza \cite{ColombiniOrruPernazza12} following the strategy 
of Kriegl, Losik, and Michor \cite{KLM04} who needed $C^{3p}$ coefficients. 
In general, these results cannot be improved, even for $C^\infty$ coefficients.

\begin{example} \label{ex:optimalhyp}
    (1) The function $f : \R \to [0,\infty)$ given by 
    \[
        f(t) := 
        \begin{cases}
            e^{-1/|t|} (\sin^2(\pi/|t|)+e^{-1/t^2})  &\text{ if } t\ne 0,
            \\
            0 &\text{ if } t= 0,
        \end{cases}
    \]
    is $C^\infty$, but no solution of $Z^2=f(t)$ is $C^{1,\al}$ for any $\al>0$ (cf.\ \cite{BBCP06}). 
    It was shown in \cite{BBCP06} that, given any modulus of continuity $\om$, there exists a $C^\infty$ function on $\R$ (depending on $\om$)
    such that no solution of $Z^2=f(t)$ is $C^{1,\om}$.

    (2) A nonnegative $C^{1,1}$ function $g$ such that no solution of $Z^2 = g(t)$ is differentiable at $t=0$ is  
    \[
        g(t) := 
        \begin{cases}
            t^2 \sin^2(\log |t|)  &\text{ if } t\ne 0,
            \\
            0 &\text{ if } t= 0.
        \end{cases}
    \] 

    (3) Replacing $t^2$ by $t^4$ in the definition of $g$ 
    we get a $C^{3,1}$ function $h$ such that no solution of $Z^2 = h(t)$ is twice differentiable at $t=0$.  

    (4) Clearly, the solutions of our equation cannot in general be locally Lipschitz if the right-hand side is not $C^{1,1}$, 
    e.g., $Z^2 = t^{1+\al}$, where $\al \in (0,1)$.

    We refer to \cite{ColombiniOrruPernazza12} for other illuminating examples (including hyperbolic polynomials of degree $\ge 3$).
\end{example}

Wakabayashi \cite{Wakabayashi86} gave a different (complex analytic) proof of a more general H\"older version of Bronshtein's theorem
which had been announced by Ohya and Tarama \cite{OhyaTarama86}. 
A proof based on Bronshtein's original method can be found in Tarama \cite{Tarama06}.

\begin{theorem}[Bronshtein's theorem III, {\cite[Theorem 1.4]{Tarama06}}] \label{thm:BronshteinIII}\index{Bronshtein's theorem}\index{theorem!Bronshtein's}
    Let $I \subseteq \R$ be an open bounded interval.
    Let $P_a(t)$, $t \in I$, be a monic hyperbolic polynomial of degree $d$ 
    such that the multiplicity of its roots does not exceed $p\ge 2$ and $a_j \in C^{k,\al}(I)$, $1 \le j \le d$.
    Then each continuous root $\la$ of $P_a(t)$ is locally H\"older continuous of index 
    \[
        \be:=\min \big\{1,\tfrac{k+\al}p\big\}    
    \]
    on $I$. 
    Moreover,
    if the coefficients form a bounded subset of  
    $C^{k,\al}(\ol I)$,
    then the continuous roots form a bounded subset of $C^{0,\be}(\ol J)$, for each relatively compact open 
    subinterval $J \Subset I$,
    as long as the polynomial is hyperbolic with the multiplicity of the roots uniformly not exceeding $p$, i.e., 
    $B$ defined in Remark \ref{rem:multiplicityhyp} is finite.
\end{theorem}

The main motivation of Bronshtein for proving his theorem was to show Gevrey well-posedness of the 
Cauchy problem with multiple characteristics; see \cite{Bronshtein80}. 
To this end, he deduces 
from the Lipschitz continuity of the roots of $P_a(t)$ the bound 
\[
    \Big|\frac{\p_t P_a(t)(z)}{P_a(t)(z)}\Big| \lesssim |\Im (z)|^{-1}, \quad 0< |\Im (z)| \le 1.
\]
In the setting of Theorem \ref{thm:BronshteinIII}, Tarama shows in \cite{Tarama06} 
that 
\begin{equation} \label{eq:Tarama}
    \Big|\frac{\p_t^j P_a(t)(z)}{P_a(t)(z)}\Big| \lesssim |\Im (z)|^{-j \max\{1,1/\be\}}, 
\end{equation}
for $1 \le j \le \min\{p-1,k\}$, $0 < |\Im(z)|\le 1$, and $t$ in a compact subinterval of $I$.
It is also proved in \cite{Tarama06} 
that the statements of Theorem \ref{thm:BronshteinIII} and \eqref{eq:Tarama} are equivalent.

\begin{remark}  \label{rem:orig-proof}
    Bronshtein's original proof of Theorem \ref{thm:BronshteinI} rests on a polynomial equation for the derivatives of the roots of $P_a$, see \eqref{eq:Bmethod3} below.
    Suppose that $\la_0$ is a $q$-fold root of $P_a(t_0)$ and consider 
    the polynomial
    \begin{equation} \label{eq:Bmethod1}
        P_a(t)(Z+\la_0) = A_0(t) Z^{q} + A_1(t) Z^{q-1} + \cdots + A_q(t) + R(t,Z), 
    \end{equation}
    where $R(t,Z)$ is a polynomial in $Z$ divisible by $Z^{q+1}$ (possibly identically zero).
    Note that 
    $A_j(t) = \frac{1}{(q-j)!}\p_Z^{q-j} P_a(t)(\la_0)$ and
    set $a_j^i := \p_t^i A_j(t_0)$. Then $a_0^0 \ne 0$, while $a_j^0 = 0$ for $1 \le j \le q$, since $\la_0$ is a $q$-fold 
    root of $P_a(t_0)$. Using that $P_a$ is hyperbolic 
    and that hyperbolicity is preserved by differentiation with respect to $Z$, Bronshtein shows that $a_j^i = 0$ for $0 \le i \le j-1$ and $1 \le j \le q$,
    or equivalently, that $A_j(t) = (t-t_0)^j \hat A_j(t)$, where $\hat A_j$ is a continuous function.

    Consequently, the change of variables $Z = (t-t_0) W$ in \eqref{eq:Bmethod1} 
    leads to
    \begin{align} \label{eq:Bmethod2}
        (t-t_0)^{-q}P_a(t)((t-t_0)W+\la_0) &= a_0^0\, W^q + \tfrac{1}{1!} a_1^1\, W^{q-1} + \cdots + \tfrac{1}{q!} a_q^q 
        \\
                                           &\quad +  \sum_{k=0}^d B_k(t) W^{d-k}, \notag
    \end{align}
    where $B_k$ are continuous functions with $B_k(t_0)=0$.
    It follows that $P_a$ has $q$ roots of the form $\la_j(t) = \la_0 + (t-t_0)\mu_j(t)$, where the $\mu_j$ are the roots of \eqref{eq:Bmethod2} 
    which are continuous at $t_0$. That means that all the roots of $P_a$ having the value $\la_0$ at $t_0$ 
    are differentiable at $t_0$ and the derivatives are the solutions of 
    \begin{equation} \label{eq:Bmethod3}
        a_0^0\, W^q + \tfrac{1}{1!} a_1^1\, W^{q-1} + \cdots + \tfrac{1}{q!} a_q^q =0.
    \end{equation}

    Now, for contradiction, Bronshtein 
    assumes that there is a sequence $(t_n,y_n) \to (t_0,y_0)$ such that $\la(t_n,y_n) \to \la(t_0,y_0)$ and $|\p_t \la(t_n,y_n)| \to \infty$,
    and that the multiplicity of $\la(t_n,y_n)$ as a root of $P_a(t_n,y_n)$  is $q$ for all $n\ge 1$.
    Then $\p_t \la(t_n,y_n)$ satisfies an equation of type \eqref{eq:Bmethod3} which is found as above, where $\la(t_n,y_n)$ plays the 
    role of $\la_0$; in particular, the coefficients depend on $n$. 
    It remains to show that the coefficients of this equation, after division by  
    the coefficient of $W^q$, are bounded in $n$. This is the hard part of Bronshtein's proof (see \cite[Lemma 4 and 4']{Bronshtein79}).
\end{remark}    

\subsection{Towards a proof of Bronshtein's theorem} \label{sec:Bronshteinproof}
We will give a simple proof of Theorem \ref{thm:Bronshtein} which is based on \cite{ParusinskiRainerHyp}.
It requires some preparation.

\subsubsection{Splitting} \label{sec:splittinghyp} \index{splitting}

We proceed as in Section \ref{sec:splittingRellich}, but replace the analytic function $\th$ by $|\tilde a_2|^{1/2}$.
Let $P_{\tilde a} \in \Hyp_T(d)$ be such that $\tilde a \ne 0$.
Then the polynomial
\[
    Q_{\ul a}(Z):= |\tilde a_2|^{-d/2} P_{\tilde a}(|\tilde a_2|^{1/2} Z) = 
    Z^d - Z^{d-2} + \sum_{j=3}^d |\tilde a_2|^{-j/2} \tilde a_j Z^{d-j}
\]    
belongs to $\Hyp_T^0(d) = \{P_{\tilde a} \in \Hyp_T(d) : \tilde a_2 =-1\}$.
By Lemma \ref{lem:splitting}, 
we have 
\[
    Q_{\ul a} = Q_{\ul b} Q_{\ul c},
\]
on some open ball $B(P_{\tilde a},r) \subseteq \R^d$ 
such that $\deg Q_{\ul b} <d$, $\deg Q_{\ul c} <d$, and
\[
    \ul b_i = \ps_i(|\tilde a_2|^{-3/2}\tilde a_3,\ldots,|\tilde a_2|^{-d/2}\tilde a_d), \quad i = 1,\dots, \deg Q_{\ul b},
\]
where $\ps_i$ are real analytic functions; likewise for $\ul c_i$. 
If $Q_{\ul a}$ is hyperbolic, then also $Q_{\ul b}$ and $Q_{\ul c}$ are hyperbolic; 
we restrict our attention to the set $B(P_{\tilde a},r) \cap \Hyp_T(d)$.
This induces a splitting
\[
    P_{\tilde a} = P_b P_c, \quad \text{ on } B(P_{\tilde a},r), 
\]
where 
\begin{equation} \label{eq:formulasbhyp}
    b_i = |\tilde a_2|^{i/2} \ps_i(|\tilde a_2|^{-3/2}\tilde a_3,\ldots,|\tilde a_2|^{-d/2}\tilde a_d), \quad i = 1,\dots, \deg P_{b}.
\end{equation}
The coefficients $\tilde b_i$ of $P_{\tilde b}$ resulting from $P_b$ by the Tschirnhausen transformation 
have an analogous representation (in view of \eqref{eq:Tschirn})
\begin{equation} \label{eq:formulasbtihyp}
    \tilde b_i = |\tilde a_2|^{i/2} \tilde \ps_i(|\tilde a_2|^{-3/2}\tilde a_3,\ldots,|\tilde a_2|^{-d/2}\tilde a_d), \quad i = 1,\dots, \deg P_{b}.
\end{equation} 
Shrinking $r>0$ slightly, we may assume that all partial derivatives of $\ps_i$ and $\tilde \ps_i$ of all orders are bounded on $B(P_{\tilde a},r)$.

\begin{lemma} \label{lem:b2a2hyp}
    In this situation we have $|\tilde b_2| \le 4\, |\tilde a_2|$.  
\end{lemma}

\begin{proof}
    Let $\la_1,\ldots,\la_d$ be the roots of $P_{\tilde a}$ such that the first $k$ of them, 
    i.e., $\la_1,\ldots,\la_k$, are the roots of $P_b$.
    Then $|b_1| \le \sum_{j=1}^k |\la_j| \le \sqrt k\, \big(\sum_{j=1}^k \la_j^2\big)^{1/2}$ and thus
    \begin{align*}
        2 \, |\tilde b_2| &= \sum_{j=1}^k \Big(\la_j + \frac{b_1}{k}\Big)^2 =  \sum_{j=1}^k \Big(\la_j^2 +  \frac{b_1^2}{k^2} + \frac{2}k \la_j b_1\Big)
        \\
                          &= \sum_{j=1}^k \la_j^2 + \frac{b_1^2}{k^2}\cdot k + \frac{2}k b_1 \sum_{j=1}^k \la_j \le (1 + 1 + 2) \sum_{j=1}^k \la_j^2 \le 8\, |\tilde a_2|,
    \end{align*}
    since $\sum_{j=1}^d \la_j^2 = -2 \tilde a_2$.
\end{proof}

\subsubsection{Glaeser's inequality}\index{Glaeser's inequality}
A classical inequality used by Glaeser in \cite{Glaeser63R} (and attributed by him to Malgrange) is the following: for nonnegative 
$C^1$ functions $f$ on $\R$ with $f'' \in L^\infty(\R)$ we have
\begin{equation} \label{eq:Glaeser}
    f'(t)^2 \le 2 f(t) \|f''\|_{L^\infty(\R)}, \quad t \in \R.
\end{equation}
Note that this shows Bronshtein's theorem in the simplest nontrivial case.
We shall need a local version.
For $t_0 \in \R$ and $r>0$, let $I(t_0,r)$ denote the open interval centered at $t_0$ with radius $r$, 
\[
    I(t_0,r) := \{t \in \R : |t-t_0|<r\}.
\]

\begin{lemma} \label{lem:Glaeser}
    Let $I\subseteq \R$ be an open bounded interval.
    Let $f \in C^{1,1}(\ol I)$ satisfy $f \ge 0$ or $f \le 0$ on $I$.
    Let $M>0$ and assume that $t_0 \in I$, $f(t_0)\ne 0$, and $I_0 := I(t_0,M^{-1} |f(t_0)|^{1/2}) \subseteq I$.    
    Then 
    \[
        |f'(t_0)| \le (M +M^{-1} \on{Lip}_{I_0}(f')) |f(t_0)|^{1/2}.
    \]
    If additionally $\on{Lip}_{I_0}(f') \le M^2$, then  
    \[
        |f'(t_0)| \le 2M\, |f(t_0)|^{1/2}.
    \]
    Note that if $f(t_0)=0$ also $f'(t_0)=0$.
\end{lemma}

\begin{proof}
    Suppose that $f\ge 0$; otherwise consider $-f$.
    Thus $f(t_0)>0$ and 
    \[
        0 \le f(t_0 + h) = f(t_0) + f'(t_0)h + \int_0^1 f'(t_0+hs) -f'(t_0) \, ds\cdot h. 
    \]
    The assertion follows from setting $h:= \pm M^{-1}|f(t_0)|^{1/2}$.
\end{proof}

\subsubsection{Interpolation estimates}

\begin{lemma} \label{lem:Vandermonde}
    Let $T(x) = a_0 + a_1 x + \cdots + a_m x^m \in \C[x]$ satisfy $|T(x)| \le A$ for $x \in [0,B] \subseteq \R$.
    Then 
    \[
        |a_j| \le C(m)\, A B^{-j}, \quad 0 \le j \le m.
    \]
\end{lemma}

\begin{proof}
    It suffices to assume $A=B=1$; for the general case consider $A^{-1}T(By)$.
    Noting that $(T(x_0),\ldots ,  T(x_m))^t = V (a_0,\ldots,a_m)^t$ and hence 
    $(a_0,\ldots,a_m)^t = V^{-1} (T(x_0),\ldots T(x_m))^t$, where $V$ is the Vandermonde matrix of the equidistant points
    $0 = x_0 < x_1 < \cdots < x_m=1$, the statement follows easily. 
\end{proof}

\begin{lemma} \label{lem:interpolation}
    Let $f \in C^{m,1}(\ol I)$, where $I \subseteq \R$ is a bounded open interval.
    Then 
    \[
        |f^{(k)}(t)| \le C(m)\, |I|^{-k} \big( \|f\|_{L^\infty(I)} + \on{Lip}_I(f^{(m)}) |I|^{m+1}\big), 
        \quad t \in I,\, 1 \le k \le m.
    \]
\end{lemma}

\begin{proof}
    Let $t \in I$. Then $[t,t+|I|/2)$ or $(t-|I|/2,t]$ is contained in $I$.
    Let $s$ be a point in the respective interval. 
    By Taylor's formula,
    \begin{align*}
        \MoveEqLeft \Big| \sum_{k=0}^{m} \frac{f^{(k)}(t)}{k!} (s-t)^k\Big| 
        \\
        &= \Big| f(s) - (s-t)^{m} \int_0^1 \frac{(1-u)^{m-1}}{(m-1)!} (f^{(m)}(t+u(s-t)) - f^{(m)}(t)) \, du\Big|
        \\
        &\le \|f\|_{L^\infty(I)} + |I|^{m+1} \Lip_{I}(f^{(m)}). 
    \end{align*}
    Now apply Lemma \ref{lem:Vandermonde}.
\end{proof}

\subsubsection{The key argument}

\begin{definition}
    Let $I_1 \subseteq \R$ be an open bounded interval and $I_0 \Subset I_1$ a relatively compact open subinterval.
    Let $P_{\tilde a}(t)$, $t \in I_1$, be a monic hyperbolic polynomial of degree $d$ in Tschirnhausen form 
    with coefficients $\tilde a_j \in C^{d -1,1}(\ol I_1)$, $j=2,\ldots,d$.
    Let $A>0$ be a constant.
    We say that $(P_{\tilde a},I_1,I_0,A)$ is \emph{admissible} if, for 
    every $t_0 \in I_0 \setminus \{t : \tilde a_2(t) = 0\}$, 
    \begin{align}
        &I_A(t_0) := I(t_0,A^{-1}|\tilde a_2(t_0)|^{1/2}) \subseteq I_1, \label{eq:A1}  
        \\
        &\frac{1}{2} \le \frac{\tilde a_2(t)}{\tilde a_2(t_0)} \le 2,\quad t \in I_A(t_0),  \label{eq:A2}
        \\
           &|\tilde a_j^{(k)}(t)| \le  A^k \, |\tilde a_2(t_0)|^{(j-k)/2},\quad 2\le j \le d,\, 1 \le k \le d-1,\,   t \in I_A(t_0), \label{eq:A3}
           \\
           &\on{Lip}_{I_A(t_0)} (\tilde a_j^{(d-1)}) \le A^d \, |\tilde a_2(t_0)|^{(j-d)/2}, \quad 2 \le j \le d.    \label{eq:A4}
    \end{align}
\end{definition}

The following lemma is an easy exercise.

\begin{lemma} \label{lem:Aconsequence}
    Let $(P_{\tilde a},I_1,I_0,A)$ be admissible. Then the functions $\ul a_j := |\tilde a_2|^{-j/2} \tilde a_j$, $2 \le j \le d$, are well-defined on $I_A(t_0)$ and satisfy
    \begin{align}
          &|\ul a_j^{(k)}(t)| \le C(d) A^k \, |\tilde a_2(t_0)|^{-k/2},\quad 2\le j \le d,\, 1 \le k \le d-1,\,   t \in I_A(t_0), \label{eq:A5}
          \\
          &\on{Lip}_{I_A(t_0)} (\ul a_j^{(d-1)}) \le C(d)  A^d \, |\tilde a_2(t_0)|^{-d/2}, \quad 2 \le j \le d.    \label{eq:A6}
    \end{align}
\end{lemma}

There is some redundancy in the conditions \eqref{eq:A1}--\eqref{eq:A4}, up to multiplying $A$ by some constant $C(d)\ge 1$:

\begin{lemma} \label{lem:enough}
    Suppose that \eqref{eq:A1}, \eqref{eq:A3} for $k\ge j$, $2 \le j \le d$, and \eqref{eq:A4} hold.
    Then there is a constant $C(d)\ge 1$ such that $(P_{\tilde a},I_1,I_0,C(d)A)$ is admissible.
\end{lemma}

\begin{proof}
    By \eqref{eq:A3} for $j=k=2$, $\on{Lip}_{I_A(t_0)}(\tilde a_2') \le  A^2$.
    Lemma \ref{lem:Glaeser} and \eqref{eq:A1} imply
    \[
        |\tilde a_2'(t_0)| \le 2  A \, |\tilde a_2(t_0)|^{1/2}.
    \]
    Thus, for $t \in I_{6  A}(t_0)$,
    \begin{align*}
        |\tilde a_2(t) - \tilde a_2(t_0)| &\le |\tilde a_2'(t_0)| |t-t_0| + \on{Lip}_{I_A(t_0)}(\tilde a_2')\, |t-t_0|^2
        \\
                                          &\le  \frac{1}{3}\, |\tilde a_2(t_0)| + \frac{1}{36} \, |\tilde a_2(t_0)| < \frac{1}{2}\, |\tilde a_2(t_0)|,
    \end{align*}
    which gives \eqref{eq:A2}. That \eqref{eq:A3} also holds for $k<j$ follows from Lemma \ref{lem:interpolation} applied to $f=\tilde a_j$ and $m=j-1$, 
    together with \eqref{eq:A2} and Lemma \ref{lem:dominant}. 
\end{proof}

\begin{proposition} \label{prop:adm}
    Let $(P_{\tilde a},I_1,I_0,A)$ be admissible and $t_0 \in I_0 \setminus \{t : \tilde a_2(t)=0\}$.
    Then there exist a constant $C(d) > 1$ and 
    open bounded intervals $J_1 \Supset J_0 \ni t_0$, $J_0$ relatively compact in $J_1$, such that the following holds. 
    We have 
    \[
        P_{\tilde a} = P_b P_c, \quad \text{ on } J_1,
    \]
    where $P_b$ and $P_c$ are monic hyperbolic polynomials of degree $<d$ and with coefficients in $C^{d-1,1}(\ol J_1)$,
    and, after Tschirnhausen transformation, 
    $(P_{\tilde b},J_1,J_0,C(d)A)$  and $(P_{\tilde c},J_1,J_0,C(d)A)$ are admissible. 
\end{proposition}

\begin{proof}
    Consider 
    \[
        \ul a := (-1,\ul a_3,\ldots, \ul a_d) : I_A(t_0) \to \R^{d-1},
    \]
    which defines a continuous bounded curve, by Lemma \ref{lem:dominant}.
    By Lemma \ref{lem:Aconsequence},  there exists a constant $C_1 = C_1(d) > 1$ such that 
    \begin{equation} \label{eq:lengthula}
        |\ul a'(t)| \le  C_1 A\, |\tilde a_2(t_0)|^{-1/2}, \quad t \in I_A(t_0).
    \end{equation}
    Choose a finite cover of $\Hyp_T^0(d)$ by balls $B_1,\ldots,B_s$ such that on 
    each $B_i$ we have a splitting of $P_{\tilde a}$ (cf.\ Section \ref{sec:splittinghyp}).
    There exists $r_1 \in (0,1)$ such that for any $p \in \Hyp_T^0(d)$ there is $i \in \{1,\ldots,s\}$ 
    such that $B(p,r_1) \subseteq B_i$ (note that $2r_1$ is a Lebesgue number of the cover $B_1,\ldots,B_s$).
    Set 
    \begin{equation*}
        J_1 := I_{C_1 A/r_1}(t_0).
    \end{equation*}
    Then $J_1 \subseteq \ul a^{-1}(B(\ul a(t_0),r_1))$ 
    so that 
    \[
        P_{\tilde a} = P_b P_c, \quad \text{ on } J_1.
    \]
    Fix $r_0 < r_1$ and set $J_0 := I_{C_1 A/r_0}(t_0)$. 
    The coefficients of $P_{\tilde b}$ (after Tschirnhausen transformation) 
    are given by the formulas \eqref{eq:formulasbtihyp}.

    Let us show that $(P_{\tilde b},J_1,J_0,B)$ is admissible, where $B = C(d)A$ for a suitable constant $C(d)$.
    If $B$ is a constant satisfying 
    \[
        B\ge \frac{2 \sqrt 2\, C_1 A}{r_1-r_0},
    \]
    then, for each $t_1 \in J_0 \setminus \{t : \tilde b_2(t) = 0\}$, 
    \[
        B^{-1} |\tilde b_2(t_1)|^{1/2} \le \frac{r_1-r_0}{C_1 A}\, |\tilde a_2(t_0)|^{1/2},
    \]
    by Lemma \ref{lem:b2a2hyp} and \eqref{eq:A2}. 
    This implies 
    \[
        J_B(t_1) := I(t_1,B^{-1} |\tilde b_2(t_1)|^{1/2}) \subseteq J_1,
    \]
    by the definition of $J_1$ and $J_0$.
    Note that $r_1$ and $r_0$ can be chosen in a universal way so that 
    we subsume the dependence of $B$ on them in $C(d)$.

    Next we claim that, on $J_1$,
    \begin{align*}
        |\p_t^k [\tilde \ps_i \o (\ul a_3,\ldots,\ul a_d)]| &\le C(d) A^k \, |\tilde a_2(t_0)|^{-k/2}, 
        \quad 0 \le k \le d-1,
        \\
        \on{Lip}_{J_1}(\p_t^{d-1} [\tilde \ps_i \o (\ul a_3,\ldots,\ul a_d)]) &\le C(d) A^d \, |\tilde a_2(t_0)|^{-d/2}. 
    \end{align*}
    Recall that the real analytic functions $\tilde \ps_i$ and its partial derivatives of all orders 
    are bounded  by universal constants. Thus the claim is obvious for $k=0$.
    We have
    \[
        \p_t [\tilde\ps_i \o (\ul a_3,\ldots,\ul a_d)] = d\tilde \ps_i(\ul a)(\ul a') 
    \]
    so that the claim for $k=1$ follows from \eqref{eq:lengthula}. 
    For $2 \le k \le d-1$, the claim follows from differentiating this equation and using \eqref{eq:A5}.
    In a similar way, using also \eqref{eq:A6}, one gets the estimate for $\on{Lip}_{J_1}(\p_t^{d-1} [\tilde \ps_i \o (\ul a_3,\ldots,\ul a_d)])$. 

    Now it is easy to conclude (from the formulas \eqref{eq:formulasbtihyp})
    \begin{align*}
        |\tilde b_i^{(k)}(t)| &\le C(d) A^k \, |\tilde a_2(t_0)|^{(i-k)/2}, 
        \quad t \in J_1,\, 1 \le k \le d-1,
        \\
        \on{Lip}_{J_1}(\tilde b_i^{(d-1)})  &\le C(d) A^d \, |\tilde a_2(t_0)|^{(i-d)/2}, 
    \end{align*}
    for all $2 \le i \le \deg P_b$. By Lemma \ref{lem:b2a2hyp} and \eqref{eq:A2}, 
    $|\tilde a_2(t_0)|^{-1} \le 2\,  |\tilde a_2(t_1)|^{-1} \le 8\, |\tilde b_2(t_1)|^{-1}$
    and thus
    \begin{align*}
        |\tilde b_i^{(k)}(t)| &\le C(d) A^k \, |\tilde b_2(t_1)|^{(i-k)/2}, 
        \quad t \in J_1,\, i \le k \le d-1,
        \\
        \on{Lip}_{J_1}(\tilde b_i^{(d-1)})  &\le C(d) A^d \, |\tilde b_2(t_1)|^{(i-d)/2},
    \end{align*}
    for all $2 \le i \le \deg P_b$. Thus we may conclude that $(P_{\tilde b},J_1,J_0,B)$ is admissible 
    from Lemma \ref{lem:enough}.
\end{proof}

\begin{remark} \label{rem:adm}
    We have the same estimates for $b_i$ instead of $\tilde b_i$. 
\end{remark}

\begin{theorem} \label{thm:adm}
    Let $(P_{\tilde a},I_1,I_0,A)$ be admissible.
    Let $\la : I_1 \to \R$ be a continuous root of $P_{\tilde a}$. 
    Then $\la$ is Lipschitz on $I_0$ with $\Lip_{I_0}(\la) \le C(d) A$.
\end{theorem}

\begin{proof}
    We proceed by induction on the degree.
    The only monic polynomial of degree $1$ in Tschirnhausen form is $Z$; so there is nothing to prove.
    So assume that $\deg P_{\tilde a} \ge 2$.
    Fix $t_0 \in I_0 \setminus \{t : \tilde a_2(t)=0\}$. 
    By Proposition \ref{prop:adm}, we may assume that on a neighborhood $J_1$ of $t_0$ 
    we have 
    \[
        \la(t) = - \frac{b_1(t)}{\deg P_b} + \mu(t),
    \]
    where $\mu$ is a continuous root of $P_{\tilde b}$ and  $(P_{\tilde b},J_1,J_0,C(d)A)$ is admissible.
    By the induction hypothesis, $\mu$ is Lipschitz on $J_0$ with $\Lip_{J_0}(\mu) \le C(d)A$. 
    Thus $\la$ is Lipschitz on $J_0$ with $\Lip_{J_0}(\la) \le C(d)A$, in view of Remark \ref{rem:adm}. 
    It is an easy exercise to show that these local uniform Lipschitz bounds on $I_0 \setminus \{t : \tilde a_2(t) =0\}$ 
    imply that $\Lip_{I_0}(\la) \le C(d)A$, since $\la$ vanishes on the set $\{t : \tilde a_2(t) =0\}$.
\end{proof}

\subsubsection{Proof of Bronshtein's theorem}

Before we start with the proof, 
we observe in the following lemma 
that the assumptions of Theorem \ref{thm:Bronshtein} (after Tschirnhausen transformation) yield an admissible quadruple.

\begin{lemma} \label{lem:ass>adm}
    Let $I_1 \subseteq \R$ be a bounded open interval and $P_{\tilde a}$ a monic hyperbolic polynomial of degree $d$ in Tschirnhausen form 
    with coefficients $\tilde a_j \in C^{d-1,1}(\ol I_1)$, $2 \le j \le d$. 
    If $I_0 \Subset I_1$ is an open subinterval, relatively compact in $I_1$,
    then $(P_{\tilde a}, I_1,I_0,A)$ is admissible with 
    \begin{align}
        A&:= 6 \max\{A_1,A_2\}, \label{eq:defA} 
        \intertext{where} 
        A_1 &:= \max \big\{ \de^{-1}\|\tilde a_2\|_{L^\infty(I_1)}^{1/2}, (\on{Lip}_{I_1} (\tilde a_2'))^{1/2}\big\}, \nonumber
        \\
        A_2 &:= \max_{2 \le j \le d} \big\{ \Lip_{I_1}(\tilde a_j^{(d-1)}) \, \|\tilde a_2\|_{L^\infty(I_1)}^{(d-j)/2}  \big\}^{1/d}, \nonumber
    \end{align}
    and $\de := \on{dist}(I_0, \R \setminus I_1)$.
\end{lemma}

\begin{proof}
    Let $t_0 \in I_0 \setminus \{t : \tilde a_2(t) = 0\}$. 
    We have $I_{A_1}(t_0) \subseteq I_1$, i.e., \eqref{eq:A1}. 
    By Lemma \ref{lem:Glaeser}, 
    \[
        |\tilde a_2'(t_0)| \le 2 A_1\, |\tilde a_2(t_0)|^{1/2}.
    \]
    This implies (as in the proof of Lemma \ref{lem:enough}) that \eqref{eq:A2} holds for $t \in I_{6 A_1}(t_0)$.
    Now \eqref{eq:A4} is clear from the definition of $A_2$. 
    Finally, Lemma \ref{lem:dominant} and Lemma \ref{lem:interpolation} imply \eqref{eq:A3}, for $t \in I_A(t_0)$.
\end{proof}

Now let $P_{\tilde a}$ be a monic hyperbolic polynomial of degree $d$ in Tschirnhausen form 
with coefficients $\tilde a_j \in C^{d-1,1}(I)$, $2 \le j \le d$.
Fix relatively compact open subintervals $I_0 \Subset I_1 \Subset I$.
By Lemma \ref{lem:ass>adm} and Theorem \ref{thm:adm}, any continuous root $\la$ of $P_{\tilde a}$
is Lipschitz on $I_0$ with $\Lip_{I_0}(\la) \le C(d)A$, 
where $A$ is defined in \eqref{eq:defA}.
This implies
\begin{align*}
    \Lip_{I_0}(\la) &\le C(d)\, \max \{\de ^{-1}, 1\} \, \max_{2 \le j \le d} \|\tilde a_j\|_{C^{d-1,1}(\ol I_1)}^{1/j}.
\end{align*}

If $P_a$ is not necessarily in Tschirnhausen form, we apply the Tschirnhausen transformation $P_a \leadsto P_{\tilde a}$. 
Then $\tilde a_j$ is a weighted homogeneous polynomial of degree $j$ in $a_1,\ldots, a_d$, where $a_j$ has the weight $j$; see \eqref{eq:Tschirn}.
The roots are shifted by $a_1/d$.
Thus any continuous root $\la$ of $P_a$ is Lipschitz on $I_0$ and 
\begin{align*}
    \Lip_{I_0}(\la) &\le C(d)\, \max \{\de ^{-1}, 1\}\, \max_{1 \le j \le d} \|a_j\|_{C^{d-1,1}(\ol I_1)}^{1/j}.
\end{align*}
This ends the proof of Theorem \ref{thm:Bronshtein}. \qed

\subsection{Sufficient conditions for $C^p$ roots}
\label{sec:sufficienthyp}

Let us analyze what causes the loss of regularity and give several sufficient conditions for better regularity of the roots.

\subsubsection{The effect of positive local minima}

A finer analysis of Example \ref{ex:optimalhyp}(1) 
indicates that the small positive local minima of the function $f$ 
prevent the solutions of $Z^2 =f$ from being $C^{1,\al}$.
This is confirmed by the following result of Bony, Broglia, Colombini, and Pernazza \cite{BBCP06}.

\begin{theorem}[{\cite[Theorem 3.5]{BBCP06}}]
    Let $f: \R \to [0,\infty)$ be of class $C^4$.
    Then $Z^2 = f$ has a $C^2$ solution if and only if 
    there exists a continuous function $\ga$ vanishing on $\{t \in \R : f^{(j)}(t) = 0, \, 0 \le j \le 4\}$ 
    such that for each local minimum $t_0$ of $f$ with $f(t_0)>0$ 
    we have $f''(t_0) \le \ga(t_0) f(t_0)^{1/2}$.
\end{theorem}

Bony, Colombini, and Pernazza \cite{BonyColombiniPernazza10} extended this result.
Forcing $f$ and sufficiently many of its derivatives to vanish on all local minima of $f$, 
turns out to be sufficient for the existence of $C^p$ square roots:

\begin{theorem}[{\cite{BonyColombiniPernazza10}}]
    Let $f: \R \to [0,\infty)$ be of class $C^{2p}$, $p\ge2$, and 
    assume that $f$ vanishes along with its derivatives up to order $2p-4$ at all its local minima.
    Then $Z^2 = f$ has a $C^p$ solution. 
    A solution with a derivative of order $p+1$ everywhere exists, provided that,
    under the same assumption on the minima, $f$ is $C^{2p+2}$. 
\end{theorem}

\subsubsection{A regularity class for taking radicals}

A different approach is due to Ray and Schmidt-Hieber \cite{RaySchmidt-Hieber17}. 
They define a regularity class $\cF^\be$, $\be>0$, \index{F@$\cF^\be$}
for nonnegative functions $f : [0,1] \to [0,\infty)$ which behaves nicely  
with respect to radicals.
For $\be>0$ we let $C^\be$ be a short notation for the H\"older class $C^{m,\al}$, where $m$ is the largest integer strictly smaller than $\be$ and $\al= \be -m$.
For $f \in C^\be([0,1])$ set 
\[
    |f|_{\cF^\be([0,1])} := \max_{1 \le j < \be} \Big( \sup_{t \in [0,1]} \frac{|f^{(j)}(t)|^\be}{|f(t)|^{\be -j}}\Big)^{1/j}.  
\]
It measures the flatness of $f$ near its zeros. In particular, if $|f|_{\cF^\be([0,1])} < \infty$ then that $f$ 
vanishes at some point entails that also all its derivatives up to order $<\be$ vanish.
Consider 
\[
    \cF^\be([0,1]) := \{ f \in C^\be([0,1]) : f \ge 0,\, \|f\|_{\cF^\be([0,1])}< \infty\},
\]
where 
\[
    \|f\|_{\cF^\be([0,1])} := \|f\|_{C^\be([0,1])} + |f|_{\cF^\be([0,1])}.
\]
Note that $f \in \cF^\be([0,1])$ may possess infinitely many nonzero local mimima.

\begin{theorem}[{\cite{RaySchmidt-Hieber17}}]
    Let $\al \in (0,1]$ and $\be >0$. For all $f \in \cF^\be([0,1])$, we have 
    \begin{equation*}
        \|f^\al\|_{\cF^{\al\be}([0,1])} \le C(\al,\be)\, \|f\|^\al_{\cF^\be([0,1])}.
    \end{equation*}
    In particular, $f^\al \in C^{\al\be}([0,1])$.
\end{theorem}

Moreover, 
bounds on the wavelet coefficients of $f^\al$ are derived in \cite{RaySchmidt-Hieber17} 
which give additional information on the local regularity of $f^\al$.

\subsubsection{Finite order of contact} \label{sec:nonflathyp}

In the case of a curve of monic hyperbolic polynomials $P_a(t)$, $t \in I$, of degree $d$, 
sufficient coefficients for the existence of $C^p$ roots can be given in terms 
of the differentiability of the coefficients and the finite order of contact of the roots; 
see Rainer \cite{RainerOmin,RainerFin}.

In order to present these results we need some terminology.
The polynomial
\begin{equation*} \label{eq:Des}
    \De_s(X_1,\ldots,X_d) 
    := \sum_{i_1 < \ldots < i_s} (X_{i_1} - X_{i_2})^2 \cdots (X_{i_1} - X_{i_s})^2 \cdots (X_{i_{s-1}} - X_{i_s})^2
\end{equation*}
is symmetric and so it is a (unique) polynomial in the elementary symmetric functions:
\[
    \De_s = \tilde \De_s(\si_1,\ldots,\si_d).
\]
Thus $t \mapsto \tilde \De_s(P_a(t))$, where we identify $P_a$ with its vector of coefficients,
is well-defined. 
We say that $P_a(t)$, $t \in I$, is \emph{normally nonflat}\index{normally nonflat} if 
for each $t_0 \in I$ the following condition is satisfied.
Let $s$ be the maximal integer such that the germ at $t_0$ of $t \mapsto \tilde \De_s(P_a(t))$ 
is not zero. Then  
\[
    m_{t_0}(\tilde \De_s(P_a)):= \sup \{m \in \N : (t-t_0)^{-m}\tilde \De_s(P_a(t)) \text{ is continuous near }t_0\} <\infty. 
\]
We call this quantity the \emph{multiplicity of $\tilde \De_s(P_a)$ at $t_0$};\index{multiplicity}  
obviously we can analogously define the multiplicity of any continuous univariate function.  
It is not hard to see that $P_a(t)$, $t \in I$, is normally nonflat 
if and only if for any continuous system $\la_j$, $1 \le j \le d$, of its roots 
$m_{t_0}(\la_i-\la_j) = \infty$ implies $\la_i = \la_j$ near $t_0$.

Let $P_a(t)$, $t \in I$, be a normally nonflat curve of monic hyperbolic polynomials.
In \cite{RainerOmin,RainerFin}, numbers $\ga(P_a), \Ga(P_a) \in \N \cup \{\infty\}$, $\ga(P_a) \le  \Ga(P_a)$, 
are defined, in terms of a splitting algorithm for $P_a$, 
that encode the conditions for $C^p$ solvability. 
We refer the reader to \cite{RainerOmin} and \cite{RainerFin} for the details of the definition.

\begin{theorem}[{\cite{RainerOmin,RainerFin}}] \label{thm:contacthyp}
    Let $P_a(t)$, $t \in I$, be a normally nonflat curve of monic hyperbolic polynomials of degree $d$.
    If $p \in \N \cup \{\infty\}$ and $P_a(t)$, $t \in I$, has $C^{p+ \Ga(P_a)}$ coefficients, 
    then it admits a $C^{p+\ga(P_a)}$ system of its roots.
\end{theorem}

Note that for $p=\infty$ the theorem yields the result of Alekseevsky, Kriegl, Losik, and Michor \cite{AKLM98}
that, under the normal nonflatness assumption, $C^\infty$ systems of the roots exist if the coefficients are $C^\infty$.

\subsubsection{Definability: no oscillation}

Let us state one more result. It shows that, actually, it is not the infinite contact 
between the roots that is to blame 
for the loss of regularity, but oscillation. Of course, smooth coefficients can admit infinite oscillation 
only in conjuction with infinite flatness. 
The result is formulated in the framework of o-minimal expansions $\sS$ of the real field.\index{o-minimal} 
That means that $\sS$ is a family $\sS=(\sS_n)_{n \ge 1}$, where $\sS_n$ is a collection of subsets of $\R^n$ such that
\begin{itemize}
    \item $\sS_n$ is a boolean algebra with respect to the usual set-theoretic operations,
    \item $\sS_n$ contains all semialgebraic subsets of $\R^n$,
    \item $\sS$ is stable by cartesian products and linear projections,
    \item each $S \in \sS_1$ has only finitely many connected components.
\end{itemize}
The \emph{$\sS$-definable sets}\index{definable} are those sets $S$ such that there exists $n\ge 1$ with $S \in \sS_n$. 
A map is $\sS$-definable if its graph is $\sS$-definable.

\begin{theorem}[{\cite[Theorem 4.12]{RainerOmin}}]
    Let $\sS$ be an arbitrary o-minimal expansion of the real field.
    Let $P_a(t)$, $t \in I$, be a curve of monic hyperbolic polynomials 
    with $C^\infty$-coefficients that are $\sS$-definable. 
    Then there exists a $C^\infty$-system of the roots of $P_a$ consisting of 
    $\sS$-definable functions.
\end{theorem}

By Miller's dichotomy theorem \cite{Miller:1994ue}, for a fixed o-minimal expansion $\sS$,
either for every $\sS$-definable function $f : \R \to \R$ there is $N \in \N$ such that $f(t) = O(t^N)$ as $t \to \infty$,
or the global exponential function $\exp : \R \to \R$ is $\sS$-definable.
In the latter case, there are infinitely flat $C^\infty$ functions in $\sS$. 
On the other hand, $\sS$ cannot contain oscillating functions.

\subsection{G{\aa}rding hyperbolic polynomials} \label{sec:Garding}

There are various ways to generalize univariate hyperbolic polynomials to multivariate polynomials.
Of particular interest are the \emph{G{\aa}rding hyperbolic polynomials} and the closely related \emph{real stable polynomials}; 
they have numerous applications in PDEs, combinatorics, optimization, functional analysis, probability, etc. 

\subsubsection{G{\aa}rding hyperbolic and real stable polynomials}
A homogeneous polynomial $f(Z_1,\ldots,Z_n) \in \R[Z_1,\ldots,Z_n]$ of degree $d$ is said to be 
\emph{G{\aa}rding hyperbolic with respect to $v \in \R^n$}\index{G{\aa}rding hyperbolic polynomial}\index{polynomial!G{\aa}rding hyperbolic} 
if $f(v) \ne0$ and 
for all $x \in \R^n$ the univariate polynomial $f(x - Tv) \in \R[T]$ is hyperbolic, i.e., has all roots real.
Geometrically, this means that any 
affine line with direction $v$ meets the real hypersurface $\{x \in \R^n : f(x) = 0\}$ in $d$ points (with multiplicities).

This notion was introduced by G\r{a}rding \cite{Garding51} in the 1950s. He showed that $f$ being G{\aa}rding
hyperbolic with respect to a direction $v$ 
is a necessary and sufficient condition for local well-posedness of the Cauchy problem with principal symbol $f$ 
and initial data on a hyperplane with normal vector $v$. 
G{\aa}rding hyperbolic polynomials have found many applications ever since, for instance, in Gurvits' proof \cite{Gurvits:2008aa} of the van der Waerden conjecture. 

An important example is the determinant on the real vector space of $d \times d$ 
Hermitian matrices: it is G{\aa}rding hyperbolic with respect to the identity matrix $\I$.  

A polynomial $f(Z_1,\ldots,Z_n) \in \C[Z_1,\ldots,Z_n]$ is called \emph{stable}\index{polynomial!stable} 
if $f(z_1,\ldots,z_n) \ne 0$ for all $(z_1,\ldots,z_n) \in \C^n$ with $\Im(z_j)>0$, $1 \le j \le n$.
If all coefficients of a stable polynomial are real, it is called \emph{real stable}.\index{real stable polynomial}\index{polynomial!real stable}
Real stable polynomials played a crucial role in the recent proof of the Kadison--Singer conjecture \cite{Marcus:2015aa}.

These notions are closely related.

\begin{proposition}[{\cite[Proposition 1.1]{Borcea:2010aa}}] \label{realstablehyperbolic}
    Let $f \in \R[Z_1,\ldots,Z_n]$ be of degree $d$ and let $f_H \in \R[Z_1,\ldots,Z_n, W]$ be the unique 
    homogeneous polynomial of degree $d$ such that 
    \[
        f_H(Z_1,\ldots,Z_n,1) = f(Z_1,\ldots,Z_n).
    \]
    Then $f$ is real stable if and only if $f_H$ is G\r{a}rding hyperbolic with respect to 
    every vector $v = (v_1,\ldots,v_n,0)$ with $v_i>0$ for all $i$.
\end{proposition}

Interesting examples of real stable polynomials 
are generated as follows (see \cite[Proposition 1.12]{Borcea:2010aa}):
Let $A_1,\ldots,A_n$ be positive semidefinite $d \times d$ matrices and $B$ a (complex) Hermitian $d \times d$ matrix. Then the 
polynomial
\[
    f(Z_1,\ldots,Z_n) = \det \Big( \sum_{j=1}^n Z_j A_j  + B\Big)
\]
is either real stable or identically zero. 
Conversely, any real stable polynomial in two variables $Z_1$ and $Z_2$ can be written as $\pm \det(Z_1 A_1 + Z_2 A_2 + B)$, 
where $A_j$ are positive semidefinite and $B$ is symmetric, see \cite[Theorem 1.13]{Borcea:2010aa}. 
The latter result is based on the Lax conjecture \cite{Lax:1958aa} for G\r{a}rding hyperbolic polynomials 
which is true: as noted in 
\cite{Lewis:2005aa} 
it follows from \cite{Helton:2007aa} and \cite{Vinnikov:1993aa}. 

\begin{remark} \label{rem:Laxconjecture}\index{Lax conjecture}
    The Lax conjecture states
    that a homogeneous polynomial $f$ on $\R^3$ is G\r{a}rding hyperbolic of degree $d$ with respect to the direction $(1,0,0)$ with $f(1,0,0) = 1$
    if and only if there exist real symmetric $d \times d$ matrices $A$ and $B$ such that
    \begin{equation} \label{eq:Lax}
        f(X,Y,Z) = \det ( X \mathbb I + Y A + Z B).
    \end{equation}
    A representation of type \eqref{eq:Lax} is in general not possible for G\r{a}rding hyperbolic polynomials on $\R^n$ with $n \ge 4$.
    Indeed, the dimension of the space of G{\aa}rding hyperbolic polynomials of degree $d$ on $\R^n$ with respect to a fixed direction 
    is $\binom{n+d -1}{d}$ while the set of polynomials in $\R[X_1,\ldots,X_n]$ of the form
    \begin{equation} \label{eq:det}
        \det(X_1 \I + X_2 A_2 + \cdots + X_n A_n),
    \end{equation}
    where $A_i$ are real symmetric $d \times d$ matrices, has dimension at most $(n-1) \cdot \binom{d+1}{2}$.
    A particular homogeneous polynomial of degree $2$ which is G\r{a}rding hyperbolic with respect to $(1,0,\ldots,0)$ but cannot be represented in the form
    \eqref{eq:det}
    is the Lorentzian polynomial $f(X) =X_1^2 - X_2^2 - \cdots - X_n^2$ for $n \ge 4$. 
    Cf.  \cite[p. 2498]{Lewis:2005aa}.
\end{remark}

See the
survey article \cite{BGLS01} for background on G\r{a}rding hyperbolic polynomials and 
more ways to generated examples.

\subsubsection{Characteristic roots}
Let $f(Z_1,\ldots,Z_n) \in \R[Z_1,\ldots,Z_n]$ be a homogeneous polynomial of degree $d$ 
which is G\r{a}rding hyperbolic with respect to a direction $v \in \R^n$.
We may factorize
\[
    f(x + T v ) = f(v) \prod_{j=1}^d (T + \la_j^\downarrow(x)),
\]
where
\[
    \la^\downarrow_1(x) \ge \ldots \ge \la^\downarrow_d(x)
\]
are the decreasingly ordered roots of $f(x - T v) \in \R[T]$.  
We call 
\[
    \la^\downarrow = (\la^\downarrow_1, \ldots, \la^\downarrow_d) : \R^n \to \R^d   
\]
the \emph{characteristic map}\index{characteristic map} of $f$ with respect to $v$ and $\la^\downarrow_1, \ldots, \la^\downarrow_d$ 
(in no particular order) the \emph{characteristic roots}.\index{characteristic roots}
The polynomial $f$ is recovered by 
$f(x) = f(v) \prod_{j=1}^d \la_j^\downarrow(x)$.    
Then, for all $j= 1,\ldots,d$, $r \in \R_{\ge 0}$, and $s \in \R$,
\begin{align}
    \begin{split} \label{eq:homogeneous}
        \la^\downarrow_j(r x + s v) &= r \la^\downarrow_j(x) + s, 
        \\
        \la^\downarrow_j(-x) &= - \la^\downarrow_{d+1-j}(x).
    \end{split}   
\end{align}

G\r{a}rding \cite[Theorem 2]{Garding59} proved that $\la^\downarrow_d$ is concave which, by \eqref{eq:homogeneous}, is equivalent to 
$\la^\downarrow_1$ being convex. 
In view of \eqref{eq:homogeneous} it follows that the largest root $\la^\downarrow_1$ is sublinear  (i.e., positively homogeneous and subadditive).  
The connected component $C_f$ of the set $\R^n \setminus \{f = 0\}$ which contains $v$ is an open convex cone, one has 
$C_f = \{x \in \R^n : \la^\downarrow_d(x) >0 \}$,  
and $f$ is G\r{a}rding hyperbolic with respect to each $w \in C_f$.

Following \cite[Theorem 3.1]{BGLS01}, we can produce new G\r{a}rding hyperbolic polynomials from the given $f$:
if $g$ is a homogeneous symmetric polynomial of degree $e$ on $\R^d$ 
which is G\r{a}rding hyperbolic with respect to $(1,1,\ldots,1)$ and has characteristic map $\mu^\downarrow$, 
then $g \o \la^\downarrow$ is G\r{a}rding hyperbolic of degree $e$ with respect to $v$ and its characteristic map is $\mu^\downarrow \o \la^\downarrow$.  

Let us apply this to the homogeneous symmetric polynomial
\[
    g_k(Y_1,\ldots,Y_d) :=  \prod_{\substack{I \subseteq  \{1,\ldots,d\}\\|I| = k}}
    \sum_{i \in I} Y_i
\]
of degree $\binom{d}{k}$ which
is G\r{a}rding hyperbolic with respect to $(1,1,\ldots,1)$ and has the characteristic roots
$\mu_I(Y) = \frac{1}{k} \sum_{i \in I} Y_i$,
where $I$ ranges over the subsets of $\{1,\ldots,d\}$ with $k$ elements.
On the set $\{Y_1 \ge Y_2 \ge \cdots \ge Y_d\}$
we have $\mu_I(Y) \ge \mu_J(Y)$ if and only if $\sum_{i \in I}  i \le \sum_{j \in J} j$, in particular,
$\mu_{\{1,\ldots,k\}}$ is the largest characteristic root.
Then $g_k \o \la^\downarrow$ is G\r{a}rding hyperbolic with respect to $v$ and has
the largest characteristic root $\mu_{\{1,\ldots,k\}} \o \la^\downarrow = \frac{1}{k} \sum_{i=1}^k \la_i^\downarrow$.
From G\r{a}rding's result we may conclude that, for all $k = 1,\ldots, d$, the sum of the $k$ largest roots
\[
    \si_k := \sum_{i = 1}^k \la_{i}^\downarrow
\]
is a sublinear function on $\R^n$. 
Any finite sublinear function on $\R^n$ is globally Lipschitz; see \cite[Corollary 10.5.1]{Rockafellar70}.
So the functions $\si_k$, $k = 1,\ldots,d$, are convex and globally Lipschitz. 
We have proved the following proposition.

\begin{proposition}[{\cite[Proposition 4.1]{Rainer:2021vk}}] \label{prop:DC}
    Let $f(Z_1,\ldots,Z_n) \in \R[Z_1,\ldots,Z_n]$ be a homogeneous polynomial of degree $d$ 
    which is G\r{a}rding hyperbolic with respect to a direction $v \in \R^n$.
    The characteristic map $\la^\downarrow : \R^n \to \R^d$ is globally Lipschitz and difference-convex on $\R^n$.
\end{proposition}

A real valued function on a convex subset of $\R^n$ is called \emph{difference-convex}\index{difference-convex} 
if it can be written as the 
difference of two continuous convex functions (cf.\ \cite{BacakBorwein11} and \cite{Hiriart-Urruty:1985aa}). 
The class $DC$\index{DC@$DC$} of difference-convex functions arises as the smallest 
vector space containing all the continuous convex functions on the given set. 
If $U \subseteq \R^n$ is open, then (cf.\ \cite[Theorem 11]{Vesely:2003aa}
and \cite[Section II]{Hiriart-Urruty:1985aa})
\[
    C^{1,1}(U) \subseteq DC_{\on{loc}}(U) \subseteq C^{0,1}(U).
\]
On the other hand, the first order partial derivatives of a difference-convex function $f : U \to \R$
have bounded variation, i.e., the weak second order partial derivatives of $f$ are signed Radon measures (cf.\ Section \ref{sec:ABV}).
This follows from Dudley's result \cite{Dudley:1977aa} that a Schwartz distribution is a convex function if and only if its second derivative
is a nonnegative matrix valued Radon measure.
In dimension one, a real valued function $f$ on a compact interval is difference-convex if and only if
$f$ is absolutely continuous and $f'$ has bounded variation. 
Note also that $DC(\R^n) = DC_{\on{loc}}(\R^n)$, by \cite{Hartman:1959aa}.

Using the regularity properties of difference-convex functions, we may further investigate 
the regularity of the roots of G{\aa}rding hyperbolic polynomials.

\begin{theorem}[{\cite[Lemma 4.5 and Theorem 4.6]{Rainer:2021vk}}] \label{thm:Gardinghyp}
    Let $f(Z_1,\ldots,Z_n) \in \R[Z_1,\ldots,Z_n]$ be a homogeneous polynomial of degree $d$ 
    which is G\r{a}rding hyperbolic with respect to some direction in $\R^n$.
    Then:
    \begin{enumerate}
        \item If $x : \R \to \R^n$ is of class $C^1$,
            then there exists a differentiable  system $\la = (\la_1,\ldots,\la_d)$ of the roots of $f$ along $x$ 
            (i.e., $\la(t)$ and $\la^\downarrow(x(t))$ coincide as unordered $d$-tuples for all $t$).
        \item Let $x \in DC(\R,\R^n) \cap W^{2,1}_{\on{loc}}(\R,\R^n)$ (for instance, $x \in C^{1,1}(\R,\R^n)$). Then 
            any differentiable system $\la = (\la_1,\ldots,\la_d)$ of the roots of $f$ along $x$ is actually of class
            \begin{equation*} 
                \la \in C^1(\R,\R^d) \cap DC(\R,\R^d)\cap  W^{2,1}_{\on{loc}}(\R,\R^d).
            \end{equation*}
        \item The result is uniform in the following sense:
            Let $I \subseteq \R$ be a bounded open interval.
            Let $U$ be an open neighborhood of  the closure of $I^{1+k}$ in $\R^{1+k}$.
            Suppose that $x : U \to \R^n$ is such that
            \begin{itemize}
                \item $x$ is locally DC on $U$,
                \item $x( \cdot,r) \in C^1(\ol I,\R^n) \cap W^{2,1}(I,\R^n)$ for all $r \in \ol I^k$.
            \end{itemize}
            Assume that, for each $r \in \ol I^k$, a $C^1$ system $\la(\cdot,r)$ of the roots of $f$ along $x(\cdot,r)$ is fixed.
            Then the family
            $\la(\cdot,r)$, for $r \in \ol I^k$, is bounded in $C^1(\ol I, \R^d)$ and there is a nonnegative
            $L^1$ function $m : I^k \to [0,\infty)$
            such that
            \begin{equation*} 
                \|\la(\cdot, r)\|_{W^{2,1}(I,\R^d)} \le m(r), \quad \text{ for a.e.\ } r \in I^k.
            \end{equation*}
    \end{enumerate}
\end{theorem}

\begin{remark}
    It is shown in \cite[Theorem 4.2]{Harvey:2013aa} that for any G{\aa}rding hyperbolic polynomial $f$ with respect to $v \in \R^n$ of degree $d$, 
    $x_0 \in \R^n$, and $w \in C_f$
    there is a system $\la= (\la_1,\ldots,\la_d) : \R \to \R^d$ of the roots of $f$ along $t \mapsto x_0 + t w$ which is
    real analytic and such that each $\la_j : \R \to \R$ is strictly increasing and surjective. The inverses of the $\la_j$
    form a system of the characteristic roots of the G{\aa}rding hyperbolic polynomial $f$ with respect to $w \in \R^n$ along $t \mapsto x_0 + t v$.
    The proof is based on the homogeneity properties \eqref{eq:homogeneous} and the description of $C_f$ as the set, where $\la^\downarrow_1\ge \cdots \ge \la^\downarrow_d$ 
    are positive.
\end{remark}

\subsection{Eigenvalues of Hermitian matrices} \label{sec:Hermitian}

Let us apply the above findings to the G{\aa}rding hyperbolic polynomial $\det$ (with respect to the identity matrix $\I$) 
on the real vector space $\on{Herm}(d)$\index{Herm@$\on{Herm}(d)$} of complex $d \times d$ Hermitian matrices. 
Its characteristic map $\la^\downarrow = (\la^\downarrow_1,\ldots,\la^\downarrow_d): \on{Herm}(d) \to \R^d$ 
assigns to a Hermitian matrix $A$ its eigenvalues in decreasing order 
\[
    \la^\downarrow_1(A) \ge \cdots \ge \la^\downarrow_d(A).
\]
As a special case of Proposition \ref{prop:DC}, we get:

\begin{corollary}[{\cite[Corollary 6.1]{Rainer:2021vk}}] \label{cor:HermLip}
    The characteristic map
    $\la^\downarrow : \on{Herm}(d)  \to \R^d$
    is globally Lipschitz and difference-convex on $\on{Herm}(d)$.
    The sum $\sum_{i=1}^k \la^\downarrow_i$, for $k = 1,\ldots,d$, of the $k$ largest eigenvalues is sublinear.
\end{corollary}

The following result is a consequence of
Theorem \ref{thm:Gardinghyp}.

\begin{theorem}[{\cite[Theorem 6.2]{Rainer:2021vk}}]  \label{thm:Hermitian}
    Let $A : \R \to \on{Herm}(d)$ be a curve of $d \times d$ Hermitian matrices.
    Then:
    \begin{enumerate}
        \item If $A$ is of class $C^1$, then there exists a differentiable system
            $\la=(\la_1,\ldots,\la_d)$ of the eigenvalues of $A$.
        \item If $A$ is of class $DC \cap W^{2,1}_{\on{loc}}$ on $\R$ (for instance, if it is of class $C^{1,1}$), 
            then any differentiable system $\la=(\la_1,\ldots,\la_d)$ of the eigenvalues of $A$ is actually of class
            \[
                \la \in C^1(\R,\R^d) \cap DC(\R,\R^d)\cap  W^{2,1}_{\on{loc}}(\R,\R^d).
            \]
        \item  The result in \thetag{2} is uniform in the following sense:
            Let $I \subseteq \R$ be a bounded open interval.
            Let $U$ be an open neighborhood of  the closure of $I^{1+k}$ in $\R^{1+k}$.
            Suppose that $A : U \to \on{Herm}(d)$ is such that
            \begin{itemize}
                \item $A$ is locally DC on $U$,
                \item $A( \cdot,r)$ is of class $C^1 \cap W^{2,1}$ on $\ol I$ for all $r \in \ol I^k$.
            \end{itemize}
            Assume that, for each $r \in \ol I^k$, a $C^1$ system $\la(\cdot,r)$ of the eigenvalues of $A(\cdot,r)$ is fixed.
            Then the family
            $\la(\cdot,r)$, for $r \in \ol I^k$, is bounded in $C^1(\ol I, \R^d)$ and there is a nonnegative
            $L^1$ function $m : I^k \to [0,\infty)$ such that
            \begin{equation*}
                \|\la(\cdot, r)\|_{W^{2,1}(I,\R^d)} \le m(r), \quad \text{ for a.e.\ } r \in I^k.
            \end{equation*}
    \end{enumerate}
\end{theorem}

The conclusion of this theorem is best-possible among all Sobolev spaces $W^{k,p}$. Indeed, by the Sobolev inequality,
$W^{k,p}$ regularity with $k+p>3$ would imply $C^{1,\al}$ regularity with some $\al>0$, contradicting the following counterexample.

\begin{example}[{\cite{KM03}}] \label{ex:2x2sym}
    There is a $C^\infty$ curve $A(t)$, $t \in \R$, of real symmetric $2 \times 2$ matrices and a convergent sequence $t_n \in \R$ 
    such that 
    \[
        A(t_n+t) = 
        \begin{pmatrix}
            2^{-n^2} & 2^{-n}t
            \\
            2^{-n}t  & - 2^{-n^2}
        \end{pmatrix},
        \quad \text{ for } |t| \le n^{-2}, \, n\ge 1.
    \]
    There is no $C^{1,\al}$-system of the eigenvalues of $A$ for any $\al>0$.
\end{example}

A little more work leads to the following stronger version of (1) in the above theorem.

\begin{theorem} \label{thm:C1Rellich}
    If $A : \R \to \on{Herm}(d)$ is of class $C^1$, then there exists a $C^1$ system
    $\la=(\la_1,\ldots,\la_d)$ of the eigenvalues of $A$.
\end{theorem}

This is due to Rellich \cite{Rellich69} in the case of symmetric matrices (see also Weyl \cite{Weyl12}), and it was proved
for normal matrices in Rainer \cite{RainerN}; see Theorem \ref{thm:normalC1} below. 

Note that, conversely, Theorem \ref{thm:Hermitian} implies Theorem \ref{thm:Gardinghyp} if $n\le 3$; cf.\ Remark \ref{rem:Laxconjecture}.

\subsection{Singular values}

Theorem \ref{thm:Hermitian} contains information on the regularity of singular values.\index{singular values}
Let $A \in \C^{m \times n}$ be any complex $m \times n$ matrix and let
$$\si^\downarrow_1(A) \ge \si^\downarrow_2(A) \ge \cdots \ge \si^\downarrow_n(A) \ge 0$$
be the singular values of $A$ in decreasing order, i.e., the nonnegative square roots of the eigenvalues of $A^* A$.
If $\on{rank} A = \ell$, then $\si^\downarrow_{\ell+1}(A) = \cdots = \si^\downarrow_{n}(A) = 0$.
Thus we set $\ell := \min\{m,n\}$ and consider only $\si^\downarrow_{j} (A)$, for $j = 1,\ldots,\ell$.
We may consider the $\si^\downarrow_j$ as functions on the vector space $\C^{m \times n}$.

Without loss of generality assume that $m \le n$ and let $\tilde A$ be the $n \times n$ matrix resulting from $A$ by
adding $n-m$ rows consisting of zeros. Then the eigenvalues of the Hermitian matrix
\begin{equation} \label{eq:Hermitianext}
    \begin{pmatrix}
        0 & \tilde A \\
        \tilde A^* & 0
    \end{pmatrix}
\end{equation}
are precisely
\[
    \si^\downarrow_1(A) \ge   \cdots \ge \si^\downarrow_n(A) \ge
    -\si^\downarrow_n(A) \ge  \cdots \ge - \si^\downarrow_1(A).
\]
By Corollary \ref{cor:HermLip}, for all $k = 1,\ldots,\ell$, the sum
\begin{equation*} 
    \sum_{j = 1}^k \si^\downarrow_j
\end{equation*}
of the $k$ largest singular values is a sublinear function on $\C^{m \times n}$ viewed as a real vector space. In fact, the sums
$A \mapsto \sum_{j = 1}^k \si^\downarrow_j(A)$ are the so-called \emph{Ky Fan norms}.

\begin{corollary}
    The mapping
    $\si^\downarrow = (\si^\downarrow_1, \ldots,\si^\downarrow_\ell) : \C^{m \times n}  \to \R^\ell$, where $\ell = \min\{m,n\}$,
    is globally Lipschitz and difference-convex on $\C^{m \times n}$.
\end{corollary}

\begin{theorem}
    Let $A : \R \to \C^{m \times n}$ be a curve of $m \times n$ complex matrices
    of class $C^1 \cap DC \cap W^{2,1}_{\on{loc}}$.
    If $\on{rank} A(t) = \min\{m,n\}=:\ell$ for all $t$, then
    there exists a system $\si = (\si_1,\ldots,\si_\ell)$ of the singular values of $A$
    such that $\si \in C^1(\R,\R^\ell) \cap DC(\R,\R^\ell)\cap W^{2,1}_{\on{loc}}(\R,\R^\ell)$.
    This result is uniform in the sense explained in Theorem \ref{thm:Hermitian}.
\end{theorem}

\begin{proof}
    Apply Theorem \ref{thm:Hermitian} to the Hermitian matrix \eqref{eq:Hermitianext}.
    The condition on the rank of $A$ guarantees that
    the nontrivial singular values of $A$ are always strictly positive and hence there exists a $C^1$ system of them,
    since there exists a $C^1$ system of the eigenvalues of \eqref{eq:Hermitianext}.
\end{proof}

The rank condition is necessary; for instance, the singular value of the symmetric $1\times 1$ matrix $A(t) = (t)$ is $|t|$
which does not admit a $C^1$ parameterization.

\subsection{Eigenvalues of normal matrices}
\label{sec:normal}

We have seen that the eigenvalues of Hermitian matrices have similar regularity properties as the roots of 
hyperbolic polynomials. Actually, they are slightly stronger under much weaker assumptions, see Theorem \ref{thm:Hermitian}, 
since the eigenvalues of Hermitian matrices are the characteristic roots of the determinant which is 
G{\aa}rding hyperbolic with respect to the identity matrix.

Interestingly, the eigenvalues of the wider class of normal complex matrices have similar strong regularity properties.

\begin{theorem}[{\cite{BhatiaDavisMcIntosh83}, \cite[VII.4.1]{Bhatia97}}] \label{d-Lip}
    Let $A,B$ be normal complex $d \times d$ matrices and
    let $\la_j(A)$ and 
    $\la_j(B)$, $1 \le j \le d$, denote the respective eigenvalues.
    Then 
    \[
        \min_{\si \in \on{S}_d} \max_{1\le j \le d} |\la_j(A)-\la_{\si(j)}(B)| 
        \le C \,\|A-B\|
    \]
    for a universal constant $C$ with $1 < C < 3$,
    where $\on{S}_d$ is the symmetric group and $\|\cdot\|$ is the operator norm.
\end{theorem}

This result shows that the unordered $d$-tuple of eigenvalues $[\la_1(A),\ldots, \la_d(A)]$ is 
Lipschitz continuous as a function of the normal matrix $A$.
On the left-hand side we have the Wasserstein metric $W_\infty$ (up to a constant factor); see Example \ref{example:n=1}.
In general, the single eigenvalues do not admit a 
continuous parameterization, as shown by the following example.

\begin{example}
    Each choice of the eigenvalues of the normal matrix 
    \[
        A(z):= 
        \begin{pmatrix}
            0 & z 
            \\
            |z| & 0
        \end{pmatrix}, \quad z \in \C,
    \]
    must be discontinuous in a neighborhood of $0$.
\end{example}

Continuous systems of the eigenvalues always exist for continuous families of Hermitian matrices (e.g., by ordering the eigenvalues 
increasingly or decreasingly) or for continuous real one-parameter families of normal matrices (cf.\ \cite[II.5.2]{Kato76}).
The following theorem should be compared to Bronshtein's theorem \ref{thm:Bronshtein}.

\begin{theorem}[{\cite[Proposition 6.3]{RainerN}}] \label{thm:normalLip}
    Let $A(t)$, $t \in \R$, be a $C^{0,1}$ curve of normal complex matrices. 
    Then any continuous system of the eigenvalues of $A$ is $C^{0,1}$.
\end{theorem}

Let us sketch the main ideas in the proof.
Fix $s \in \R$ and an eigenvalue $z$ of $A(s)$. Let $m$ be the multiplicity of $z$.
For $t$ near $s$, there are $m$ eigenvalues $\la_1(t), \ldots,\la_m(t)$, close to $z$ 
and $\la_j$, $1 \le j \le m$, are continuous functions.
For each $t$, we may choose an orthonormal system of eigenvectors $v_1(t), \ldots,v_m(t)$ corresponding to $\la_1(t), \ldots,\la_m(t)$.
For each $t$ and each sequence $t_k \to t$, we have, after passing to a subsequence, $v_j(t_k) \to w_j(t)$, 
where $w_1(t), \ldots,w_m(t)$ form an orthonormal system of eigenvectors of $A(t)$ (restricted to the direct sum of eigenspaces of the $\la_j(t)$, $1 \le j \le m$). 
Then
\begin{equation*}
    \frac{A(t) - \la_j(t)}{t_k-t} v_j(t_k) + \frac{A(t_k) - A(t)}{t_k-t} v_j(t_k) = \frac{\la_j(t_k) - \la_j(t)}{t_k-t} v_j(t_k).
\end{equation*}
By taking the inner product with $w_j(t)$ and passing to the limit, one can show that 
for any continuous system $\la_j$ of the eigenvalues of $A$
the derivatives $\la_j'(t)$ exist whenever $A'(t)$ exists, 
they form the set of eigenvalues of $A'(t)$ and satisfy 
\begin{equation} \label{eq:normalLipkey}
    \la_j'(t)= \langle A'(t)w_j(t) \mid w_j(t) \rangle,
\end{equation}
possibly after applying a suitable permutation of the eigenvalues on one side of $t$.
Now it suffices to check that $\la_j$ is locally absolutely continuous. 
Then, by means of \eqref{eq:normalLipkey}, we see that $\la_j'$ exists almost everywhere and is locally bounded.
We may conclude that $\la_j$ is $C^{0,1}$.

With the help of the formula \eqref{eq:normalLipkey},
one can show that there are $C^1$ systems of the eigenvalues provided that $t \mapsto A(t)$ is $C^1$.

\begin{theorem}[{\cite[Proposition 6.11]{RainerN}}] \label{thm:normalC1}
    Let $A(t)$, $t \in \R$, be a $C^{1}$ curve of normal complex matrices. 
    Then there is a $C^1$ system of the eigenvalues of $A$.
    If $t \mapsto A(t)$ is $C^2$, then there  is a $C^1$ system of the eigenvalues of $A$ 
    which is twice differentiable everywhere.
\end{theorem}

Note that this is best possible, by Example \ref{ex:2x2sym}. 
The next example shows that the assumption that $A$ is normal cannot be omitted. 

\begin{example}[{\cite[II.5.9]{Kato76}}] \label{ex4}
    Let $\al >1$ and $\be >2$. Then 
    \[
        A(t) = 
        \begin{pmatrix}
            |t|^\al & |t|^\al-|t|^\be \big(2+\sin \frac{1}{|t|} \big) \\
            -|t|^\al & -|t|^\al
        \end{pmatrix}, 
        \quad t \in \R \setminus \{0\}, \quad A(0)=0,
    \]
    forms a $C^1$ curve of diagonalizable (but not normal) matrices.
    The eigenvalues of $A$ are given by 
    \[
        \la_\pm(t) = \pm |t|^{\frac{\al+\be}{2}} \Big(2+\sin \frac{1}{|t|} \Big)^{\frac{1}{2}}, \quad t \in \R \setminus \{0\}, \quad \la_\pm(0)=0.
    \] 
    The derivatives $\la_\pm'$ exist everywhere, but they are discontinuous at $0$ if $\al +\be \le 4$ and even unbounded near $0$ if $\al +\be < 4$.   
\end{example}

For multiparameter families of normal matrices we have the following.

\begin{corollary}[{\cite[Theorem 6.19]{RainerN}}]
    Let $A(x)$, $x \in U$, be a $C^{0,1}$ family of normal complex matrices on an open subset $U \subseteq \R^n$
    and let $\la : V \to \C$ be a continuous eigenvalue of $A$ defined on an open subset $V \subseteq U$.
    Then $\la$ is of class $C^{0,1}$. 
    Moreover, if $x_0 \in U \cap \ol V$ and $c: \R \to U$ is a $C^{0,1}$ curve with $c(0) = x_0$ 
    and $c((0,1)) \subseteq V$, then $\la \o c|_{(0,1)}$ is globally Lipschitz on $(0,1)$.
\end{corollary}

The first statement follows from Theorem \ref{thm:normalLip} and the fact that $C^{0,1}$ regularity can be tested on $C^\infty$ curves (cf.\ \cite{KMbook}).
For the second assertion, we note that the formula \eqref{eq:normalLipkey} yields that 
the Lipschitz constant of $\la \o c|_{(0,1)}$ is globally bounded by the one of $A \o c|_{(0,1)}$.

In the spirit of Theorem \ref{thm:contacthyp}, it is possible to give sufficient conditions for the existence 
of $C^p$ eigenvalues and eigenvectors of curves of normal matrices in terms of the differentiability of the entries and 
the order of contact of the eigenvalues.

\begin{theorem}[{\cite[Theorem 7.8]{RainerFin}}] \label{thm:contactnormal}
    Let $A(t)$, $t \in I$, be a curve of normal complex matrices.
    Suppose that the characteristic polynomial is normally nonflat. 
    There exists $\Th(A) \in \N \cup \{\infty\}$
    such that if $p \in \N_{\ge 1} \cup \{\infty\}$ and $A(t)$, $t \in I$, has $C^{p+ \Th(A)}$ entries, 
    then it admits a $C^{p+\Th(A)}$ system of its eigenvalues and a $C^p$ system of eigenvectors.
\end{theorem}

$\Th(A)$ is defined in terms of a splitting algorithm for the characteristic polynomial of $A$. 
If the entries of the normal matrix $A$ are definable in an o-minimal expansion of the real field,
then the eigenvalues have the same regularity without the assumption of normal nonflatness. 

\begin{theorem}[{\cite[Theorem 8.1]{RainerFin}}] \label{thm:defnormal}
    Let $A(t)$, $t \in I$, be a definable curve of normal complex matrices.
    There exists an integer $\Th(A) \in \N$
    such that if $p \in \N \cup \{\infty\}$ and $A(t)$, $t \in I$, has $C^{p+ \Th(A)}$ entries, 
    then it admits a $C^{p+\Th(A)}$ system of its eigenvalues.
\end{theorem}

That there is no hope for continuous eigenvectors in this setup (even if there is no oscillation) is seen by the following example.

\begin{example}[{\cite[Example 9.5]{RainerFin}}] \label{ex:defnormal}
    Let 
    \[
        B(t) := \begin{pmatrix}
            1 & 0 
            \\ 
            0 & 1
        \end{pmatrix} \quad \text{ for } t \ge 0 \quad \text{ and } \quad 
        B(t) := \begin{pmatrix}
            1 & 1 
            \\ 
            1 & -1
        \end{pmatrix} \quad \text{ for } t < 0.
    \]
    Then $A(t) := t^{p+1} B(t)$ is a definable $C^p$ curve of real symmetric matrices 
    which does not admit a continuous system of eigenvectors. 
    If the underlying o-minimal expansion of the real field defines the exponential function, then 
    $A(t) := e^{-1/t^2} B(t)$ is a definable $C^\infty$ curve of real symmetric matrices 
    not admitting continuous eigenvectors.
\end{example}

\begin{remark}
    Many of the results are true (in slightly adjusted form) 
    for parameterized families of unbounded normal operators in Hilbert space with common domain of definition 
    and with compact resolvent. Cf.\ \cite{KMRp,KMR,RainerN} and references therein.
\end{remark}

\section{Regularity of the roots in the general (nonhyperbolic) case}\label{sec:regularity}

Let $I \subseteq \R$ be a bounded open  interval and let 
\begin{equation} \label{eq:curveofpolynomials}
    P_a(t)(Z)=  Z^d + \sum_{j=1}^d a_j(t) Z^{d-j}, \quad t \in I, 
\end{equation}
be a monic polynomial with complex valued functions $a_j : I \to \C$, $j = 1,\ldots,d$, as coefficients. 
It is not hard to see that if the coefficients $a_j$ are continuous then $P_a$ always admits a continuous system of the roots 
(e.g.\ \cite[II.5.2]{Kato76}). 
In the absence of hyperbolicity assumptions, the roots of $P$ cannot be chosen Lipschitz even in the simplest radical case:  $a_d =t$ and $a_j\equiv 0$ for $j\le d-1$.  

Motivated by the analysis of certain systems of pseudo-differential equations, Spagnolo \cite{Spagnolo00} posed 
the question of absolute continuity of the roots of $P_a$.  
Using explicit formulas for the roots, he gave the positive answer for the polynomials of degree 2 and 3 in  \cite{Spagnolo99}. 

The positive answer to this question in the general case was given in Parusi\'nski and Rainer \cite{ParusinskiRainerAC, ParusinskiRainerOpt} 
with two different proofs.  
They both use in a substantial way the work of Ghisi and Gobbino \cite{GhisiGobbino13} 
on the radical case, where the optimal regularity conditions were given.  The paper \cite{ParusinskiRainerOpt} also contains the optimal Sobolev regularity of the roots in the general case, see below for precise statements.

\subsection{The case of radicals}

The first result towards absolute continuity of the roots 
is probably Lemma~1 in Colombini, Jannelli, and Spagnolo \cite{ColombiniJannelliSpagnolo83} which states that 
for a real valued nonnegative function $f$ of class $C^{k,\al}(\ol I)$ on a bounded open interval $I$, with 
$k \in \N_{\ge 1}$ and $0 < \al \le 1$, the real radical $f^{1/(k+\al)}$ is absolutely continuous on $I$ and satisfies
\[
    \|(f^{\frac{1}{k+\al}})'\|_{L^1(I)}^{k+\al} \le C(k,\al,I) \|f\|_{C^{k,\al}(\ol I)}.
\] 
Tarama \cite{Tarama00} extended this lemma to real valued functions (not necessarily nonnegative). 
A better summability for the weak partial derivatives of $f^{1/(k+1)}$ was obtained by Colombini and Lerner \cite{ColombiniLerner03} for nonnegative $C^{k+1}$ functions $f$ of several real variables.

The case of radicals was settled by Ghisi and Gobbino \cite{GhisiGobbino13} by 
finding their optimal regularity. 

\begin{theorem}[Ghisi and  Gobbino {\cite[Theorem 2.2]{GhisiGobbino13}}]\label{thm:GG}
    Let $I \subseteq \R$ be a bounded open interval and let $f$ be a real valued continuous function such that there exists 
    $g \in C^{k,\al}(\ol I)$ so that $|f|^{k+\al{}}=|g|$ on $I$. Then   
    $f' \in L_w^p(I)$, where $\frac{1}{p}+\frac{1}{k+\al}=1$, and
    \begin{align}\label{eq:GG}
        \|f'\|_{p,w,I} \le 
        C(k) \max\Big\{[\Hoeld_{\al,I}(g^{(k)})]^{\frac{1}{k+\al}} |I|^{\frac{1}{p}}, 
        \|g'\|_{L^\infty(I)}^{\frac{1}{k+\al}}\Big\};
    \end{align}
    in particular, $f$ belongs to the Sobolev space $W^{1,q}(I)$ for every $q \in [1,p)$.  
\end{theorem}

Here $L^p_w(I)$\index{Lpw@$L^p_w$} denotes the weak Lebesgue space\index{weak Lebesgue space} equipped with the quasinorm 
$\|f\|_{p,w,I} := \sup_{r\ge 0} \big\{r \cdot \cL^1(\{t \in I : |f(t)| > r\})^{\frac{1}{p}}\big\}$,
where $\cL^1$ is the one dimensional Lebesgue measure. 
By $\Hoeld_{\al,I}(g^{(k)})$ we mean the $\al$-H\"older constant of $g^{(k)}$ on $I$, and $|I|=\cL^1(I)$ is the length of the interval $I$, see also Section \ref{Appendix:A}. 

Ghisi and Gobbino also provided examples that show that the assumptions as well as the conclusion in their theorem are 
best possible.
\begin{itemize}
    \item Take $g(t)=t$ and $f(t)=|t|^{\frac{1}{k+\al}}$ on the interval $I=(-1,1)$.
        Then $f' \not\in L^p(I)$ for $p$ defined by $\frac{1}{p}+\frac{1}{k+\al}=1$. 
        So the summability of $f'$ in Theorem \ref{thm:GG} is optimal. See \cite[Example 4.3]{GhisiGobbino13}.
    \item In \cite[Example 4.4]{GhisiGobbino13}, a $C^\infty$ function $g : I \to \R$ on an interval $I=(0,r)$ with the following properties is constructed: 
        \begin{itemize}
            \item $g$ belongs to $C^{k,\be}(\ol I)$ for every $\be <\al$ but not for $\be=\al$.
            \item $f := |g|^{\frac{1}{k+\al}}$ has unbounded variation on $I$ and hence $f' \not\in L^1(I)$.
        \end{itemize}
        So the differentiability assumption on $g$ in Theorem \ref{thm:GG} is optimal.
\end{itemize}

Higher order Glaeser inequalities  \cite[Theorems 2.1 and 3.2]{GhisiGobbino13} play an important role in the proof of Theorem \ref{thm:GG}. 
We recall a slightly more general statement in the next subsection, see Corollary \ref{cor:hoGleaser}.

\subsection{Optimal Sobolev regularity of the roots}

\begin{theorem}[{\cite[Theorem 1]{ParusinskiRainerOpt}}] \label{thm:optimal}
    Let $I \subseteq \R$ be a bounded open interval and let $P_a$   
    be a monic polynomial \eqref{eq:curveofpolynomials} with complex valued 
    coefficients $a_j \in C^{d-1,1}(\ol I)$, $j = 1,\ldots,d$. 
    Let $\la \in C^0(I)$ be a continuous root of $P_a$.
    Then $\la$ is absolutely continuous on $I$ and belongs to the Sobolev space 
    $W^{1,p}(I)$ for every $1 \le p < d/(d-1)$. The derivative $\la'$ satisfies  
    \begin{align} \label{bound} 
        \| \la' \|_{L^p(I)}  
   &\le C(d,p) \max\{1, |I|^{1/p}\} 
   \max_{1 \le j \le d} \|a_j\|^{1/j}_{C^{d-1,1}(\ol I)},
    \end{align}
    where the constant $C(d,p)$ depends only on $d$ and $p$.
\end{theorem}

The above result can be understood as a complex analogue of Bronshtein's theorem, see Theorem \ref{thm:Bronshtein}. 
Its proof given in \cite{ParusinskiRainerOpt} follows  the strategy of the proof of 
Bronshtein's theorem as explained in Section \ref{sec:Bronshteinproof}, though many details are 
significantly more complex.   The universal splitting $P_{\tilde a} = P_b P_c$ allows the induction 
on the degree of $P_a$. 
After the Tschirnhausen transformation $P_b \leadsto P_{\tilde b}$ the new polynomials 
$P_{\tilde b}$ split again near points $t_1 \in I $, 
where not all $\tilde b_i$ vanish.  The central idea of the underlying induction is to show an inequality 
between the combined coefficients of $P_{\tilde a}$ and $P_{\tilde b}$ and a similar expression for 
$P_{\tilde b}$ and the polynomial issued from the splitting of  $P_{\tilde b}$, that is \cite[Formula (18)]{ParusinskiRainerOpt}.  
The proof of this technical step is based on the radical case, Theorem \ref{thm:GG}, and on the following interpolation lemma.

\begin{lemma}[{\cite[Lemma 4]{ParusinskiRainerOpt}}] 
    Let $I \subseteq \R$ be a bounded open interval, $m \in \N_{>0}$, and $\al \in (0,1]$.
    If $f\in C^{m,\al}(\overline I)$, then for all $t\in I$ 
    and  $s = 1,\ldots,m$,  
    \begin{align}\label{eq:1}  
        |f^{(s)}(t) | \le C |I|^{-s} \bigl(V_I(f) + V_I(f)^{(m+\al-s)/(m+\al)} (\Hoeld_{\al,I}(f^{(m)}))^{s/(m+\al)}  |I|^s
        \bigr),  
    \end{align}
    for a universal constant $C$ depe{}nding only on $m$ and $\al$.
\end{lemma}

Here $
V_I(f) := \sup_{t,s \in I} |f(t)-f(s)|$.  As a corollary we get the following generalizations of 
the higher order Glaeser inequalities of \cite{GhisiGobbino13}.\index{higher order Glaeser inequalities} 

\begin{corollary}[{\cite[Corollary 3]{ParusinskiRainerOpt}}]\label{cor:hoGleaser}
    Let $m \in \N_{>0}$ and $\al \in (0,1]$.
    Let $I = (t_0- \de,t_0+\de)$ with $t_0 \in \R$ and $\de>0$.
    If $f\in C^{m,\al}(\overline I)$ is such that $f$ and $f'$ do not change their sign on $I$,
    then for all $s = 1,\ldots,m$,
    \begin{align}\label{eq:2}  
        |f^{(s)}(t_0) | \le C |I|^{-s} \bigl(|f(t_0)| + |f(t_0)|^{(m+\al-s)/(m+\al)} (\Hoeld_{\al,I}(f^{(m)}))^{s/(m+\al)}  |I|^s
        \bigr),  
    \end{align}
    for a universal constant $C$ depending only on $m$ and $\al$.
\end{corollary}

Note that the bound \eqref{bound} 
is not invariant under rescaling of the interval $I$, 
in contrast to \eqref{eq:GG}. 

\begin{open}
    Replace \eqref{bound} by a scale invariant bound.    
\end{open}

Unlike in the case of radicals, i.e., Theorem \ref{thm:GG}, it is not known whether in the setting of Theorem \ref{thm:optimal}, 
$\la'$ belongs to the weak Lebesgue space $L^{d/(d-1)}_w(I)$.   
The proof cannot be, at least easily, adapted because $\|\cdot\|_{p,w,I}^p$ is not $\si$-additive.

\begin{open}
    Is $\la'$ in the setting of Theorem \ref{thm:optimal} an element of the weak Lebesgue space $L^{d/(d-1)}_w(I)$?
    If so is there an explicit bound for $\|\la'\|_{d/(d-1),w,I}$ in terms of the coefficients $a_j$ and the 
    interval $I$?  
\end{open}

\subsubsection{Multivalued Sobolev functions}\index{multivalued Sobolev functions}

In \cite{Almgren00} Almgren developed a theory of $d$-valued Sobolev functions.   
A $d$-valued function is a mapping with values in the set $\cA_d(\R^m)$ of unordered $d$-tuples of points in $\R^m$. 
Let us denote by $[x] = [x_1,\ldots,x_d]$ the unordered $d$-tuple\index{unordered tuple} consisting of $x_1,\ldots,x_d \in \R^m$; then 
$[x_1,\ldots,x_d] = [x_{\si(1)},\ldots,x_{\si(d)}]$ for each permutation $\si \in \on{S}_d$. 
The set $\cA_d(\R^m) = \{[x] = [x_1,\ldots,x_d] : x_i \in \R^m\}$\index{Ad@$\cA_d(\R^m)$} forms a 
complete metric space when endowed with the metric
\[
    d([x],[y]) := \min_{\si \in \on{S}_d} \Big(\sum_{i=1}^d |x_i - y_{\si(i)}|^2\Big)^{1/2}.
\]
This is (up to a constant factor) the Wasserstein metric $W_2$, see Example \ref{example:n=1}.
Almgren proved that 
there is an integer $N= N(d,m)$, a positive constant $C= C(d,m)$, and an  
injective Lipschitz mapping $\De : \cA_d(\R^m) \to \R^N$, with Lipschitz constant  
$\on{Lip}(\De) \le 1$ and $\on{Lip}(\De|^{-1}_{\De(\cA_d(\R^m))}) \le C$; moreover, there is a
Lipschitz retraction of $\R^N$ onto $\De(\cA_d(\R^m))$.

One can use this bi-Lipschitz embedding to define Sobolev spaces of $d$-valued functions: 
for open $U \subseteq \R^n$ and $1 \le p \le \infty$ define 
\[
    W^{1,p}(U,\cA_d(\R^m)) := \{f : U \to \cA_d(\R^m) : \De \o f \in W^{1,p}(U,\R^N)\}.  
\]
For an intrinsic definition, see \cite[Definition~0.5 and Theorem~2.4]{De-LellisSpadaro11}.

Let us identify $\R^2 \cong \C$. 
Theorem \ref{thm:optimal} implies a sufficient condition for a $d$-valued function $U \to \cA_d(\C)$ to belong to the 
Sobolev spaces $W^{1,p}(U,\cA_d(\C))$ for every $1 \le p < d/(d-1)$; see Theorem \ref{thm:multivalued} below.

We shall use the following terminology. By a \emph{parameterization}\index{parameterization of multivalued function} of a $d$-valued function $f : U \to \cA_d(\C)$ we mean a
function $\vh: U \to \C^d$ such that $f(x) = [\vh(x)] = [\vh_1(x),\ldots,\vh_d(x)]$ for all $x \in U$. 
Let $\pi : \C^d \to \cA_d(\C)$ be defined by $\pi(z) := [z]$; it is a Lipschitz mapping with Lipschitz constant $\on{Lip}(\pi) = 1$.
Then a parameterization of $f$ amounts to a lifting $\vh$ of $f$ over $\pi$, i.e.,
$f = \pi \o \vh$.  
The elementary symmetric polynomials induce a bijective mapping  
$a : \cA_d(\C) \to \C^d$, 
\[
    a_j ([z_1,\ldots,z_d]) := (-1)^j \sum_{i_1<\cdots<i_j} z_{i_1} \cdots z_{i_j}, \quad  1 \le j \le d.
\]
In other words, the monic complex polynomials of degree $d$ are in one-to-one correspondence 
with their unordered $d$-tuples of roots.

\begin{theorem} [{\cite[Theorem 6]{ParusinskiRainerOpt}}]\label{thm:multivalued}
    Let $U \subseteq \R^n$ be open and
    let $f : U \to \cA_d(\C)$ be continuous.
    If $a \o f \in C^{d-1,1}(U,\C^d)$, then $f \in W^{1,p}(V,\cA_d(\C))$ for each relatively compact open $V \Subset U$ 
    and each $1 \le p < d/(d-1)$. 
    Moreover,
    \begin{equation*}
        \|\nabla (\De \o f)\|_{L^p(V)} 
        \le C(d,n,p,\cK,\De) \big(1 + \max_{1 \le j \le d} \|a_j \o f\|^{1/j}_{C^{d-1,1}(\overline W)}\big),
    \end{equation*}
    where $\cK$ is any finite cover of $\overline V$ by open boxes $\prod_{i=1}^d (\al_i,\be_i)$ 
    contained in $U$ and $W = \bigcup \cK$.
\end{theorem}

In particular, the roots of a polynomial $P_a$ of degree $d$ with coefficients 
$a_j \in C^{d-1,1}(U)$, $j = 1,\ldots,d$, form a $d$-valued function  
$\la : U \to \cA_d(\C)$ which belongs to $W^{1,p}_{\on{loc}}(U,\cA_d(\C))$ for each $1 \le p < d/(d-1)$; 
in fact, it is well-known that $\la : U \to \cA_d(\C)$ is 
continuous (cf.\ \cite{Kato76} or \cite[Theorem 1.3.1]{RS02}).
Theorem \ref{thm:multivalued} implies that the push-forward  
\begin{equation*}
    (a^{-1})_* : C^{d-1,1}(U,\C^d) \to \bigcap_{1 \le p <d/(d-1)} W^{1,p}_{\on{loc}}(U,\cA_d(\C)).
\end{equation*}
is a bounded mapping.  Similarly to the hyperbolic case, the continuity of this map is an important open problem; 
compare with  Example \ref{ex:Br-continuity}, Remark \ref{rem:Br-continuity}, and Open Problem \ref{op:Br-continuity}.

\begin{open}
    Is the map $(a^{-1})_*$ continuous?\footnote{For $C^d(U,\mathbb C^d)$ this was recently proved in \url{https://arxiv.org/abs/2410.01326}.} 
\end{open}

\subsection{Absolute continuity via desingularization. Formulas for the roots}\label{sec:formulas}

The main theorem of Parusi{\'n}ski and Rainer \cite{ParusinskiRainerAC} also implies the absolute continuity of the roots of \eqref{eq:curveofpolynomials} 
provided the coefficients $a_j$ are of class  $C^{k}$, for $k$ sufficiently large.  
Its statement, see Theorem \ref{thm:byresolution} below, is weaker than the optimal bound of Theorem \ref{thm:optimal}.  On the other hand, it presents a different approach 
through modifications of the space of coefficients of the universal monic polynomial
that replaces the arguments for each curve of polynomials separately.  
This approach can be better suited in some cases, for instance in the proof of the bounded variation property, see Theorem \ref{thm:BVroots} below.  

\begin{theorem}[{{\cite[Main Theorem]{ParusinskiRainerAC}}}]\label{thm:byresolution}
    For every $d\in \N_{>0}$ there is $k=k(d) \in \N_{>0}$ and  $p=p(d)>1$ such that the following holds.  
    Let $I\subseteq \R$ be a bounded open interval and let $P_a$
    be a monic polynomial of degree $d$ with complex valued coefficients $a_j \in C^{k}(\overline I)$, $j = 1,\ldots,d$.   Then: 
    \begin{enumerate}
        \item 
            Let $\la_j \in C^0(I)$, $j = 1,\ldots,d$, be a continuous parameterization of the roots of $P_a$ on $I$.
            Then the distributional derivative of each $\la_j$ in $I$ is a measurable function $\la_j' \in L^q(I)$ for every 
            $q \in [1,p)$. 
            In particular, each $\la_j \in W^{1,q}(I)$ for every $q \in[1,p)$.
        \item This regularity of the roots is uniform.  Let $\{P_{a_\nu} : \nu \in \cN\}$,
            \[
                P_{a_\nu(t)}(Z) = Z^d + \sum_{j=1}^d a_{\nu,j}(t) Z^{d-j} \in C^{k}(\overline I)[Z], ~\nu \in \cN,   
            \] 
            be a family of curves of polynomials, indexed by $\nu$ in some set $\cN$, so that the set of coefficients 
            $\{a_{\nu,j} : \nu \in \cN, j=1,\ldots,d\}$ is bounded in $C^{k}(\overline I)$. 
            Then the set
            \[
                \qquad \{\la_{\nu} \in C^0(I):  P_{a_\nu}(\la_\nu)=0
                \text{ on $I$}, ~\nu \in \cN\}
            \]    
            is bounded in $W^{1,q}(I)$ for every $q \in [1,p)$. 
    \end{enumerate} 
\end{theorem}

For instance, for $d=3$ the proof given in \cite[Part 3]{ParusinskiRainerAC} provides $ k(3)=6$ and $p(3)= 6/5$, 
which is worse than the optimal $k=3$, $p=3/2$, but still better than Spagnolo's result which gives the absolute continuity
(i.e. $p=1$) for $k=25$.

The proof of the above theorem is based on formulas for the roots of the universal monic polynomial 
$P_a(Z)=Z^d + \sum_{j=1}^d a_j Z^{d-j}$ in terms of its coefficients $a=(a_1,\ldots,a_d)\in \C^d$.  
They are obtained as follows.  By Hironaka's resolution of singularities \cite{Hironaka:1964}, there is a tower of 
smooth principalizations 
\begin{align}\label{eq:modifications}
    \C^d  \stackrel{ \sigma_2 }\longleftarrow M_2  \stackrel{ \sigma_{3,2} }\longleftarrow  M_3
    \stackrel{ \sigma_{4,3} }\longleftarrow  \cdots \stackrel{ \sigma_{d,d-1} }\longleftarrow M_{d} 
\end{align}
which successively principalize the generalized discriminant ideals $\sD_m$ in $\C[a]$, $m=2,\ldots,d$, defined in  
{\cite[Section 1.1]{ParusinskiRainerAC}}. 
These ideals give the stratification of the space of polynomials by root multiplicities. 
Then, locally on $M_{d}$, the roots of the pulled back polynomial $P_{\si_d^* (a)}$ are given by linear combinations with 
rational coefficients 
\begin{align}\label{eq:formulasforroots}
& \sum_{m=1}^d A_m \, (\vh_m \o \si_{d,m}), \quad \text{ where }\\ 
\notag
&  \varphi_\kk  = f_\kk^ {\al_\kk}  
\psi_\kk (y_{\kk,1}^{1/q_\kk}, \ldots , y_{\kk, r_\kk}^{1/q_\kk}, y_{\kk,r_\kk+1}, \ldots , y_{\kk,d}).
\end{align}
Here $\si_m = \si_2 \o \si_{3,2} \o \cdots \o \si_{m,m-1}$, $\si_{d,m} = \si_{m+1,m} \o \cdots \o \si_{d,d-1}$, 
$f_m \in \sD_m$ is a local generator of $\si_m^*(\sD_m)$, $\psi_\kk$ is a convergent power series, 
$q_m \in \N_{\ge 1}$, $\al_m \in \frac{1}{q_m} \N_{\ge 1}$,   
and 
$y_{\kk,i}$ is a local system of coordinates so that $f_m^{-1}(0)$ is given by $y_{\kk,1} y_{\kk,2} \cdots y_{\kk,r_\kk}=0$.   
These formulas reduce the problem to radicals and then one may use  Theorem~\ref{thm:GG}.   

The theorem is first proven for, sufficiently generic, real analytic curves of polynomials.
Let $a(t)$ be such a curve, for genericity it suffices that the discriminant of $P_{a(t)}$ is not identically equal to zero. 
Note that $a(t)$  can be uniquely lifted over the sequence of blowings-up \eqref{eq:modifications} 
by the universal property of blowing-up.  The main difficulty is to get uniform bounds for  the norms independent of the choice of the curve.   

If   $a(t)$ is arbitrary, then by the Weierstrass approximation theorem it can be approximated by a sequence of  
polynomial curves $(a_\nu) \subseteq C^\om(\ol I,\C^d)$,   such that 
\[
    a_\nu \to a \quad \text{ in } C^{k}(\ol I,\C^d)  \quad  \text{ as } \nu \to \infty. 
\]
Then, thanks to the uniformity of the bounds,  one concludes by the Arzel\'a--Ascoli theorem,  or alternatively by the   
Rellich--Kondrachov compactness theorem, by passing to a subsequence.

\subsection{Multiparameter case}

Though Theorems \ref{thm:GG}, \ref{thm:optimal}, \ref{thm:byresolution}  are one dimensional, 
they can be fairly easily extended to the multiparameter case by a standard sectioning argument.  
We recall the statement for the analogue of Theorem \ref{thm:optimal}.  For Theorem \ref{thm:GG} and Theorem \ref{thm:byresolution}, 
see \cite[Theorem 2.3]{GhisiGobbino13} and \cite[Multiparameter Theorem]{ParusinskiRainerAC}, respectively.

\begin{theorem} [{\cite[Theorem 2]{ParusinskiRainerOpt}}] \label{thm:multioptimal}
    Let $U \subseteq \R^n$ be open and let  
    \begin{equation} \label{polynomialmulti}
        P_a(x)(Z)= Z^d + \sum_{j=1}^d a_j(x) Z^{d-j}, \quad x \in U,
    \end{equation}
    be a monic polynomial with complex valued coefficients $a_j \in C^{d-1,1}(U)$, $j = 1,\ldots,d$.   
    Let $\la \in C^0(V)$ be a root of $P_a$ on a relatively compact open subset $V \Subset U$. 
    Then $\la$ belongs to the Sobolev space $W^{1,p}(V)$ for every $1 \le p < d/(d-1)$. 
    The distributional gradient $\nabla \la$ satisfies  
    \begin{equation} \label{multbound} 
        \|\nabla \la \|_{L^p(V)}  \le  C(d,n,p,\cK ) \max_{1 \le j \le d} \|a_j\|^{1/j}_{C^{d-1,1}(\overline W)},
    \end{equation}
    where $\cK$ is any finite cover of $\overline V$ by open boxes $\prod_{i=1}^n (\al_i,\be_i)$ contained 
    in $U$ and $W = \bigcup \cK$; 
    the constant $C(d,n,p,\cK )$ depends only on $d$, $n$, $p$, and the cover $\cK$.   
\end{theorem}

There is a simpler version of the above theorem, avoiding covers by open boxes in the statement, valid for bounded Lipschitz domains, see 
\cite[Theorem A.1]{ParusinskiRainerBV}.  

\subsubsection{Bounded variation of the roots}\label{sec:bvariation}

There is nevertheless an important difference between the one parameter and the multiparameter case.  
Recall that by a theorem of Kato   \cite[II.5.2]{Kato76}  
a continuous family of complex monic polynomials depending on one real parameter always admits a continuous choice of roots. This is 
no longer true if the dimension of the parameter space is at least two. In that case, monodromy may 
prevent the existence of continuous roots.  For instance, the radical $\R^2 = \C \ni (x+iy) \mapsto (x+iy)^{1/d}$, $d>1$, 
does not admit  continuous parameterizations near the origin.   

Functions of bounded variation ($BV$) are integrable functions whose distributional derivative is a vector valued finite Radon measure (cf.\ Section \ref{sec:ABV}). 
They form an algebra of discontinuous functions. 
Due to their ability to deal with discontinuities, they are widely used in the applied sciences, 
see e.g.\ \cite{KhudyaevVol'pert85}.  It is shown by Parusi{\'n}ski and Rainer in \cite{ParusinskiRainerBV} that 
the roots of a polynomial \eqref{polynomialmulti} with coefficients in a differentiability class of 
sufficiently high order can   be represented by functions of bounded variation.

\begin{theorem}[{\cite[Theorem 1.1]{ParusinskiRainerBV}}] \label{thm:BVroots}
    For all integers $d,n\ge 2$ there exists an integer $k=k(d,n) \ge \max \{d,n\}$ such that the following holds.
    Let $\Om \subseteq \R^n$ be a bounded Lipschitz domain and let \eqref{polynomialmulti} 
    be a monic polynomial of degree $d$ with complex valued coefficients
    $a=(a_1,\ldots,a_d)  \in C^{k-1,1}(\ol \Om,\C^d)$. 

    Then the roots of $P_a$ admit a parameterization $\la = (\la_1,\ldots,\la_d)$ by 
    special functions of bounded variation ($SBV$) on $\Om$ such that
    \begin{equation} \label{mainbound}
        \|\la\|_{BV(\Om)} \le C(d,n,\Om)\, \max\big\{1, \| a\|_{L^\infty(\Om)}\big\} \max 
        \big\{1 , \| a\|_{C^{k-1,1}(\ol \Om)}\big\}. 
    \end{equation}
    There is a finite collection of $C^{k-1}$ hypersurfaces $E_j$ in $\Om$ 
    such that $\la$ is continuous in the complement of $E :=\bigcup_j E_j$. 
    Any hypersurface $E_j$ is closed in an open subset of $\Om$ 
    but possibly not in $\Om$ itself. All discontinuities of $\la$ are jump discontinuities. 
\end{theorem}

The main difficulty of the problem is to make a good choice of the discontinuity set $E$ of the roots.
On the complement of $E$, the roots are of optimal Sobolev class $W^{1,p}$, 
for all $1 \le p < \frac{d}{d-1}$,   
by  \cite{ParusinskiRainerOpt}, see Theorem \ref{thm:multioptimal}. 
In general, it is not possible to choose the discontinuity set $E$ of finite codimension one Hausdorff measure $\cH^{n-1}$, 
\cite[Example 2.4]{ParusinskiRainerBV}. 
Thus, in order to have bounded variation it is crucial that the jump height of a selection of a root is 
integrable (with respect to $\cH^{n-1}$) along its discontinuity set, see also the end of Section \ref{sec:applications1}.

The proof is based on the radical case, see \cite[Theorem 1.4]{ParusinskiRainerBV}, which is significantly simpler,  
and on the formulas \eqref{eq:formulasforroots} for the roots of the universal polynomial $P_a$, $a \in \C^d$, obtained in 
\cite{ParusinskiRainerAC}.  
The result of Ghisi and Gobbino \cite{GhisiGobbino13} is used in several places.   
Interestingly, the method of \cite{ParusinskiRainerAC} seems to be better suited for the control of 
the discontinuities and integrability along them than a more elementary method of \cite{ParusinskiRainerOpt}.

\section{Lifting maps over invariants of group representations} \label{sec:lifting}

\subsection{A reformulation of the regularity problem for hyperbolic polynomials} \label{sec:repLreform}

The regularity problem for the roots of hyperbolic polynomials can be seen as a special lifting problem:
let the symmetric group $\on{S}_d$ act on $\R^d$ by permuting the coordinates.
The algebra of invariant polynomials $\R[\R^d]^{\on{S}_d}$ is generated by the elementary symmetric functions\index{elementary symmetric functions}\index{polynomial!elementary symmetric} 
\[
    \si_i  = \sum_{j_1 < \cdots <j_i} x_{j_1} \cdots x_{j_i}, \quad 1 \le i \le d.
\]
Each point in the image of the map $\si =(\si_1,\ldots,\si_d) : \R^d \to \R^d$ 
represents, in view of Vieta's formulas, a monic hyperbolic polynomial $P_a$ of degree $d$, in a unique way. 
The fiber of $\si$ over this point is the orbit $\on{S}_d(\la_1,\ldots,\la_d) = \{(\la_{\ta(1)}, \ldots,\la_{\ta(d)}) : \ta \in \on{S}_d\}$, where 
$\la_1,\ldots,\la_d$ are the roots of $P_a$ (with multiplicities). Thus $\si(\R^d) = \Hyp(d)$.

In this picture, a family of monic hyperbolic polynomials $P_a(x)$, $x \in U \subseteq \R^n$, is a map
$p : U \to \R^d$ with image contained in $\si(\R^d)$ and a system of roots for the family $P_a$ 
is a lifting $\ol p : U \to \R^d$ of $p$ over $\si$, i.e., $\si \o \ol p = p$.\index{lifting}
The regularity problem can be rephrased as follows:
Given that $p$ has some regularity, how regular can a lifting of $p$ be chosen?
\[
  \xymatrix{
     && \R^d \ar[d]^{\si}    \\
      U \ar@{-->}[rru]^{\ol p} \ar[rr]_p && \si(\R^d) 
  } 
\]

This setup can be generalized considerably and we will see that Bronshtein's theorem \ref{thm:Bronshtein} and many other results 
on the perturbation theory of hyperbolic polynomials hold in greater generality. 

We shall also discuss a generalization of the nonhyperbolic case along the same lines in Section \ref{sec:reprpoly} 
and subsequent sections.

\subsection{Orthogonal representations of compact Lie groups} \label{sec:repLreal}

Let $G$ be a compact Lie group and let $\rho : G \rightarrow \on{O}(V)$ be an 
orthogonal representation in a real finite dimensional Euclidean 
vector space $V$ with inner product $\langle \cdot \mid \cdot \rangle$. We will often just write $G \acts V$.\index{GV@$G \acts V$}
In particular, $G$ may be a finite group. 
By a classical theorem of Hilbert and Nagata, 
the algebra $\mathbb{R}[V]^{G}$\index{RVG@$\mathbb{R}[V]^{G}$} of invariant polynomials on $V$ 
is finitely generated. 
So let $\{\si_i\}_{i=1}^n$ be a system of homogeneous generators 
of $\mathbb{R}[V]^{G}$ and call it a system of \emph{basic invariants}.\index{basic invariants}

A system of basic invariants $\{\si_i\}_{i=1}^n$
is called \emph{minimal}\index{basic invariants!minimal system of} if there is no polynomial relation of the 
form $\si_i = P(\si_1,\ldots,\widehat{\si_i},\ldots,\si_n)$, or equivalently, 
$\{\si_i\}_{i=1}^n$ induces a basis of the real vector space $\R[V]^G_+/(\R[V]^G_+)^2$, 
where $\R[V]^G_+ = \{f \in \R[V]^G : f(0)=0\}$; cf.\ \cite[Section 3.6]{DK02}.   
The elements in a minimal system of basic invariants may not be unique, but their number and their degrees 
$d_i := \deg \si_i$ are unique. Let us set
\[
  d:= \max_{1 \le i \le n} d_i
\]
and consider the \emph{orbit map}\index{orbit map}
$\sigma = (\sigma_1,\ldots,\sigma_n) : V \rightarrow \mathbb{R}^n$. 
The image $\sigma(V)$ is a semialgebraic set in the categorical quotient\index{categorical quotient} 
$V /\!\!/ G := \{y \in \mathbb{R}^n : P(y) = 0 ~\mbox{for all}~ P \in \sI\}$,\index{VG@$V \cq G$}  
where $\sI$ is the ideal of relations between $\sigma_1,\ldots,\sigma_n$. 
Since $G$ is compact, $\sigma$ is proper and separates orbits of $G$, and 
it thus induces a homeomorphism between the orbit space $V/G$\index{VG@$V/G$}\index{orbit space} and $\sigma(V)$.

Note that, by the differentiable slice theorem, $V/G$ is a local model for the orbit space of differentiable $G$-manifolds.

The given data, i.e., the representation
$G \acts V$, a fixed system of basic invariants $\{\si_i\}_{i=1}^n$ with degrees $d_i =\deg \si_i$, 
and the corresponding 
orbit map $\si$, will be abbreviated by the 
triple $\rep$.\index{GVdsigma@$\rep$}  

In this setting, we may consider the following lifting problem:
let $f : U \to \R^n$, $U \subseteq \R^m$, be a map with some regularity ($C^\om$, $C^\infty$, $C^p$, etc.)
and image contained in $\si(V)$; for brevity we will write $f \in \sC(f,\si(V))$, where $\sC$ is the respective 
regularity class.
How regular can a lifting\index{lifting} $\ol f : U \to V$ of $f$ be? 
\[
  \xymatrix{
     && V \ar[d]^{\si}    \\
      U \ar@{-->}[rru]^{\ol f} \ar[rr]_f && \si(V) 
  } 
\]
Note that any two systems of basic invariants differ by a polynomial diffeomorphism. 
Thus this question is independent of the choice of $\{\si_i\}_{i=1}^n$ and hence of $\si$.

\subsubsection{Differentiable lifting}

The results known for this lifting problem generalize the results for hyperbolic polynomials.
The role of the polynomial degree (as an order of differentiablity of the coefficients needed to guarantee 
certain regularity properties of the roots) is here 
played by $d$, the maximal degree of a minimal system of basic invariants.

\begin{theorem}[{\cite[Proposition 3.1 and Theorem 4.4]{KLMR05}}]
   Let $\rep$ be a real finite dimensional orthogonal representation of a compact Lie group.
   Let $I \subseteq \R$ be an open interval. Then:
   \begin{enumerate}
       \item Each continuous curve $c : I \to \si(V)$ has a continuous lifting $\ol c : I \to V$.
       \item Each $C^d$ curve $c : I \to \si(V)$ has a differentiable lifting $\ol c : I \to V$.
   \end{enumerate}
\end{theorem}

In Parusi\'nski and Rainer \cite{ParusinskiRainer14}, the proof of Bronshtein's theorem presented in Section \ref{sec:Bronshteinproof}
was generalized and led to the following essentially optimal results. 

The next theorem is the analogue of Bronshtein's theorem \ref{thm:Bronshtein}.

\begin{theorem}[{\cite[Theorem 1]{ParusinskiRainer14}}] \label{thm:repLLip}
   Let $\rep$ be a real finite dimensional orthogonal representation of a compact Lie group.
   Let $I \subseteq \R$ be an open interval.
   Each curve $c \in C^{d-1,1}(I,\si(V))$ has a lifting $\ol c \in C^{0,1}(I,V)$ such that, 
   for any pair of relatively compact open intervals $I_0 \Subset I_1 \Subset I$, 
   \begin{equation} \label{eq:repLLip}
       \Lip_{I_0}(\ol c) \le C\, \max_{1 \le i \le n} \|c_i\|_{C^{d-1,1}(\ol I_1)}^{1/d_i}.
   \end{equation}
    The constant $C>0$ depends only on $I_0$, $I_1$, and  the isomorphism classes of the representation $G \acts V$ 
    and respective minimal systems of basic invariants.
\end{theorem}

In analogy to Theorem \ref{thm:C1roots}, we have

\begin{theorem}[{\cite[Theorem 2]{ParusinskiRainer14}}] \label{thm:repLC1}
   Let $\rep$ be a real finite dimensional orthogonal representation of a compact Lie group.
   Let $I \subseteq \R$ be an open interval.
   Each curve $c \in C^{d}(I,\si(V))$ has a lifting $\ol c \in C^{1}(I,V)$.
\end{theorem}

For finite groups and polar representations, a precursor of this theorem was obtained by \cite{KLMR06,KLMRadd} 
(by reduction to the case $\on{S}_d \acts \R^d$ and thus to Bronshtein's theorem).

An orthogonal representation $G \acts V$ is called \emph{polar},\index{polar representation}
if there exists a linear subspace $\Si \subseteq V$, 
called a \emph{section}\index{section} or a \emph{Cartan subspace},\index{Cartan subspace}
which meets each orbit orthogonally; cf.\ \cite{Dadok85}, \cite{DK85}. 
The trace of the $G$-action on $\Si$ is the action of the \emph{generalized Weyl group}\index{Weyl group} 
$W(\Si) = N_G(\Si)/Z_G(\Si)$ on $\Si$, where 
$N_G(\Si) := \{g \in G : g\Si = \Si\}$ and 
$Z_G(\Si) := \{g \in G : gs = s \text{ for all } s \in \Si\}$. 
This group is finite, and it is a reflection group if 
$G$ is connected. The algebras $\R[V]^G$ and $\R[\Si]^{W(\Si)}$ are isomorphic via restriction, 
by a generalization of Chevalley's restriction theorem, due to \cite{DK85} and  
independently \cite{Terng85}, 
and thus the orbit spaces $V/G$ and $\Si/W(\Si)$ are isomorphic. 
Consequently, the lifting problem can be reduced to a finite group action.

For finite group actions, Theorem \ref{thm:repLLip} and Theorem \ref{thm:repLC1} can be sharpened.

\begin{theorem}[{\cite[Corollary 1 and 3]{ParusinskiRainer14}}]
   Let $\rep$ be a real finite dimensional orthogonal representation of a finite group.
   Let $I \subseteq \R$ be an open interval. Then:
   \begin{enumerate}
       \item Any continuous lifting $\ol c$ of $c \in C^{d-1,1}(I, \si(V))$ is locally Lipschitz and satisfies \eqref{eq:repLLip}.
       \item Any differentiable lifting $\ol c$ of $c \in C^{d}(I, \si(V))$ is $C^1$.
       \item Each curve $c \in C^{2d}(I, \si(V))$ has a $C^1$ lifting whose second order derivatives exist everywhere.
   \end{enumerate}
\end{theorem}

If $G$ is a compact Lie group of positive dimension, then a Lipschitz (or a $C^1$) lifting $\ol c$ of $c$ can be distorted to a 
continuous lifting that is not locally Lipschitz simply by taking $g\ol c$, where $g : I \to G$ is a suitable continuous curve.

As a consequence of the uniformity of Theorem \ref{thm:repLLip}, we obtain a lifting result for maps in several variables.

\begin{theorem}[{\cite[Corollary 2]{ParusinskiRainer14}}] \label{thm:repLsev}
   Let $\rep$ be a real finite dimensional orthogonal representation of a finite group.
   Let $U \subseteq \R^m$ be open and $f \in C^{d-1,1}(U,\si(V))$.
   If $\ol f \in C^0(\Om,V)$ is a continuous lifting of $f$ on an open subset
   $\Om \subseteq U$,
   then $\ol f$ is locally Lipschitz and, for any pair of relatively compact 
   open subsets $\Om_0 \Subset \Om_1 \Subset \Om$,
   \begin{equation}
       \Lip_{\Om_0}(\ol f) \le C\, \max_{1\le i\le n} \|f_i\|_{C^{d-1,1}(\ol \Om_1)}^{1/d_i}.
   \end{equation}
    The constant $C>0$ depends only on $\Om_0$, $\Om_1$, $m$, and on the isomorphism classes of the representation $G \acts V$ 
    and respective minimal systems of basic invariants.
\end{theorem}

Not every representation admits unrestricted continuous lifting.
    For instance, the orbit space of a finite rotation group acting in the standard way on $\R^2$ 
    is homeomorphic to the set $C$ obtained from the sector 
    $\{r e^{i\vh} \in \C : r \in [0,\infty), 0 \le \vh \le \vh_0\}$ 
    by identifying the rays that constitute its boundary. 
    A loop on $C$ cannot be lifted to a loop in $\R^2$ unless it is homotopically trivial in $C \setminus \{0\}$.
    Reflection groups admit continuous lifting defined everywhere, e.g., the hyperbolic case $\on{S}_d \acts \R^d$.

In view of the examples for hyperbolic polynomials, 
these lifting results are generally optimal.

    \subsubsection{Real analytic lifting} \label{sec:repLra}
    
Real analytic curves can always be lifted.

\begin{theorem}[{\cite[Lemma 3.8 and Theorem 4.2]{AKLM00}, \cite[Theorem 4]{ParusinskiRainer14}}]
   Let $\rep$ be a real finite dimensional orthogonal representation of a compact Lie group.
   Let $I \subseteq \R$ be an open interval.
   Each curve $c \in C^\om(I,\si(V))$ has a lifting $\ol c \in C^\om(I,V)$.
\end{theorem}

That there always exist local real analytic liftings was shown in \cite{AKLM00}. 
As observed in \cite{ParusinskiRainer14}, 
the local liftings can be glued to a global real analytic lifting, since the first \v{C}ech 
cohomology group $\check H^1(I,G^a)$ vanishes, where $G^a$ denotes the sheaf of real analytic maps $I \supseteq U \to G$, 
by \cite{Tognoli69}.

Similarly, one sees that \emph{generic} $C^\infty$ curves admit liftings.
The genericity condition is a generalization of normal nonflatness considered for polynomials in 
Section \ref{sec:nonflathyp}:
for integer $s\ge 0$ let $A_s$ be the union of all orbit type strata $S$ of $V/G$ (cf.\ Section \ref{sec:repLstrat}) of dimension $\dim S\le s$ 
and let $I_s \subseteq \R[V]^G$ be the ideal of invariant polynomials vanishing on $A_{s-1}$. 
Now $c \in C^\infty(I,\si(V))$ is called \emph{normally nonflat}\index{normally nonflat} 
if, for each $t_0 \in I$,
there exists $f \in I_s$ such that the Taylor series of $f \o c$ at $t_0$ is not zero, 
where $s$ is minimal with the property that the germ of $c$ at $t_0$ is contained in $A_s$.  
Roughly speaking, it means that lower dimensional orbit type strata are not met with infinite 
order of flatness.

\begin{theorem}[{\cite[Theorem 4.1]{AKLM00}}]
   Let $\rep$ be a real finite dimensional orthogonal representation of a compact Lie group.
   Let $I \subseteq \R$ be an open interval.
   Each normally nonflat curve $c \in C^\infty(I,\si(V))$ has a lifting $\ol c \in C^\infty(I,V)$.
\end{theorem}

In the $C^\infty$ case, the gluing of local liftings is much easier; cf.\ \cite[Lemma 3.8]{AKLM00}.

Real analytic maps depending on several variables cannot be lifted in general,
but they admit liftings after composition with local blowings-up\index{local blowings-up} 
(i.e., blowings-up over an open subset composed with the inclusion of this subset).

\begin{theorem}[{\cite[Theorem 5.4]{RainerRG}}] \label{thm:repLblow}
   Let $\rep$ be a real finite dimensional orthogonal representation of a compact Lie group.
    Let $M$ be a real analytic manifold and $f \in C^\om(M,\si(V))$.
    For any compact subset $K \subseteq M$ there exist a neighborhood $U$ of $K$ and 
    a finite covering $\{\pi_j : U_j \to U\}$ of $U$, where each $\pi_j$ is a composite 
    of finitely many local blowings-up with smooth center,
    such that, for all $j$, the map $f \o \pi_j$ has a real analytic lifting on $U_j$.
\end{theorem}

Actually, this continues to hold in suitable quasianalytic classes of $C^\infty$ functions (instead of $C^\om$); cf.\ \cite{RainerRG}.

\subsection{The main tools} \label{sec:repLtools}
Let us briefly discuss the main tools used for the lifting problem 
and identify their counterpart in the study of hyperbolic polynomials.

\subsubsection{Removing fixed points}
Let $V^G := \{v \in V : Gv=v\}$\index{VG@$V^G$} be the linear  subspace of fixed points\index{fixed points}
and let $V'$ be its orthogonal complement in $V$. 
Then $V= V^G \oplus V'$, $\R[V]^G = \R[V^G] \otimes \R[V']^G$, and $V/G = V^G \times V'/G$.
Thus it suffices to study the lifting problem for $G \acts V'$ or equivalently to assume $V^G=\{0\}$. 
This corresponds to the Tschirnhausen transformation (see Section \ref{sec:Tschirnhausen}); indeed, 
$(\R^d)^{\on{S}_d} = \{x \in \R^d : x_1 = x_2 = \cdots = x_d\}$ and 
$(\R^d)' = \{x \in \R^d : x_1 + x_2 + \cdots + x_d =0\}$.

\subsubsection{Dominant invariant}
We may assume that $v \mapsto \langle v \mid v \rangle =\|v\|^2$ belongs to the system of basic invariants $\{\si_i\}_{i=1}^n$, 
say it is $\si_1$. This does not change $d$; in fact, $V^G = \{0\}$ implies $d\ge 2$. 
Then $\si_1$ is dominant in the sense that 
\begin{equation*}
    |\si_i|^{1/d_i} \le C(\si)\, |\si_1|^{1/d_1}, \quad 1 \le i \le n,
\end{equation*}
by homogeneity. This corresponds to the dominance of the second coefficient of monic hyperbolic polynomials in Tschirnhausen form (see Lemma \ref{lem:dominant}).

\subsubsection{The slice theorem} \index{slice theorem}\index{theorem!slice}
For $v \in V$, let $N_v:= T_v(Gv)^{\bot}$ be the normal subspace of the orbit $Gv$ at $v$.
It carries a natural action $G_v \acts N_v$ of the \emph{isotropy subgroup}\index{isotropy subgroup} $G_v := \{g \in G: gv=v\}$.
The crossed product (or associated bundle) $G \times_{G_v} N_v$ carries the structure of an affine real algebraic 
variety as the categorical (and geometric) quotient $G \times N_v \cq G_v$ 
with respect to the action $G_v \acts (G \times N_v)$, $h(g,x):= (gh^{-1},hx)$.
The $G$-equivariant polynomial mapping $\ph : G \times_{G_v} N_v \to V$, $[g,x] \mapsto g(v+x)$, 
induces a polynomial mapping $\ps : (G \times_{G_v} N_v) \cq G \to V\cq G$ 
sending $(G \times_{G_v} N_v)/G$ into $V/G$.
The $G_v$-equivariant embedding $\al : N_v \hookrightarrow G \times_{G_v} N_v$, $x \mapsto [e,x]$, 
induces an isomorphism $\be : N_v\cq G_v \to (G \times_{G_v} N_v)\cq G$
mapping $N_v/G_v$ onto $(G \times_{G_v} N_v)/G$.
\[
  \xymatrix{
    N_v \ar@{_(->}[rr]^{\al} \ar[d]_{\ta} \ar@/^20pt/[rrrr]^{\et}
    && G \times_{G_v} N_v \ar[rr]^{\ph} \ar[d] && V \ar[d]^{\si} \\
    N_v/G_v \ar[rr] \ar@{^(->}[d] && (G \times_{G_v} N_v)/G \ar@{^(->}[d] \ar[rr] && V/G \ar@{^(->}[d] \\ 
    N_v \cq G_v \ar[rr]^{\be}  \ar@/_20pt/[rrrr]_{\th}
    && (G \times_{G_v} N_v)\cq G  \ar[rr]^{\ps} && V \cq G 
  } 
\]

\begin{theorem}[{\cite{Luna73}, \cite{Schwarz80}}] \label{thm:slice}
  There is an open ball $B_v \subseteq N_v$ centered at the origin such that the restriction of $\ph$ to 
  $G \times_{G_v} B_v$ is an analytic $G$-isomorphism onto a $G$-invariant neighborhood of $v$ in $V$. 
  The mapping $\th$ is a local analytic isomorphism at $0$ which induces a local homeomorphism of 
  $N_v/G_v$ and $V/G$.
\end{theorem}

The slice theorem allows to reduce the lifting problem locally to 
the slice representation $G_v \acts N_v$.\index{slice representation}
This reduction replaces the splitting principle (see Lemma \ref{lem:splitting} and Section \ref{sec:splittinghyp}) for polynomials.
(The passage to the slice representation at $v = (\la_1,\ldots,\la_d)$ for the standard action $\on{S}_d \acts \R^d$ 
corresponds to a \emph{full} splitting $P_{\tilde a} = P_{b^1} P_{b^2} \cdots P_{b^k}$, where each factor $P_{b^i}$
corresponds to precisely one of the distinct elements among the $\la_1,\ldots,\la_d$ and 
$\deg P_{b^i}$ is its multiplicity.) 

The reduction and the assumption $V^G = \{0\}$ allow to proceed by induction on the ``size'' of $G$:
if $H$ and $G$ are compact Lie groups, then $H$ is \emph{smaller} than $G$ if either $\dim H < \dim G$ 
or, given that $\dim H= \dim G$, $H$ has fewer connected components than $G$. 
For finite groups this reduces to induction on the group order.

\subsubsection{Orbit type stratification} \label{sec:repLstrat}

The conjugacy class $(H)$ of $H=G_v$ in $G$ is called  
the \emph{type}\index{orbit type} of the orbit $Gv$. Let $V_{(H)}$ be the union of all
orbits of type $(H)$. 
The collection of the connected components of the smooth manifolds $V_{(H)}/G$ 
forms the \emph{orbit type stratification of $V/G$}; cf.\ \cite{Schwarz80}.\index{orbit type stratification}  
By \cite{Bierstone75}, 
there is an identification between the orbit type stratification of $V/G$ 
and the natural stratification of $\si(V)$ as a semialgebraic set 
(for the construction of the latter see \cite[p.~246]{Bierstone75} or \cite[Section 25]{Lojasiewicz65}); it is analytically locally trivial
and thus satisfies Whitney's conditions $(a)$ and $(b)$.
The inclusion relation on the set of subgroups of $G$ induces a partial ordering on the family of orbit types. 
There is a unique minimal orbit type, the \emph{principal orbit type},\index{orbit type!principal} corresponding 
to the open and dense submanifold $V_{\on{reg}}$ 
consisting of points $v$, where the slice 
representation $G_v \acts N_v$ is trivial.
The projection $V_{\on{reg}} \to V_{\on{reg}}/G$ is a locally trivial fiber bundle.
There are only finitely many isomorphism classes of slice representations.

\subsection{Examples and applications}

\subsubsection{Differentiable eigenvalues of real symmetric matrices} \label{sec:partiallifting}
Let the orthogonal group $\on{O}(d) = \on{O}(\R^d)$ act by conjugation on the real vector space $\on{Sym}(d)$\index{Sym@$\on{Sym}(d)$} of real 
symmetric $d \times d$ matrices, 
\[
    \on{O}(d) \times \on{Sym}(d) \ni (S,A) \mapsto SAS^{-1}=SAS^t \in \on{Sym}(d). 
\]
    The algebra of invariant polynomials 
    $\R[\on{Sym}(d)]^{\on{O}(d)}$ is isomorphic to $\R[\on{Diag}(d)]^{\on{S}_d}$ by restriction, 
    where $\on{Diag}(d)$\index{Diag@$\on{Diag}(d)$} is the vector space 
of real diagonal $d \times d$ matrices upon which $\on{S}_d$ acts by permuting the diagonal entries. 
More precisely, $\R[\on{Sym}(d)]^{\on{O}(d)} = \R[\Si_1,\ldots,\Si_d]$, where
\[
    \Si_i(A) = \textstyle{\on{Trace}(\bigwedge^i A : \bigwedge^i \R^d \to \bigwedge^i \R^d)}
\] 
is the $i$-th
characteristic coefficient of $A$ and $\Si_i|_{\on{Diag}(d)} = \si_i$, where $\si_i$ is the $i$-th 
elementary symmetric polynomial and we identify $\on{Diag}(d) \cong \R^d$ (cf.\ \cite[7.1]{MichorH}). 
This means that 
the representation $\on{O}(d) \acts \on{Sym}(d)$ is polar and $\on{Diag}(d)$ forms a section.  

Let $A : \R \to \on{Sym}(d)$ be a curve of symmetric matrices of some regularity.
The characteristic polynomial $\ch_A$ is a curve of monic hyperbolic polynomials of the same regularity. 
In this case, as we already know from 
Section \ref{sec:Hermitian}, the regularity results provided by the general theory for hyperbolic polynomials are not optimal. 

A heuristic explanation is that in this problem only a ``partial'' lifting is necessary: 
the curve $\ch_A$ in the orbit space is the projection of the curve $A$ 
under $\on{Sym}(d) \to \on{Sym}(d)/\on{O}(d)$
and is then lifted over $\on{Diag}(d) \to \on{Diag}(d)/\on{S}_d$. 
\[
  \xymatrix{
     && & \R \ar[dl]_{A} \ar[dd]_{\ch_A} \ar@/_20pt/@{-->}[dlll] & \\
    \on{Diag}(d) \ar@{^{(}->}[rr] \ar[d] && \on{Sym}(d) \ar[d] &&   \\
    \on{Diag}(d)/\on{S}_d \ar@{=}[rr] && \on{Sym}(d)/\on{O}(d) \ar@{=}[r] & \si(\R^d) \ar@{^{(}->}[r] & \R^d   
  } 
\]

\subsubsection{Differentiable decomposition of nonnegative functions into sums of squares}
Let the orthogonal group $\on{O}(n)$ act in the standard way on $\R^n$. The algebra of invariant polynomials 
$\R[\R^n]^{\on{O}(n)}$ is generated by $\si = \sum_{i=1}^n x_i^2$; note that $d=2$. 
The orbit space $\R^n/\on{O}(n)$ can be 
identified with the half-line $[0,\infty) = \si(\R^n)$. Each line through the origin of $\R^n$ forms a
section of $\on{O}(n) \acts \R^n$.

Lifting a nonnegative function $f$ over $\si$ means decomposing $f$ into a sum of $n$ squares. 
We conclude (from Theorem \ref{thm:repLsev}) that 
any nonnegative $C^{1,1}$ function $f : \R^m \to [0,\infty)$ is the square of a $C^{0,1}$ function $g$.
The image of the lifting $(g,0,\ldots,0)$ lies in the section $\R (1,0,\ldots,0)$ of $\on{O}(n) \acts \R^n$. 

There are stronger results that benefit from the 
additionally 
available space:

\begin{theorem}[{\cite{FP78}}]
    Any nonnegative $C^{3,1}$ function $f : \R^m \to [0,\infty)$ is a sum of $n=n(m)$ squares of $C^{1,1}$ functions.
\end{theorem}
 
This is sharp in the sense that there exist 
$C^\infty$ functions $f : \R^m \to [0,\infty)$,
for $m\ge 4$, 
that are not sums of squares of $C^2$ functions; see \cite{BBCP06}.
In fact, a real nonnegative homogeneous polynomial $p$ of degree $2d$ on $\R^m$ that is a sum of squares  
of $C^d$ functions $f_i$ is necessarily a sum of squares of polynomials, since 
Taylor expansion of $f_i$ gives $f_i(x) = q_i(x) + o(|x|^d)$ at $x=0$, where $q_i$ is a homogeneous polynomial 
of degree $d$, so that $p = \sum_i q_i^2$. But there exist nonnegative homogeneous polynomials 
of degree $4$ on $\R^4$ that are not sums of squares of polynomials; cf.\ \cite[Section 6.3]{BCR98}. 
For the same reason, there are nonnegative $C^\infty$ functions on $\R^3$ that are not 
sums of squares of $C^3$ functions. 

In dimension $m=1$ (where there are no algebraic obstructions) there is the following result.

\begin{theorem}[{\cite{Bony05}}]
    Let $p \in \N$. Any nonnegative $C^{2p}$ function $f: \R \to [0,\infty)$ is the sum of two squares of $C^p$ functions. 
\end{theorem}

The decomposition depends on $p$.

\subsubsection{Polynomials with symmetries}
In \cite{LR07} the lifting problem was used to improve upon the regularity problem for hyperbolic polynomials with symmetries.
The idea is the following.
A curve of monic hyperbolic polynomials of degree $d$ is a curve $p$ in the semialgebraic set $\Hyp(d) = \si(\R^d)\subseteq \R^d$, 
where $\si$ consists of the elementary symmetric functions $\si_i$.
If $p$ lies in some proper semialgebraic subset of $\Hyp(d)$, then 
$p$ has to fulfill more constraints and thus,
in general, the conditions that guarantee the existence of differentiable 
systems of the roots are weaker.
For instance,
if we know that the roots of $p$ fulfill some linear relations, i.e., they lie in some linear subspace $V$ of $\R^d$, 
then $p$ lies in $\si(V) \subseteq \Hyp(d)$. The symmetries of the roots are represented by the induced action $G \acts V$
of 
\[
    G := \frac{\{\ta \in \on{S}_d : \ta V = V\}}{\{\ta \in \on{S}_d : \ta v = v \text{ for all } v \in V\}}. 
\]
The solution of the lifting problem can be brought to bear, if the restrictions $\si_i|_V$, $1 \le i \le d$, generate $\R[V]^G$.
Indeed, the orbit type stratification of $G \acts V$ is generally coarser than the restriction to $V$ of the orbit type stratification of $\on{S}_d \acts \R^d$.

\subsection{A reformulation of the regularity problem for general polynomials}
\label{sec:reprpoly}

If we let the symmetric group $\on{S}_d$ act on $\C^d$, by permuting the coordinates, then 
again $\C[\R^d]^{\on{S}_d}$ is generated by the elementary symmetric functions $\si_i$, $1 \le i \le d$.
The map $\si = (\si_1,\ldots,\si_d) : \C^d \to \C^d$ is onto. Each point in $\C^d$ represents a monic polynomial 
of degree $d$ with complex coefficients and the $\si$-fiber over this point is the $\on{S}_d$-orbit through the 
$d$-tuple of its roots (with multiplicities).

A family of monic polynomials $P_a(x)$, $x \in U \subseteq \R^m$, is a map
$p : U \to \C^d$ and a system of roots for the family $P_a$ 
is a lifting $\ol p : U \to \C^d$ of $p$ over $\si$, i.e., $\si \o \ol p = p$.\index{lifting}
Given that $p$ is regular of some kind, how regular can a lifting of $p$ be?
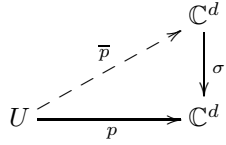
\[
  \xymatrix{
     && \C^d \ar[d]^{\si}    \\
      U \ar@{-->}[rru]^{\ol p} \ar[rr]_p && \C^d 
  } 
\]
Note that, in contrast to the hyperbolic case (cf.\ Section \ref{sec:repLreform}), 
there are no constraints for the map $p$, since $\si(\C^d) = \C^d$.

Let us see now how this problem can be studied in greater generality.
From now on, representation spaces $V$ will be complex and finite dimensional.

\subsection{Representations of linearly reductive groups}

An algebraic group $G$ is called \emph{linearly reductive}\index{linearly reductive group} if for each rational representation 
$V$ and each subrepresentation $W \subseteq V$ there is a subrepresentation $W' \subseteq V$ such that 
$V = W \oplus W'$. 

By Hilbert's finiteness theorem, for rational representations $V$ of linearly reductive groups $G$,
the algebra of $G$-invariant polynomials $\C[V]^G$\index{CVG@$\C[V]^G$} is finitely generated. Let $\{\si_i\}_{i=1}^n$ be a system of homogeneous generators
which we call \emph{basic invariants}.\index{basic invariants}\index{basic invariants!minimal system of} 
In a \emph{minimal} system of basic invariants the number $n$ of elements $\si_i$ and their degrees $d_i := \deg \si_i$ are uniquely determined; 
set 
\[
    d:= \max_{1 \le i \le n} d_i.
\]

The map $\si = (\si_1,\ldots,\si_n)  : V \to \si(V) \subseteq \C^n$ can be identified with the morphism $V \to V\cq G$ induced by the 
inclusion $\C[V]^G \to \C[V]$; the \emph{categorical quotient}\index{categorical quotient} $V\cq G$\index{VG@$V \cq G$} is the affine variety with 
coordinate ring $\C[V]^G$. In this setting, $V\cq G$ is generally not a \emph{geometric quotient}:\index{geometric quotient} the 
$G$-orbits in $V$ are not in a one-to-one correspondence with the points in $V\cq G$. In fact, 
for every point $z \in V\cq G$ there is a unique closed orbit in the fiber $\si^{-1}(z)$ which lies in the 
closure of every other orbit in this fiber. 
On the other hand, if $G$ is a finite group, then all $G$-orbits are closed and thus $V \cq G$ is a geometric quotient 
isomorphic to the orbit space $V/G$.\index{orbit space}\index{VG@$V/G$}

We write again $\rep$ for this setup.\index{GV@$G \acts V$}\index{GVdsigma@$\rep$} 
If $f : U \to \C^n$, $U \subseteq \R^m$, is a map with some regularity and image contained in 
$\si(V)$ (which in general is a proper subset of $\C^n$), 
we ask how regular a lifting $\ol f$ of $f$ can be. 
By a \emph{lifting}\index{lifting} of $f$ we mean a map $\ol f : U \to V$ such that $f = \si \o \ol f$ and, for all $x \in U$,  
$\ol f(x)$ lies in the unique closed orbit in the fiber $\si^{-1}(f(x))$.
\[
  \xymatrix{
     && V \ar[d]^{\si}    \\
      U \ar@{-->}[rru]^{\ol f} \ar[rr]_f && \si(V) 
  } 
\]
The problem is independent of the choice of the basic invariants, since any two choices differ by a polynomial diffeomorphism.

The main difference to the real problem introduced in Section \ref{sec:repLreal}
is that there is no invariant polynomial that dominates all the others. 
Even for finite $G$, where we can always choose an invariant Hermitian inner product on $V$  
(by averaging over $G$) and hence assume that the representation is unitary, the 
invariant form $v \mapsto \|v\|^2$ is not a member of $\C[V]^G$. 
This makes the complex case much more difficult and causes a larger loss of regularity. 

\subsubsection{Sobolev lifting}

Let us first assume that $G$ is a finite group.
Then each continuous curve $c : I \to  \si(V)$, where $I\subseteq \R$ is an interval, 
has a continuous lifting $\ol c : I \to V$, by \cite[Theorem 5.1]{LMRac}. 
In Parusi\'nski and Rainer \cite{Parusinski:2020ab},
the optimal Sobolev regularity for the roots of polynomials was extended to group representations. 
The following result generalizes Theorem \ref{thm:optimal}.

\begin{theorem}[{\cite[Theorem 1.1]{Parusinski:2020ab}}] \label{thm:repCW}
    Let $\rep$ be a complex finite dimensional representation of a finite group $G$.
  Let $c \in C^{d-1,1}([\al,\be],\si(V))$ be a curve     
  defined on an open bounded interval 
  $(\al,\be)$. 
  Then each continuous lifting $\overline c : (\al,\be) \to V$ of $c$ over $\si$ is absolutely continuous and 
  belongs to $W^{1,p}((\al,\be),V)$ 
  with
  \begin{equation} \label{eq:repCW} 
     \|\ol c'\|_{L^p((\al,\be))} \le 
      C\,   \max_{1 \le j \le n} \|c_j\|^{1/d_j}_{C^{d-1,1}([\al, \be])}        
    \end{equation}  
  for all $1 \le p < d/(d-1)$,
  where $C$ is a constant which depends only on the representation $G \acts V$, the length of the interval $(\al,\be)$, and $p$.
\end{theorem}

The bound \eqref{eq:repCW}, similarly as \eqref{bound}, is not scale invariant.

\begin{remark} \label{rem:repCW}
  Let us comment on how the constant $C$ in \eqref{eq:repCW} depends on the length of the interval $(\al,\be)$.

  (a)  In general, the constant $C$ is of the form
  \begin{equation*}
      C(G \acts V,p)\, \max\{1, (\be-\al)^{1/p}, (\be-\al)^{-1+1/p} \}.
     \end{equation*}   
  
  (b) If the curve $c$ starts, ends, or passes through $0$ (the most singular point in $\si(V)$), 
  then $C$ is of the form  
  \begin{equation} \label{eq:betterbound} 
      C(G \acts V,p)\, \max\{1, (\be-\al)^{1/p} \}. 
  \end{equation}
  
  (c) If the representation is coregular, then again the constant is of the form \eqref{eq:betterbound}. 
  A representation $G \acts V$ is called \emph{coregular}\index{coregular representation} if $\C[V]^G$ is isomorphic to a polynomial algebra, i.e., 
  there is a system of basic invariants without polynomial relations among them. By the Shephard--Todd--Chevalley theorem 
  (\cite{Shephard:1954aa}, \cite{Chevalley:1955aa}, \cite{Serre:1968aa}), 
  this is the case if and only if $G$ is generated by pseudoreflections\index{pseudoreflection} (i.e., invertible linear transformations 
  of finite order and fixed point space a hyperplane). 
  
  (d) The constant is also of the form \eqref{eq:betterbound}
  if the curve $c$ satisfies $c^{(j)}(\al) = c^{(j)}(\be)=0$ for all $j=1,\ldots,d-1$.   
\end{remark}

Note that, in view of the counterexamples for the regularity of the roots of polynomials,
the results of Theorem \ref{thm:repCW} (as most of the subsequent results) are optimal.

The lifting of mappings defined in open domains of dimension $m > 1$ essentially  
admits the same regularity as for curves, provided that continuous lifting is possible. 
But, in general, there are topological obstructions for continuous lifting.

\begin{theorem}[{\cite[Theorem 1.4]{Parusinski:2020ab}}] \label{thm:repCsev}
    Let $\rep$ be a complex finite dimensional representation of a finite group $G$.
    Let $f \in C^{d-1,1}(\ol \Om,\si(V))$, 
    where $\Om \subseteq \R^m$ is an open bounded box $\Om = I_1 \times \cdots \times I_m$. 
    Then each continuous lifting $\ol f : U \to V$ of $f$ over $\si$ 
    defined on an open subset $U \subseteq \Om$
    belongs to $W^{1,p}(U,V)$ 
  and satisfies
  \begin{equation} \label{eq:repCsev} 
    \|\nabla \ol f\|_{L^p(U)} \le 
      C(G \acts V,\Om,m,p)\,  \max_{1 \le j \le n} \|f_j\|^{1/d_j}_{C^{d-1,1}(\ol \Om)} 
  \end{equation}
  for all $1 \le p < d/(d-1)$.
\end{theorem}

The case $U = \Om$ is not excluded.  
It is clear that Theorem \ref{thm:repCsev} implies a version of the statement, where 
$\Om \subseteq \R^m$ is any bounded open set,
$U \Subset \Om$ is relatively compact open in $\Om$, 
and the constant also depends on $U$ (or more precisely on a cover of $U$ by boxes contained in $\Om$). 
Concerning a global result we have the following.

\begin{remark} \label{rem:repCsev} 
    If $G \acts V$ is coregular, then Theorem \ref{thm:repCsev} holds as stated for any bounded Lipschitz domain $\Om$; cf.\ \cite[Remark 1.5]{Parusinski:2020ab}.
\end{remark}

If there is no continuous lifting, it is natural to ask how bad the discontinuities can be:

\begin{open} \label{q:repCBV}
When continuous lifting is impossible, 
do there exist liftings of bounded variation?  
\end{open}

See Theorem \ref{thm:repCSBV} for a partial answer.

In the general case, for an infinite linearly reductive group $G$, it is not clear if a continuous curve in $\si(V)$ admits a continuous lifting to $V$. The 
notion of \emph{stability} in geometric invariant theory provides a remedy.
A point $v \in V$ is called \emph{stable}\index{stable point} if the orbit $Gv$ is closed and the isotropy group 
$G_v$ is finite. 
The subset $V^s \subseteq V$ of stable points is $G$-invariant and open in $V$, 
and its image $\si(V^s)$ is open in $V\cq G \cong \si(V)$ (cf.\ \cite[Proposition 5.15]{Mukai:2003aa}).  
The restriction $\si : V^s \to \si(V^s)$ of the map $\si$ provides a one-to-one correspondence 
between points in $\si(V^s) \cong V^s/G$ and $G$-orbits in $V^s$, that is $V^s/G$ is a geometric quotient.
A continuous curve $c : I \to \si(V^s)$ has a continuous lifting $\ol c : I \to V^s$,
cf.\ \cite[Lemma 1.6]{Parusinski:2020ab}.
Theorem \ref{thm:repCW} has the following corollary.

\begin{corollary}[{\cite[Theorem 1.7]{Parusinski:2020ab}}]
    Let $\rep$ be a rational complex finite dimensional representation of a linearly reductive group $G$.
    Let $c \in C^{d-1,1}([\al,\be],\si(V^s))$ be a curve defined on a compact interval $[\al,\be]$ with       
  $c([\al,\be]) \subseteq \si(V^s)$. 
  Then there exists an absolutely continuous lifting 
  $\overline c : [\al,\be] \to V^s$ of $c$ over $\si$ which 
  belongs to $W^{1,p}([\al,\be],V^s)$ 
  with
  \begin{equation} \label{eq:stable}
     \|\ol c'\|_{L^p([\al,\be])} \le 
      C(G \acts V,[\al,\be],p)\,   \max_{1 \le j \le n} \|c_j\|^{1/d_j}_{C^{d-1,1}([\al, \be])}   
    \end{equation}  
  for all $1 \le p < d/(d-1)$.  
\end{corollary}

For a mapping $f$ defined on a compact subset $K$ of $\R^m$ with $f(K) \subseteq \si(V^s)$ the situation is 
more complicated:
we can apply Theorem \ref{thm:repCsev} to the slice representations at any point $v \in V^s$,
but it is not clear if these local (and partial) lifts can be glued in a continuous fashion.

More can be said for polar representations (which include e.g.\ the adjoint actions). 
We encountered polar representations already in Section \ref{sec:repLreal} in the real setup.
Let us briefly discuss them in the framework of rational representations $V$ of linearly reductive groups $G$; 
we follow \cite{DK85}.
Let $v \in V$ be such that the orbit $Gv$ is closed and consider the linear subspace
$\Si_v =\{x \in V : \mathfrak g x \subseteq \mathfrak g v\}$, where $\mathfrak g$ is the Lie algebra of $G$. 
All orbits that intersect $\Si_v$ are closed, whence $\dim \Si_v \le \dim V \cq G$. 
The representation $G \acts V$ is said to be \emph{polar}\index{polar representation} if there exists $v \in V$ with closed orbit $Gv$
and $\dim \Si_v = \dim V \cq G$. Then $\Si_v$ is called a \emph{section}\index{section} or \emph{Cartan subspace}\index{Cartan subspace} of $V$. 
Any two sections are $G$-conjugate. Let us fix a section $\Si$. 
All closed orbits in $V$ intersect $\Si$. 
The \emph{Weyl group}\index{Weyl group} $W := N_G(\Si)/Z_G(\Si)$, where $N_G(\Si)= \{ g \in G : g \Si = \Si\}$ is the 
normalizer and $Z_G(\Si)= \{ g \in G : g x = x \text{ for all }x \in\Si\}$ is the centralizer of $\Si$ in $G$, 
is finite and the intersection of any closed $G$-orbit in $V$ with $\Si$ is precisely 
one $W$-orbit. 
The ring $\C[V]^G$ is isomorphic via restriction to the ring $\C[\Si]^W$.
If $G$ is connected, then $W$ is a pseudoreflection group and hence $\C[V]^G \cong \C[\Si]^W$ is a polynomial ring, 
by the Shephard--Todd--Chevalley theorem.

\begin{corollary}[{\cite[Theorem 1.8]{Parusinski:2020ab}}] \label{cor:repCpol}
    Let $\rep$ be a polar representation of a linearly reductive group $G$. Then:
  \begin{enumerate}
    \item Each curve $c \in C^{d-1,1}([\al,\be],\si(V))$ has an absolutely continuous lifting $\ol c : (\al,\be) \to V$ of $c$ 
  over $\si$ which belongs to $W^{1,p}((\al,\be),V)$ for all $1 \le p <d/(d-1)$ and satisfies \eqref{eq:repCW}.
    \item Let $f \in C^{d-1,1}(\ol \Om,\si(V))$, 
    where $\Om \subseteq \R^m$ is an open bounded box $\Om = I_1 \times \cdots \times I_m$.  
    Each continuous lifting $\ol f$ defined on an open subset $U\subseteq \Om$ with values in a section $\Si$
    is of class $W^{1,p}$ on $U$ for all $1 \le p <d/(d-1)$ and satisfies \eqref{eq:repCsev}. 
    \item If $G$ is connected, then the constant in \eqref{eq:repCW} is of the form \eqref{eq:betterbound} 
    and $\Om$ can be any bounded Lipschitz domain.  
   \end{enumerate} 
\end{corollary}

This follows from applying Theorem \ref{thm:repCW} and Theorem \ref{thm:repCsev} to $W \acts \Si$.

\begin{open}
    To what extent are the above results true for general (nonpolar) representation?
\end{open}

\subsubsection{Analytic lifting}

Let $\rep$ be a complex finite dimensional representation of a finite group $G$.
Then the orbit space coincides with
the categorical quotient $V \cq G$ which is a normal affine variety.
Thus the orbit space  has the natural structure of a complex analytic
set and there are several types of morphisms into $V\cq G$, like
regular, rational, or holomorphic. 
In \cite{KLMR08} those regular, holomorphic, or formal maps into $V \cq G$ 
that admit a 
regular, holomorphic, or formal lifting to $V$ are characterized.
The conditions are formulated in terms of 
the orbit type stratification of the orbit space
which, in this case (compare with Section \ref{sec:repLstrat}), is finer than its stratification as affine variety
and with the use of spaces of jets at $0 \in \C^m$ of morphisms $\C^m \to V\cq G$ 
of finite and infinite order.

Now let $\rep$ be a rational representation of a linearly reductive group $G$.
Suppose that $c : \R \to V\cq G \cong \si(V) \subseteq \C^n$ is real analytic.
In contrast to the real case, there does not always exist a real analytic lifting of $c$ over $\si$.
But there is the following generalization of Puiseux's theorem.

\begin{theorem}[{\cite[Theorem 3.3 and 3.4]{LMRac}}]
    Let $\rep$ be a rational representation of a linearly reductive group $G$. Then:
    \begin{enumerate}
        \item Let $I \subseteq \R$ be an open interval, $t_0 \in I$, and $c \in C^\om(I,\si(V))$.
    Then there exists a positive integer $N$ 
    such that, locally near $t_0$, 
    $t \mapsto c(t_0 \pm (t-t_0)^N)$ admits a real analytic lifting $\ol c_\pm$ to $V$.
\item Let $U \subseteq \C$ be open and connected, $z_0 \in U$, and $c \in \cH(U,\si(V))$, i.e., $c$ is holomorphic.
 Then there exists a positive integer $N$ 
    such that, locally near $z_0$, 
    $z \mapsto c(z_0 + (z-z_0)^N)$ admits a holomorphic lifting $\ol c$ to $V$.
    \end{enumerate}
\end{theorem}

In analogy to Section \ref{sec:repLra},
one defines that $c$ is normally nonflat at $t_0$, see \cite[Section 2.5]{LMRac};
roughly speaking, it means that 
$c$ does not meet lower dimensional strata of the orbit type stratification of $V\cq G$ 
with infinite order of flatness at $t_0$.
Then in (1) we may take $c \in C^\infty(I,\si(V))$ if $c$ is normally nonflat at $t_0$ 
and conclude that $\ol c_\pm$ is of class $C^\infty$. Or we may work with suitable quasianalytic subclasses of $C^\infty$. 

As a consequence, we may conclude that real analytic curves in $\si(V)$ have liftings that are locally absolutely continuous.  

\begin{theorem}[{\cite[Theorem 5.4]{LMRac}}]
    Let $\rep$ be a rational representation of a linearly reductive group $G$. 
    Let $I \subseteq \R$ be an open interval.
    Each $c \in C^\om(I,\si(V))$ admits a lifting $\ol c : I \to V$ that is locally absolutely continuous. 
\end{theorem}

Again this remains true for smooth curves in $\si(V)$ that are normally nonflat.

For maps in several variables into $\si(V)$ there is an analogue of Theorem \ref{thm:repLblow}, where, however, 
local blowings-up \emph{and} local power substitutions must be used.
A \emph{local power substitution}\index{local power substitution} is the composite of the inclusion of a coordinate chart $U$ 
and a map $W \to U$ given in local coordinates by 
\begin{equation*}
    (x_1,\ldots,x_m) \mapsto ((-1)^{\ep_1} x_1^{\ga_1}, \ldots,(-1)^{\ep_m} x_m^{\ga_m}),
\end{equation*}
for some integers $\ga_i \ge 1$ and $\ep_i \in \{0,1\}$.

\begin{theorem}[{\cite[Theorem 4.6]{RainerRG}}]
    Let $\rep$ be a rational representation of a linearly reductive group $G$. 
      Let $M$ be a real analytic manifold and $f \in C^\om(M, \si(V))$.
    For each compact $K \subseteq M$,
    there exist a neighborhood $U$ of $K$ in $M$ 
    and a finite covering $\{\pi_j : U_j \to U\}$ 
    of $U$, where each $\pi_j$ is a composite of finitely many 
    maps each of which is either a local blowing-up with smooth center
    or a local power substitution, 
    such that, for all $j$, 
    the map $f \o \pi_j$ has a real analytic lifting on $U_j$.
\end{theorem}

The theorem remains true for suitable quasianalytic subclasses of $C^\infty$
as well as for holomorphic maps. In the latter case, 
a local power substitution is simply a map 
of the form
$(z_1,\ldots,z_m) \mapsto (z_1^{\ga_1}, \ldots, z_m^{\ga_m})$.

It is possible to extract information about the existence of weak liftings 
of $f$ (without modification of its domain). For instance, there is the following result 
which gives a partial answer to Open Problem \ref{q:repCBV}. 

\begin{theorem}[{\cite[Theorem 6.11]{RainerRG}}] \label{thm:repCSBV}
    Let $\rep$ be a rational representation of a linearly reductive group $G$. 
    Let $U \subseteq \R^m$ be open and $f \in C^\om(U, \si(V))$.
    For each compact $K \subseteq U$,
    there exist a neighborhood $W$ of $K$ in $U$
    and a lifting $\ol f : W \to V$ of $f$ 
    of class SBV (i.e., special functions of bounded variation).
\end{theorem}

Actually, $f$ can be taken in a suitable quasianalytic class.

\subsection{Some remarks on the proofs} We focus on the proof of Theorem \ref{thm:repCW}.
The foundation of the proof is that the result holds for finite rotation groups $C_d \cong \Z/d\Z$ acting 
in the standard way on $\C$, 
where $\C[\C]^{C_d}$ is generated by $z \mapsto z^d$ and a lifting of a map $f$ is a solution of the radical equation 
$z^d = f$. This follows from Theorem \ref{thm:GG}. 
Among all representations of finite groups $G$ of a fixed order $|G|$,
it is the one 
with the worst loss of 
regularity, since in general $d \le |G|$, by Noether's degree bound, and equality can only happen for 
cyclic groups. 

The strategy in the general case is similar to the one described in Section \ref{sec:repLtools}.
Evidently, one may reduce to the case that the linear subspace 
$V^G$ of invariant vectors is trivial. Then 
Luna's slice theorem (see \cite{Luna73} or \cite[Theorem 5.3]{Schwarz80}) allows us to reduce the problem locally to the 
slice representation $G_v \acts N_v$ of the isotropy group $G_v$ on $N_v$, where 
$T_v V \cong T_v (Gv) \oplus N_v$ is a $G_v$-splitting. Since in our case $G$ is finite, we have $N_v \cong V$.   
The assumption $V^G = \{0\}$ entails that for all $v \in V \setminus \{0\}$ the isotropy group $G_v$ 
is a proper subgroup of $G$ which suggests to use induction on the group order. 
For this induction scheme to work it is important that the slice reduction is uniform in the sense that it  
does not depend on the parameter $t$ of the curve $c$ in $\si(V) \subseteq \C^n$.

The lack of a dominant invariant polynomial (in contrast to the real case)
forces one to choose at points $t_0$ with $c(t_0) \ne 0$ a \emph{dominant} component $c_k$ of $c$ such that 
\[
    |c_k^{1/d_k}(t_0)| = \max_{1 \le j \le n} |c_j^{1/d_j}(t_0)|
\]
and to work on small neighborhoods of $t_0$ with this distinguished component $c_k$.

The better bound in the case that $G \acts V$ is coregular (cf.\ Remark \ref{rem:repCW}; see also Remark \ref{rem:repCsev} and Corollary \ref{cor:repCpol}(3))
is achieved by an application of Whitney's extension theorem to 
extend the curve $c$ to a larger interval on which it can be suitably modified without changing it on the original interval $(\al,\be)$.
Coregularity guarantees that the extended curve still lies in $\si(V)$ (because $\si(V)=\C^n$).

\section{Applications} \label{sec:applications}

\subsection{Zero sets of smooth functions}
\label{sec:applications1}

It is well-known that every closed subset of $\R^n$ is the zero set $Z_f$ of some 
real valued $C^\infty$ function $f : \R^n \to \R$. So in general $Z_f$  can be 
arbitrarily irregular. On the other hand, constraints for the derivatives of $f$ 
entail regularity properties of $Z_f$; think of the implicit function theorem or 
of the solutions of elliptic and parabolic PDEs.

Let us assume here that our function $f$ vanishes to some finite order $\ga$ at a 
point $x_0$. Then Malgrange's preparation theorem puts one in the position to study
$Z_f$ in a neighborhood of $x_0$, by using the regularity theorems for the roots of polynomials, 
that is Theorem \ref{thm:optimal} or Theorem \ref{thm:byresolution}. Indeed, 
in suitable local coordinates at $x_0$, we have $f = p \cdot u$, where 
$p$ is a monic polynomial in one of the coordinates with $C^\infty$ coefficients in the other 
coordinates and $u$ is a $C^\infty$ unit.

This setting is widely applicable, 
for instance, to solutions of second order elliptic equations,
Laplace eigenfunctions and finite linear combinations of such (\cite{Donnelly:1988kj}, \cite{Lin:1991aa}, \cite{Jerison:1999aa}), etc.
Often $Z_f$ is called the \emph{nodal set} of $f$, but we will stick to the term 
\emph{zero set}.

The following result is taken from B\"ar \cite{Bar:1999aa}.

\begin{theorem}[{\cite[Lemma 3]{Bar:1999aa}}] \label{thm:Baer}
    Let $f : \R^n \to \R$ be a $C^\infty$ function that vanishes to order $\ga$ at $x_0$.
    Then there exist $r_0>0$ and an affine hyperplane $H$ in $\R^n$ such that 
    $Z_f \cap B(x_0,r_0)$ is contained in the union of countably many graphs of $C^\infty$ functions 
    from $H \cap B(x_0, r_0)$ to $H^\bot$.
    Moreover, for all $r<r_0$,
    \begin{equation*}
        \cH^{n-1}(Z_f \cap B(x_0,r)) \le n 2^{n-1} \ga\, r^{n-1}.
    \end{equation*}
\end{theorem}

Note that Yomdin \cite{Yomdin:1984dg} obtained a similar result; 
he assumed that the partial derivatives of some order $j$ of $f$ are ``small'' in a certain 
precise sense and concluded that $f$ behaves in many regards like a polynomial of degree $j-1$, 
in particular, it has vanishing order at most $j-1$.

The following theorem is a consequence of Theorem \ref{thm:optimal} or Theorem \ref{thm:byresolution} (together with Malgrange's preparation theorem),
due to Beck, Becker-Kahn, and Hanin \cite{Beck:2018hg}.

\begin{theorem}[{\cite[Theorem 1]{Beck:2018hg}}] \label{thm:Hanin}
    Let $U \subseteq \R^n$ be open and let $\Th$ be a compact smooth manifold (possibly with boundary). 
    Let $f \in C^\infty(U \times \Th)$ and write $f^\th := f(\cdot, \th)$ for $\th \in \Th$.
    Suppose that $f^{\th_0}$ has finite vanishing order $\ga$ at $x_0 \in U$.
    Then there exist $p>1$, $r_0>0$, a neighborhood $V$ of $\th_0$ in $\Th$, an affine hyperplane 
    $H \subseteq \R^n$, and functions $f_i^\th \in W^{1,p}(H,H^\bot)$, $1 \le i \le \ga$, $\th \in V$,
    with
    \begin{equation*}
        \sup_{1 \le i \le \ga, \, \th \in V} \|f_i^\th\|_{W^{1,p}} < \infty
    \end{equation*}
    such that 
    \begin{equation*}
        Z_{f^\th} \cap B(x_0,r_0) \subseteq \bigcup_{i=1}^\ga \Ga_{f_i^\th}, \quad  \th \in V, 
    \end{equation*}
    where $\Ga_{f^\th_i}$ is the graph of $f^\th_i$.
\end{theorem}

\begin{remark}
    Theorem \ref{thm:Hanin} was proved in \cite{Beck:2018hg} based on Theorem \ref{thm:byresolution}. 
    Using the optimal result, Theorem \ref{thm:optimal}, in the proof, 
    one can get more precise information on $p$: the statement is true for all $1 \le p < \ga/(\ga-1)$. 
\end{remark}

By the same strategy, one gets similar parameterization results in the setting studied by Yomdin \cite{Yomdin:1984dg} 
(where the partial derivatives of some order are small)
as well as for $C^\infty$ functions all of whose derivatives are suitably controlled by a weight sequence 
(e.g.\ functions in a quasianalytic Denjoy--Carleman class), 
see Rainer \cite{Rainer:2022aa}. In the latter case, the vanishing order is bounded (globally on a convex body) 
by a quantity computed from the weight sequence 
which has similarities with the polynomial degree.

Note that, if one is interested only in the properties of the zero set of a function near a 
point of finite vanishing order, then there is some freedom in the choice of the local coordinates 
(i.e., the hyperplane $H$ in Theorem \ref{thm:Hanin}).

\begin{open}
    Is is possible to improve the above results by choosing the local system of coordinates wisely?  
\end{open}

As a consequence of Theorem \ref{thm:Baer} or Theorem \ref{thm:Hanin}, one gets 
a nonconcentration estimate \cite[Proposition 4]{Beck:2018hg}: suppose that the vanishing order of $f^\th$ is at most $\ga$ for all $\th \in \Th$.
For compact $K \subseteq U$ and $m$-rectifiable $E \subseteq \R^n$, for some $m \le n-1$, 
there exists $r_0>0$ such that 
\[
    \cH^{n-1}(Z_{f^\th} \cap E_r' \cap K) \le C(n)\, r^{-1} \cH^n(E_r), \quad r \le r_0,
\]
for all $\th \in \Th$ and all images $E'$ of Lipschitz maps $E \to K$ with Lipschitz constant $\le 1$,
where $E_r, E_r'$ denote the $r$-neighborhoods of $E,E'$ in $\R^n$.
This entails the following continuity result.
Let $\on{Sing}_f := \{x \in Z_f : \nabla f(x)=0\}$.

\begin{corollary}[{\cite[Corollary 2]{Beck:2018hg}}]
    Let $U \subseteq \R^n$ be open and $f \in C^\infty(U \times [0,1])$ 
    such that $f^1 := f(\cdot,1)$ has 
    finite vanishing order on $U$.
    Then for any compact $K \subseteq U$ with $\cH^{n-1}(K \cap \on{Sing}_f)=0$ 
    we have 
    \[
        \lim_{\th \to 1} \cH^{n-1}(Z_{f^\th} \cap K) = \cH^{n-1}(Z_{f^1} \cap K).
    \]
\end{corollary}

In particular, the authors of \cite{Beck:2018hg} conclude a continuity result for the zero sets of 
the heat flow on a compact smooth Riemannian manifold $(M,g)$: if $u : M \times (0,\infty) \to \R$ 
solves 
\[
    (\p_t + \De_g)u(x,t)=0
\]
with initial data $u^0 \in L^2(M)$, then
\[    
    \lim_{t \to 1} \cH^{n-1}(Z_{u^t}) = \cH^{n-1}(Z_{\ps}),
\]
where $\ps$ is the first nonzero eigenspace projection of $u^0$ and the time variable 
underwent a change from $t$ to $\frac{2}{\pi} \arctan(t)$.

Another application of Theorem \ref{thm:Hanin} presented in \cite{Beck:2018hg}
is that, given $C^\infty$ functions $f_0,f_1,\ldots,f_p$ on a compact Riemannian manifold $(M,g)$ of dimension $n$
such that for all $(\th_0,\th_1,\ldots,\th_p) \in \R^{p+1}\setminus \{0\}$ 
the vanishing order of $\th_0 f_0+ \cdots + \th_p f_p$ is at most $\ga$ on $M$,
the map 
\[
    \R\mathbb P^p\ni  [\th_0 : \cdots : \th_p] \mapsto \p \{\th_0 f_0+ \cdots + \th_p f_p<0\}
\]
into the space of modulo 2 flat $(n-1)$-cycles in $M$
is an admissible $p$-sweepout (see \cite{Beck:2018hg} for the definition). This was proposed by Marques and Neves in \cite{Marques:2017aa}, 
where they settled a conjecture of Yau on the existence of infinitely many minimal surfaces in 
the setting of positive Ricci curvature. 
This proof uses an argument involving the growth in $p$ of the min-max width $\om_p(M)$ which is defined  
in terms of admissible $p$-sweepouts and is a nonlinear version of the spectrum of the Laplacian 
(cf.\ Gromov \cite{Gromov:1988aa}).
In fact, by Guth \cite{Guth:2009aa}, one has 
\begin{equation*}
    \om_p(M) \le C\, p^{1/n},
\end{equation*}
which is reproved in \cite[Theorem 3]{Beck:2018hg} as a corollary of Theorem \ref{thm:Hanin}.

Let us finish this section with a few consequences of the results for radicals, i.e.\, Theorem \ref{thm:GG},
due to Ghisi and Gobbino \cite{GhisiGobbino13} and the co-area formula. 
These facts play an essential role in the proof of the existence of $BV$ roots, Theorem \ref{thm:BVroots}, in \cite{ParusinskiRainerBV}.
They concern the behavior of differentiable functions near their zero set and 
the level sets of their sign; they are interesting in their own right.
Note that there is no assumption on the vanishing order.

\begin{theorem}[{\cite[Theorem 3.5]{ParusinskiRainerBV}}] \label{prop:levelgrowth}
    Let $k \in \N_{\ge 1}$, $\al \in (0,1]$, and set $s = k +\al$. 
    Let $\Om\subseteq \R^n$ be a bounded Lipschitz domain and
    $f \in C^{k,\al}(\overline \Om,\R^{\ell+1})$,  $f \not \equiv 0$.  
    Then there is a constant $C = C(n,\ell,k,\al,\Om)$ such that 
    for all $0<\ep \le 1$ and all small $\de >   0$ we have
    \begin{equation*} 
        \big|\big\{y \in (0,\de) : y^{1/s}\, \cH^{n-1}(|f|^{-1}(y)) 
        \ge \ep^{-1}C \,\|f\|^{1/s}_{C^{k,\al}(\overline \Om)} \big\}\big| \le \ep \de.
    \end{equation*}
\end{theorem}

In particular, 
if $A \subseteq [0,\infty)$ has full measure near $0$, i.e.,
$|A \cap [0,\ep)| = \ep$ for some $\ep>0$, 
then there is a sequence $A \ni y_j \to 0$ with
\begin{equation*} 
    \sup_j \Big(y_j^{1/s}\, \cH^{n-1}(|f|^{-1}(y_j)) \Big)\le C \,\|f\|^{1/s}_{C^{k,\al}(\overline \Om)}.
\end{equation*}
That means that the $\cH^{n-1}$-measures of the level sets $\{|f| = y_j\}$ do not grow faster than $y_j^{-1/s}$ 
as $y_j \to 0$.

The \emph{sign}\index{sign of a function} of $f : \Om \to \R^{\ell+1}$ is the map $\on{sgn}(f) : \Om \setminus f^{-1}(0) \to \mathbb S^\ell$ given by
\[
    \on{sgn}(f):= \frac{f}{|f|}.
\]

\begin{theorem}[{\cite[Theorem 3.3]{ParusinskiRainerBV}}] \label{prop:level}
    Let $k \in \N_{\ge 1}$, $\al \in (0,1]$, and set $s = k +\al$. 
    Let $\Om\subseteq \R^n$ be a bounded Lipschitz domain
    and $f \in C^{k,\al}(\overline \Om, \R^{\ell+1})$, where $n \ge \ell \ge 1$. 
    Then there is a constant $C= C(n, \ell,k,\al,\Om)$ such that
    for each small $\ep >0$ 
    \begin{equation*} 
        \cH^\ell\Big(\Big\{y \in \mathbb S^\ell : \int_{\on{sgn}(f)^{-1}(y)}  |f|^{\ell/s} \, d \cH^{n-\ell} 
        \ge  \ep^{-1} \, C \, \|f\|^{\ell/s}_{C^{k,\al}(\overline \Om)} \Big\}\Big) \le \ep.
    \end{equation*}
\end{theorem}

In particular, we have $\ell/s$-integrability with respect to $\cH^{n-\ell}$ of $f$ along almost every level set of the sign of $f$:
\begin{equation*} 
    \int_{\on{sgn}(f)^{-1}(y)}  |f|^{\ell/s} \, d \cH^{n-\ell} < \infty, \quad \text{ for $\cH^{\ell}$-a.e. } y \in \mathbb S^{\ell}. 
\end{equation*}
Moreover,
for every relatively compact open $K \Subset \Om \setminus f^{-1}(0)$, 
\begin{equation*} 
    \cH^{n-\ell}\big(K \cap \on{sgn}(f)^{-1}(y)\big) < \infty, \quad \text{ for $\cH^{\ell}$-a.e. } y \in \mathbb S^{\ell}.  
\end{equation*}

\subsection{Extension to the optimal transport between algebraic hypersurfaces}
\label{sec:optimaltransport}

A recent paper \cite{ACL:preprint}, by Antonini, Cavalletti, and Lerario, contains an interesting reinterpretation and 
extension of the regularity problem of the roots, Section~\ref{sec:regularity}, 
to the study of the Wasserstein distance on the space of $d$-degree hypersurfaces in $\C \mathbb P^n$.

Let $(M, g)$ be a smooth compact Riemannian manifold and let  $\sP (M)$ denote  the space of Borel probability measures on $M$.  Fix $q\ge 1$. For any $\mu_0, \mu_1 \in \sP (M)$, 
the \emph{$q$-Wasserstein distance}\index{Wasserstein distance} is defined by
$$
W_{q} (\mu_{0},\mu_{1}) : = \left(\inf_{\xi \in \Pi(\mu_{0},\mu_{1})}
\int_{M\times M} \mathsf{d}_{g}(x,y)^q \, \xi(dxdy)\right)^{\frac{1}{q}},  
$$
where $\mathsf{d}_{g}$ is the geodesic distance on $M$ and $\Pi(\mu_{0},\mu_{1})$
is the set of transport plans between $\mu_{0}$ and $\mu_{1}$, i.e., 
the set of probability measures $\xi$ on $M \times M$ with marginals $\mu_0$ and $\mu_1$ 
(the push-forward measures of $\xi$ with respect to the natural projections). 
Intuitively, this distance minimizes the transport cost from $\mu_0$ to $\mu_1$ over all such plans.

Denote by $H_{n,d}$ the space of complex homogeneous polynomials of degree $d$ in $n+1$ variables.  
Let $P_{n,d}:=\mathbb{P}(H_{n,d})$ be the projectivization.  Note that $P_{n,d}$ can be identified with 
$\C \mathbb P^N,  N=\binom{n+d}{d}-1$.  Let $Z_P$ denote the zero set of $P\in H_{n,d}$.  
This set depends only on the class of $P$ in $P_{n,d}$.  
Therefore, to simplify the notation, we will consider the zero sets $Z_P$ of $P\in P_{n,d}$, abusing the terminology slightly. 

Let us associate to each $P\in P_{n,d}$ a measure $\mu(P)\in \sP(\C \mathbb P^n)$ 
as follows.
If $Z_P$ is nonsingular, then 
we take the restriction to ${Z_P}$ of the Hausdorff measure $\cH^{2n-2}$  (with respect to the  Fubini--Study metric of $\C \mathbb P^n$), normalized:
\[ 
    \label{eq:measureintro}
    \mu(P):=\frac{1}{\mathrm{vol}(Z_P)}\, \cH^{2n-2} \llcorner {Z_P} , 
\]
where $\mathrm{vol}(Z_P) := (\cH^{2n-2} \llcorner {Z_P}) ({Z_P})$.
(Note that all nonsingular  hypersurfaces of degree $d$ have the same volume 
$\mathrm{vol}(Z_P)=d\, \mathrm{vol}(\C \mathbb P^{n-1})$ so that the normalization factor does not depend on $P$.)

Denote by $\Delta_{n,d}\subseteq P_{n,d}$ the discriminant locus, i.e., the set of polynomials $P\in P_{n,d}$ such that 
$Z_P$ is singular.  It is a hypersurface in $P_{n,d} = \C \mathbb P^N$ given by 
$$
\Delta_{n,d}:=\left\{P\in P_{n,d} : \exists z\in \C^{n+1}\setminus \{0\}, \, P(z)=\frac{\partial P}{\partial z_0}(z)=\cdots=\frac{\partial P}{\partial z_n}(z)=0\right\},
$$ 
Using integral geometry (a Cauchy--Crofton kind of argument) one can extend the definition of $\mu(P)$ to singular hypersurfaces, 
that is for $P\in \Delta_{n,d}$, as follows, see  \cite[Theorem 1.1]{ACL:preprint}, 
\[
    \int_{\C \mathbb P^n}f \, d \mu(P)
    :=\frac{1}{d}\int_{\mathbb{G}(1,n)}\left(\sum_{z\in Z_P\cap \ell}m_z(P|_{\ell}) f(z)\right)\mathrm{vol}_{\mathbb{G}(1,n)}(d \ell), 
    \quad f \in C^0(\C \mathbb P^n),
\]
where ${\mathbb{G}(1,n)}$ stands for the set of projective lines in projective $n$-space, 
which is equivalent to the Grassmannian $\on{Gr}(2, n+1)$, 
and $m_z(P|_\ell)$ is the multiplicity of $z$ as a zero of $P|_\ell$.
Then for every $q\geq 1$, the map  $\mu:P_{n,d} \to \sP_q(\C \mathbb P^n)$ 
is continuous and injective.    
(Here by $\sP_q(\C \mathbb P^n)$ we mean $\sP(\C \mathbb P^n)$ with the Wasserstein distance $W_q$.)

\begin{example}[$n=1$]\label{example:n=1}
    For simplicity, 
    we prefer to work with a polynomial of one complex variable 
    $P\in \C[z]$.  The passage to homogeneous polynomials of two complex variables $z_0,z_1$ 
    is classical be means of homogenization, see e.g. \cite[Remark 3.2]{ACL:preprint}. 

    Consider the space of monic polynomials of degree $d$ in one complex variable, 
    $P_a(z)=z^d+a_{1}z^{d-1}+\cdots +a_{d-1}z+a_d$, $a=(a_1, \ldots, a_d)\in \C^d$. 
    To each polynomial $P=P_a\in \C^d$ one associates a probability measure 
    $\mu(P)\in \sP(\C)$ defined by
    \[
        \label{eq:deltaintro}\mu(P):=\frac{1}{d}\sum_{P(z)=0}m_z(P)\cdot\delta_z,
    \]
    where $m_z(P)\in \N$ denotes the multiplicity of $z$ as a zero of $P$ and $\delta_z$ is the Dirac measure.  This defines a map
    \[
        \label{eq:solmap} 
        \mu:\C^d\to \sP (\C).
    \]
    Geometrically, the image of $\mu$ can be identified with the quotient of $\C^d$ under the action 
    of the symmetric group $\on{S}_d$.  
    The $q$-th Wasserstein metric on $\sP(\C)$   is given by
    \[
        \label{eq:W21}
        W_q([x], [y]) =\Big( {\frac{1}{d}} \min_{\sigma\in \on{S}_d} \sum_{j=1}^d|x_j-y_{\sigma(j)}|^q\Big)^{1/q},
    \]
    where $[x]=[x_1, \ldots, x_d]$ and $[y]=[y_1, \ldots, y_d]$ are unordered $d$-tuples.  It is geodesically convex and 
    coincides with the restriction of the $q$-Wasserstein metric of  $\sP_q (\C)$. 
\end{example}

Contrary to the case $n=1$, in general, the image of $\mu$ is not geodesically convex for the $W_q$-metric, 
that is the geodesic joining $\mu(P_0)$ to $\mu(P_1)$ in $\sP_q(\C \mathbb P^n)$ will in general not stay on $\mu(P_{n,d})$.
To overcome this issue it is proposed in \cite{ACL:preprint} to consider 
the inner $W_q$-distance, adopting the approach of 
J.-D. Benamou and Y. Brenier.  For this one needs the notion of a $q$-absolutely continuous curve in this setup.  
A curve $\mu_t : I \to \mathscr{P}_q(\C \mathbb P^n)$ belongs to $AC^q(I, \mathscr{P}_q(\C \mathbb P^n))$ and is called \emph{$q$-absolutely continuous}\index{absolutely continuous} 
if there is $f \in L^q(I)$ such that
\[
    W_q(\mu_s,\mu_t) \le \int_s^t f(\ta)\, d\ta, \quad \text{for all } s\le t.
\]
In that case, for $\cL^1$-a.e.\ $t \in I$ the limit $\lim_{h \to 0} W_q(\mu_{t+h},\mu_t)/|h|$ exists,
is denoted by $|\dot\mu_t|_q$, and called \emph{metric speed}.\index{metric speed}
Then the \emph{inner $W_q$-distance on $P_{n,d}$} is defined by \index{inner Wasserstein distance}\index{Wasserstein distance!inner}
\[ \label{eq:bb}
    W_q^\textrm{in}(P_{0},P_{1}): =   \inf_{\mu_t:=\mu(\gamma(t))}\Big(\int_{I}|\dot\mu_t|_q^q\, d t \Big)^{\frac{1}{q}},
\]
where the infimum is taken over all $q$-absolutely continuous curves $\gamma:I=[0,1]\to P_{n,d}$ 
joining $P_0$ and $P_1$.

It is shown in Section 5.2 of \cite{ACL:preprint}, see Theorem 5.13 and Remark 6.10 therein, 
that  the metric space $(P_{n,d}, W_q^{\textrm{in}})$ is  complete and geodesic.  
Moreover, away of the discriminant locus $\Delta_{n,d}$ the $W_q^{\textrm{in}}$ distance and the Fubini--Study distance 
$\mathrm {d}_{\mathrm{FS}} $ on $P_{n,d}$ defined via the identification $P_{n,d}=\C\mathbb P^N$ are related as follows. 

\begin{theorem}\label{thm:Lipschitz}
    Let $\epsilon>0$ and denote 
    $P_{n,d}(\epsilon):=\{p\in P_{n,d} : \mathrm {d}_{\mathrm{FS}}(p,\Delta_{n,d})\geq \epsilon\}$.
    Then the identity map
    $$\mathrm{id}:(P_{n,d}(\epsilon), \mathrm {d}_{\mathrm{FS}})\to (P_{n,d}(\epsilon), W_q^\mathrm{in})$$
    is Lipschitz. In particular, $P_{n,d}(\epsilon)$ is compact in the $W_q^\mathrm{in}$--topology.
\end{theorem}

In the proof, continuous semialgebraic curves are used.  
By Puiseux's theorem, they can be reparametrized, near their singular points, to get $C^m$ regularity, for arbitrary finite $m$.  

Another important question is the compactness of $(P_{n,d}, W_q^\mathrm{in})$.  
This problem was only partially solved in \cite{ACL:preprint}.

\begin{theorem}[{\cite[Theorem 6.11]{ACL:preprint}}]\label{thm:compactq}
    There is $e(n)\in \N$ such that, for every $1 \leq q < e(n)d/(e(n)d-1)$, 
    the complete and separable geodesic space $(P_{n,d}, W_q^{\on{in}})$ is compact.
\end{theorem}

Theorem \ref{thm:optimal} is used in the proof of the above theorem to provide $q$-absolutely continuous curves as follows.  
This is fairly straightforward for $n=1$ because, any curve of polynomials 
$$
P_a(t)(z) 
= z^d + 
\sum_{j = 1}^{d} a_j(t) z^{d-j}, \quad t\in [0,1],
$$
with $a_j\in C^{d-1,1}([0,1])$, defines, after homogenization, an element of $AC^q([0,1];\sP_q(\C\mathbb P^1))$ for every $1\le q<d/(d-1)$. Thus one may take $e(1) =1$. 

Thanks to the uniformity of the bound of Theorem \ref{thm:optimal}, 
the general case can be reduced to the case $n=1$ by restriction to all affine lines parameterized by ${\mathbb{G}(1,n)}$ 
and then integration over this Grassmannian, provided we can avoid the lines entirely included in $Z_P$.   
This is always possible for generic $P$ if $d> 2n-3$, \cite[Theorem 6.3]{ACL:preprint}, but not in general.   
Therefore it is proposed in \cite{ACL:preprint} to use the integration over the group of unitary transformations of a 
rational curve of sufficiently high degree (instead of a line).
It follows from a recent result of B. Lehmann, E. Riedl and S. Tanimoto that for a given curve of polynomials $P_t \in P_{n,d}$ 
there is a rational curve of degree $e(n)$, 
given by an explicit formula, such that none of its unitary transformations can be entirely included in $Z_{P_t}$. 
For all details of the proof we refer the reader to \cite{ACL:preprint}.

\begin{open}
    For which  $q\ge 1$ is the space $(P_{n,d}, W_q^{\on{in}})$ compact?  
    In particular, is it always compact for $q=2$?
\end{open}

\section{Appendix A. Function spaces} \label{Appendix:A}

\subsection{H\"older spaces} \label{sec:AHoeld}

Let $\Om \subseteq \R^n$ be open. 
We denote by $C^0(\Om)$ the space of continuous complex valued functions on $\Om$.
For $k \in \N_{\ge 1} \cup \{\infty\}$, 
$C^k(\Om)$ is the space of $k$-times continuously differentiable functions,\index{Ck@$C^k$}\index{Cinf@$C^\infty$}
\[
    C^k(\Om) = \{f \in \C^\Om : \p^\al f \in C^0(\Om) \text{ for } 0 \le |\al| \le k\}, \quad k \in \N,
\]
and
\[
    C^\infty(\Om) = \bigcap_{k \in \N} C^k(\Om). 
\]
The space of real analytic functions on $\Om$ is denoted by $C^\om(\Om)$.\index{Com@$C^\om$}\index{H@$\cH$}
If $\Om \subseteq \C^n$, then $\cH(\Om)$ is the space of holomorphic functions on $\Om$.

If the open set $\Om \subseteq \R^n$ is bounded, we write $C^k(\ol \Om)$ for the space of 
all $f \in C^k(\Om)$ such that $\p^\al f$, $0 \le |\al| \le k$, has a continuous extension to the closure 
$\ol \Om$. Then $C^k(\ol \Om)$ endowed with the norm 
\[
    \|f\|_{C^k(\ol \Om)} := \sup_{|\al|\le k} \sup_{x \in \Om} |\p^\al f(x)|   
\]
is a Banach space. 
The spaces $C^k(\Om)$, $C^\infty(\Om)$, $C^\om(\Om)$, and $\cH(\Om)$
carry their natural locally convex topologies, where $\Om$ is not necessarily bounded;
note that $C^k(\Om)$, $C^\infty(\Om)$, and $\cH(\Om)$ are Fr\'echet spaces. 

Let $\al \in (0,1]$ and let $\Om \subseteq \R^n$ be open and bounded. 
A function is \emph{$\al$-H\"older continuous in $\Om$} if\index{Hoeld@$\Hoeld_{\al,\Om}(f)$} 
\[
    \Hoeld_{\al,\Om}(f) := \sup_{x,y \in \Om,\, x\ne y} \frac{|f(x)-f(y)|}{|x-y|^\al} < \infty. 
\]
It is \emph{Lipschitz continuous in $\Om$} if it is $1$-H\"older continuous in $\Om$; in that case, we write \index{Lip@$\Lip_{\Om}(f)$} 
\[
    \Lip_\Om(f) := \Hoeld_{1,\Om}(f)< \infty.
\]
For $k \in \N$ and $\al \in (0,1]$, we have the scale of \emph{H\"older(--Lipschitz) spaces}\index{Ckal@$C^{k,\al}$}\index{H\"older--Lipschitz space}
\[
    C^{k,\al}(\ol \Om) := \{f \in C^k(\ol \Om) : \p^\be f \text{ is $\al$-H\"older continuous for } |\be|=k   \}
\]
and we equip them with the norm
\[
    \|f\|_{C^{k,\al}(\ol \Om)} := \|f\|_{C^k(\ol \Om)} + \sup_{|\be| = k} \Hoeld_{\al,\Om}(\p^\be f).
\]
Then $C^{k,\al}(\ol \Om)$ is a Banach space. 
We write $C^{k,\al}(\Om)$ (where $\Om$ is not necessarily bounded) for the space of all functions $f \in C^k(\Om)$ such that 
$\|f|_{\Om'}\|_{C^{k,\al}(\ol \Om')}< \infty$ for all relatively compact $\Om' \Subset \Om$ 
with its natural Fr\'echet topology.

\subsection{Lebesgue spaces and weak Lebesgue spaces} \label{sec:ALeb}

Let $\Om \subseteq \R^n$ be open and $1 \le p \le \infty$. 
The \emph{Lebesgue space}\index{Lebesgue space}\index{Lp@$L^p$} $L^p(\Om)$ is the space of measurable complex valued functions $f : \Om \to \C$ that are $p$-integrable
with respect to the $n$-dimensional Lebesgue measure $\cL^n$; actually, the elements of $L^p(\Om)$ are equivalence classes of functions
that coincide $\cL^n$-almost everywhere.
If not stated otherwise, ``measurable'' always means ``Lebesgue measurable''. 
The $L^p$-norms make the Lebesgue spaces to Banach spaces:
\begin{align*}
    \|f\|_{L^p(\Om)} &:= \Big(\int_\Om |f(x)|^p dx\Big)^{1/p}, \quad 1 \le p<\infty, 
    \\
    \|f\|_{L^\infty(\Om)} &:= \on{ess\,sup}_{x \in \Om} |f(x)|, \quad p =\infty. 
\end{align*}
We denote by $L^p_{\on{loc}}(\Om)$ the space of measurable functions $f : \Om \to \C$ such that $\|f\|_{L^p(K)} < \infty$ for each compact subset $K \subseteq \Om$.

Let $\Om \subseteq \R^n$ be open and bounded. For $1 \le p < \infty$, the 
\emph{weak $L^p$-space}\index{weak Lebesgue space}\index{Lpw@$L^p_w$} is the space of all measurable functions $f : \Om \to \C$ such that 
\[
    \|f\|_{p,w,\Om} := \sup_{r> 0} \big(r \cdot \cL^n(\{x \in \Om : |f(x)|>r\})^{1/p}\big) < \infty;
\]
again, we identify functions that coincide $\cL^n$-almost everywhere.
Note that $\|\cdot\|_{p,w,\Om}$ is only a quasinorm: the triangle inequality fails, but for $f_j \in L^p_w(\Om)$ one still 
has 
\[
    \Big\| \sum_{j=1}^m f_j \Big\|_{p,w,\Om} \le m \sum_{j=1}^m \|f_j\|_{p,w,\Om}.
\]
The spaces $L^p_w(\Om)$ are complete with respect to the quasinorms $\|\cdot\|_{p,w,\Om}$, i.e., they are quasi-Banach spaces.
If $p>1$, then there exists a norm equivalent to $\|\cdot\|_{p,w,\Om}$ which makes $L^p_w(\Om)$ into a Banach space.
For $1 \le q <p < \infty$, we have
\[
    \|f\|_{q,w,\Om} \le \|f\|_{L^q(\Om)} \le \Big(\frac{p}{p-q}\Big)^{1/q} \cL^n(\Om)^{1/q - 1/p}\, \|f\|_{p,w,\Om}
\]
and thus there are continuous strict inclusions $L^p_w(\Om) \subseteq L^q(\Om) \subseteq L^q_w(\Om)$.
The $L^p_w$-quasinorm is $\si$-subadditive (i.e., $\|f\|^p_{p,w,\Om} \le \sum_j \|f\|^p_{p,w,\Om_j}$ for 
countable open covers $\Om= \bigcup_j \Om_j$) but not $\si$-additive.

\subsection{Sobolev spaces} \label{sec:ASob}

Let $\Om \subseteq \R^n$ be open.
For $k \in \N_{\ge 1}$ and $1 \le p \le \infty$, we consider the \emph{Sobolev space}\index{Sobolev space}\index{Wkp@$W^{k,p}$} 
\[
    W^{k,p}(\Om) := \{f \in L^p(\Om) : \p ^\al f \in L^p(\Om) \text{ for } |\al|\le k\},
\]
where $\p^\al f$ are distributional derivatives, with the norm
\[
    \|f\|_{W^{k,p}(\Om)} := \sum_{|\al|\le k} \|\p^\al f\|_{L^p(\Om)}.
\]
Then $W^{k,p}(\Om)$ are Banach spaces.
We denote by $W^{k,p}_{\on{loc}}(\Om)$ the space of functions $f \in L^p_{\on{loc}}(\Om)$ such that $\p^\al f \in L^p_{\on{loc}}(\Om)$ 
for $|\al|\le k$.

Let $k=1$ and suppose $f \in L^p(\Om)$.
Then $f \in W^{1,p}(\Om)$ if and only if $f$ has a representative $\ol f$ that is absolutely continuous on almost all line segments in 
$\Om$ parallel to the coordinate axes and whose classical derivatives $\p_i \ol f$ belong to $L^p(\Om)$.
In particular, for a bounded interval $I \subseteq \R$, $W^{1,1}(I)$ can be identified with the space $AC(I)$ of absolutely continuous\index{absolutely continuous}\index{AC@$AC$} 
functions on $I$.
We recall that, by definition, $f \in AC(I)$ if for every $\ep>0$ there exists $\de>0$ such that 
$\sum_{j=1}^k |f(b_j) - f(a_j)| \le \ep$
for every finite number of disjoint intervals $(a_j,b_j)$, $j=1,\ldots,k$, satisfying $[a_j,b_j] \subseteq I$ and 
$\sum_{j=1}^k (b_j-a_j) \le \de$.

In the case that $p=\infty$ and that $\Om$ is a bounded Lipschitz domain, 
$f \in W^{1,\infty}(\Om)$ if and only if $f$ has a representative that is Lipschitz continuous on $\Om$.

\subsection{Functions of bounded variation} \label{sec:ABV}

Let $\Om \subseteq \R^n$ be open.
A real valued $f \in L^1(\Om)$ is a \emph{function of bounded variation in $\Om$}\index{bounded variation} 
if the distributional derivative of $f$ is representable 
by a finite Radon measure in $\Om$, i.e., for all $C^\infty$ functions $\vh : \Om \to \R$ with compact support in $\Om$ and all $1 \le i \le n$,
\[
    \int_\Om f(x) \p_i \vh(x) \, dx = - \int_{\Om} \vh \, dD_i f,
\]
for some $\R^n$-valued Radon measure $Df = (D_1f,\ldots,D_nf)$ on $\Om$.
By definition, a complex valued function is of bounded variation in $\Om$ if its real and imaginary parts are.
The space $BV(\Om)$\index{BV@$BV$} of all functions of bounded variation in $\Om$ 
equipped with the norm 
\[
    \|f\|_{BV(\Om)} := \|f\|_{L^1(\Om)}  + |Df|(\Om),
\]
where $|Df|(\Om)$ is the total variation measure, is a Banach space.

The Sobolev space $W^{1,1}(\Om)$ is strictly contained in $BV(\Om)$ (for $f \in W^{1,1}(\Om)$ we have $Df = \nabla f \, \cL^n$).

Functions of bounded variation may have discontinuities.
Let $f \in BV(\Om)$. Then the Lebesgue decomposition with respect to $\cL^n$ induces a decomposition 
\[
    Df = D^a f + D^j f + D^c f,
\]
where $D^a f = \nabla f \, \cL^n$ is the \emph{absolutely continuous part},
$D^j f$ is the \emph{jump part} (given by integration of the jump height against the restriction of $\cH^{n-1}$ to the set of approximate jump points $J_f$, 
i.e., $D^j f = ((f^+-f^-) \otimes \nu_f) \, (\cH^{n-1} \llcorner J_f)$, where $\nu_f$ is the direction of the jump), 
and $D^c f$ is the \emph{Cantor part}.

The closed subspace $SBV(\Om) := \{f\in BV(\Om) : D^c f = 0\}$ of $BV(\Om)$ is called the space of
\emph{special functions of bounded variation}.\index{SBV@$SBV$}\index{special functions of bounded variation} 
We have continuous strict inclusions $W^{1,1}(\Om) \subseteq SBV(\Om) \subseteq BV(\Om)$.

\section{Appendix B. The space of hyperbolic polynomials}
\label{sec:spaceHd}

The fact that the regularity of the roots of hyperbolic polynomials is essentially better than that of the roots of general polynomials  
is related to the geometry of the space $\Hyp(d)$\index{Hypd@$\Hyp(d)$} of monic hyperbolic polynomials of degree $d$
and the constraints a sufficiently differentiable curve of hyperbolic polynomials is subdued to at the boundary of $\Hyp(d)$. 
In this section, we collect some interesting topological and geometric properties of $\Hyp(d)$.

The following results are due to Kostov \cite{Kostov89}; see also Arnol'd \cite{Arnold86} and Givental' \cite{Givental87}.
Let $p=(p_1,\ldots,p_d)$ be a probability vector, i.e., $p_i>0$ and $p_1+\cdots+p_d=1$.
The \emph{Vandermonde mapping}\index{Vandermonde mapping} $W = W(p) : \R^d \to \R^d$ is defined by 
\[
    W_j(x) := p_1 x_1^j + \cdots + p_d x_d^j, \quad 1 \le j \le d.
\]
If $p_i = 1/d$ for $1 \le j \le d$, then $W_j$, $1 \le j \le d$, are the Newton polynomials (up to a constant factor) which 
generate the algebra of symmetric polynomials on $\R^d$ and 
consequently the image $W(\R^d)$ is isomorphic to $\Hyp(d)$ in this case. 

The set $K := \{x \in \R^d : x_1 \le x_2 \le \cdots \le x_d\}$ is a fundamental domain for the 
action of the symmetric group $\on{S}_d$ on $\R^d$ by permuting the coordinates.
Every set of the form $\{x \in K : W_j(x) = c_j, \, 1 \le j\le k\}$ is called a
\emph{$(d-k)$-dimensional Vandermonde manifold}.\index{Vandermonde manifold}

\begin{theorem}[{\cite{Kostov89}}]
    Every Vandermonde manifold is either contractible or empty. 
    The map $W : K \to W(K)$ is a homeomorphism. 
    Every set of the form $\{y \in W(K) : y_j=c_j, \, 1\le j \le k\}$ 
    is either contractible or empty. 
\end{theorem}

As a consequence, $W^k:=(W_1,\ldots,W_k)$, $k \le d$, is a homeomorphism of the closure of every $k$-dimensional stratum of the polyhedron $K$ 
onto its image; indeed, the restriction of $W^k$ to such a stratum is the $k$-dimensional Vandermonde mapping for some probability vector.

Let $K_0 := \{x \in K : W_1(x)=0,\, W_2(x)\le 1\}$.

\begin{theorem}[{\cite{Kostov89}}]
    The sets $W(K_0)$ and $W^k(K_0)$, $1 \le k \le d-1$, are quasiconvex, 
    i.e., there is a constant $C>0$ such that any two points can be joint by a piecewise smooth curve in the set 
    of length not exceeding $C$ times the Euclidean distance of the two points.
\end{theorem}

In particular, the space of monic hyperbolic polynomials of degree $d$ in Tschirnhausen form with $|\tilde a_2| \le 1$ 
is quasiconvex. We refer to \cite{Kostov89} for additional information on the stratification of $W(K)$.

The semialgebraic set $\Hyp(d) \subseteq \R^d$ can be described by explicit inequalities.
Let $s_j = x_1^j + \cdots + x_d^j$, $j \ge 0$, be the Newton polynomials and consider the Bezoutiant matrix 
\[
    B:=
    \begin{pmatrix}
        s_0 & s_1 & \cdots & s_{d-1}
        \\
        s_1 & s_2 & \cdots & s_d
        \\
        \vdots & \vdots & \ddots & \vdots
        \\
        s_{d-1} & s_d & \cdots & s_{2d-2}
    \end{pmatrix}.
\]
We have $B = \tilde B(\si_1,\ldots,\si_d)$ for a unique matrix valued polynomial $\tilde B$,
where the $\si_i$ are the elementary symmetric functions. The following result is Hermite's \cite{Hermite1853} and Sylvester's \cite{Sylvester1853} 
version of Sturm's theorem;
the indicated references give a modern account and a far reaching generalization.

\begin{theorem}[{\cite{Procesi78, PS85}}] \index{theorem!Hermite--Sylvester version of Sturm's} 
    The space $\Hyp(d)$ of monic hyperbolic polynomials of degree $d$ can be identified with $\{y \in \R^d : \tilde B(y) \ge 0\}$. 
    The rank of $\tilde B$ equals the number of distinct roots and its signature equals the number of distinct real roots.
\end{theorem}

Since a real symmetric matrix is positive semidefinite if and only if the determinants of the principal minors are 
nonnegative, this description yields explicit inequalities that determine $\Hyp(d)$.

Recall that $\Hyp_T(d) = \Hyp(d) \cap \{y_1 =0\}$\index{HypTd@$\Hyp_T(d)$} is the space of monic hyperbolic polynomials of degree $d$ in Tschirnhausen form. 
Then $\Hyp_T(1)$ is the origin, $\Hyp_T(2)$ is $(-\infty,0]$, $\Hyp_T(3)$ is the closure of the inside of a cusp,
and $\Hyp_T(4)$ is the closure of the inside of the swallowtail.

In \cite{Meguerditchian92} Meguerditchian 
studies the ``escape'' from the set $\Hyp(d)$: roughly speaking, for a polynomial $P$ in $\Hyp(d)$ an integer $s_P$, which depends only on the 
multiplicity vector of $P$, is introduced 
such that one can decide whether $P+Q$ belongs to $\Hyp(d)$ or not 
in terms of how the degree of $Q$ relates to $s_P$. 

By Nuij \cite{Nuij68}, for $P \in \Hyp(d)$ also $P+s P' \in \Hyp(d)$, for all $s \in \R$.
A generalization of this result is due to Kurdyka and P\u aunescu \cite{Kurdyka:2017aa}.

\def\cprime{$'$}
\providecommand{\bysame}{\leavevmode\hbox to3em{\hrulefill}\thinspace}
\providecommand{\MR}{\relax\ifhmode\unskip\space\fi MR }
\providecommand{\MRhref}[2]{%
  \href{http://www.ams.org/mathscinet-getitem?mr=#1}{#2}
}
\providecommand{\href}[2]{#2}

\end{document}